\documentclass[11pt]{article}
\usepackage{amsthm,amsmath,mathrsfs,amssymb,mathtools,mathrsfs}
\usepackage{xcolor}
\usepackage{enumitem}
\usepackage{hyperref}
\usepackage[english]{babel}
\usepackage{graphicx,tikz}
\usepackage[hypcap]{caption}

\theoremstyle{plain}
\newtheorem{theorem}{Theorem}[section]
\newtheorem{lemma}[theorem]{Lemma}

\newtheorem{corollary}[theorem]{Corollary}

\theoremstyle{definition}

\theoremstyle{remark}
\newtheorem{remark}[theorem]{Remark}
\numberwithin{equation}{section}

\setlist{nosep}

\newcommand*\samethanks[1][\value{footnote}]{\footnotemark[#1]}
\newcommand{\R}{\mathbb{R}}
\newcommand{\X}{\mathcal{X}}
\newcommand{\Y}{\mathcal{Y}}
\newcommand{\EX}{\mathcal{E}_{\mathcal{X}}}
\newcommand{\EY}{\mathcal{E}_{\mathcal{Y}}}
\newcommand{\mx}{M_1(\mathcal{X})}
\newcommand{\mnx}{M_1^N(\mathcal{X})}
\newcommand{\my}{M(\mathcal{Y})}
\newcommand{\DX}{D([0,T],\mx)}
\newcommand{\DY}{D_{\uparrow}([0,T],\my)}
\newcommand{\Prob}{{P}}
\newcommand{\E}{E}
\newcommand{\F}{\mathcal{F}}
\newcommand{\Ft}{\mathcal{F}_t}
\newcommand{\dcy}{D\mathcal{C}_{\Y}}
\newcommand{\dcx}{D\mathcal{C}_{\X}}

\title{Large deviations of mean-field interacting particle systems in a fast varying environment}
\author{Sarath Yasodharan\thanks{Supported by the Indo-French Centre for Applied Mathematics.} \thanks{Supported by a fellowship grant from the Centre for Networked Intelligence (a Cisco CSR initiative) of the Indian Institute of Science, Bangalore.} and Rajesh Sundaresan\samethanks[1]
\\ Indian Institute of Science}

\begin{document}
\maketitle 

\begin{abstract}
This paper studies large deviations of a ``fully coupled" finite state mean-field interacting particle system in a fast varying environment. The empirical measure of the particles evolves in the slow time scale and the random environment evolves in the fast time scale. Our main result is the path-space large deviation principle for the joint law of the empirical measure process of the particles and the occupation measure process of the fast environment. This extends previous results known for two time scale diffusions to two time scale mean-field models with jumps. Our proof is based on the method of stochastic exponentials. We characterise the rate function by studying a certain variational problem associated with an exponential martingale.  

\vspace{10pt}
\end{abstract}

\section{Introduction}
\label{section:introduction}
Let $\X, \Y$ be finite sets and $(\X, \EX)$ and $(\Y, \EY)$ be directed graphs on $\X$ and $\Y$ respectively. Let $\mx$ denote the space of probability measures on $\X$. For each $N \geq 1$, we consider Markov processes with infinitesimal generators acting on functions $f$ on $\mnx \times \Y$ of the form
\begin{align*}
\sum_{(x,x^\prime) \in \EX} & N \xi(x) \lambda_{x,x^\prime}(\xi, y) \left[ f\left(\xi+\frac{\delta_{x^\prime}}{N} -\frac{\delta_x}{N}, y \right)- f(\xi,y) \right] + N\sum_{y^\prime: (y,y^\prime) \in \EY} (f(\xi,y^\prime)-f(\xi,y))\gamma_{y,y^\prime}(\xi),
\end{align*}
$\xi \in \mnx$ and $y \in \Y$; here $\mnx \subset \mx$ denotes the set of probability measures on $\X$ that can arise as empirical measures of $N$-particle configurations on $\X^N$, $\lambda_{x,x^\prime}(\cdot, y) : \mx \to \R_+$, $(x,x^\prime) \in \EX$ and $y \in \Y$, and $\gamma_{y,y^\prime} : \mx \to \R_+$, $(y,y^\prime) \in \EY$, are given functions. Such processes arise in the context of weakly interacting Markovian mean-field particle systems in a fast varying environment where the empirical measure of the particle system evolves in the slow time scale and the environment process evolves in the fast time scale. An important feature of such processes is that they are ``fully coupled", i.e., the evolution of the empirical measure depends on the state of the environment, and the environment itself changes its state depending on the empirical measure of the particle system. This paper establishes a process-level large deviation principle (LDP) for the joint law of the empirical measure process and the occupation measure of the fast environment for such fully coupled two time scale mean-field models (see Section~\ref{section:system-model} for the precise mathematical model and Theorem~\ref{thm:main-finite-duration} for the statement of the main result).

Our study of the LDP for such a two time scale mean-field model is motivated by the metastability phenomenon in networked systems. Many networked systems that arise in practice can be modelled using a two time scale mean-field model; see Appendix~\ref{appendix:examples} for details of a retrial queueing system with $N$ orbit queues, and a wireless local area network with local interactions. In such networks, there could be multiple seemingly ``stable points of operation", or metastable points. Some of these may be desirable but some others undesirable in terms of some performance metrics. One is often interested in understanding the following metastable phenomena: (i) the mean time spent by the network near an operating point, (ii) the mean time required for transiting from one stable operating point to another, (iii) the mean time for the system to be sufficiently close to stationarity, etc. The process level large deviations result established in this paper helps to answer such questions on the large time behaviour of these systems.

The above two time scale mean-field model is an example of a stochastic process with time scale separation where a certain component of the process evolves in the slow time scale (i.e. $O(1)$-change in a given $O(1)$ time duration) and another component evolves in the fast time scale (i.e. $O(N)$-change in a given $O(1)$ time duration). Such processes that evolve on multiple time scales have been well studied in the past, and it is known that, under mild conditions, they exhibit the ``averaging principle": when the time scale separation $N$ becomes large, the slow component tracks the solution to a certain dynamical system whose driving function is  ``averaged" over the stationary behaviour of the fast component. In his seminal work, Khasminskii~\cite{khasminskii-68} first proved the averaging principle for two time scale diffusions. Freidlin and Wentzell~\cite[Chapter~7, Section~9]{freidlin-wentzell} studied the averaging phenomenon in a fully coupled system of diffusions where both the drift and the diffusion coefficients of the slow component depend on the fast component and vice-versa. Their proof is based on discretisation arguments. The averaging phenomenon has also been studied in the context of jump processes with applications to performance analysis of various computer communication systems and queueing networks -- Castiel et al.~\cite{castiel-etal-19} studied a carrier sense multiple access algorithm in the context of wireless networks, Bordenave et al.~\cite{bordenave-etal-12} studied  performance analysis of wireless local area networks, Hunt and Kurtz~\cite{hunt-kurtz-94} studied scaling limits of loss networks, Hunt and Laws~\cite{hunt-laws-97} studied analysis of trunk reservation policy in the context of loss networks; also see Kelly~\cite{kelly-91} and the references therein for other works on loss networks in the two time scale framework. While the above works on jump processes study the averaging principle in the large-$N$ limit, this paper focuses on process-level large deviations from the large-$N$ limit.

Various authors have studied process level large deviations of diffusion processes evolving on multiple time scales under various assumptions -- see Freidlin~\cite{freidlin-wentzell}, Veretennikov~\cite{veretennikov-99, veretennikov-00}, Liptser~\cite{liptser-96}, Puhalskii~\cite{puhalskii-16} and the references therein. Liptser~\cite{liptser-96} established the large deviation principle for the joint law of the slow process and the occupation measure of the fast process for one-dimensional diffusions when the fast process does not depend on the slow variable. More recently, Puhalskii~\cite{puhalskii-16} extended this for multidimensional diffusions when the slow and fast processes are fully coupled. His approach is based on the method of stochastic exponentials for large deviations~\cite{puhalskii-94}, where one identifies a suitable exponential martingale associated with the process and characterises the rate function in terms of this exponential martingale. In identifying the rate function, the main ingredient in the proof is to study a certain variational problem and show certain continuity property of its solution. 

In this paper, our proof of the process-level large deviation result is based on the method of stochastic exponentials, see Puhalskii~\cite{puhalskii-94,puhalskii-16}, but the main difficulty lies in extending the approach of Puhalskii~\cite{puhalskii-16} to our two time scale mean-field model with jumps. In particular, our setting requires us to study certain variational problems in an Orlicz space, instead of the usual $L^2$ space in the context of diffusions, to characterise the rate function; see Theorem~\ref{thm:varproblem} and Theorem~\ref{thm:muhat-bdd-away-from-0}. While Puhalskii~\cite{puhalskii-16} uses tools from the theory of elliptic partial differential equations for the characterisation of the rate function, we use tools from convex analysis and parametric continuity of optimisation problems. Also, our mean-field setting makes the solutions to these variational problems blow up near the boundary of the state space, and one of the main novelties of our work is the methodology to obtain a characterisation of the rate function in such cases via suitable approximations -- see Section~\ref{section:uniqueness-rate-function-whole-space}. 

Other works in the two time scale regime include Budhiraja et al.~\cite{budhiraja-etal-18} who studied the case where the slow process is a diffusion and the  fast process is a Markov chain on a finite set; their proof is based on the weak convergence approach to large deviations where one establishes the LDP by studying certain controlled versions of the processes. Kumar and Popovic~\cite{kumar-popovic-17} established the LDP for two time scale jump-diffusions under some general conditions via convergence of nonlinear semigroups, but their approach requires verification of the comparison principle for a certain nonlinear operator. While this is a possible alternative approach for the mean-field problem under consideration, we have used the more probabilistic stochastic exponentials approach.

Let us also mention some works on large deviations of mean-field models that do not involve the fast environment. Dawson and G\"artner~\cite{dawson-gartner-87} established process-level large deviations of interacting diffusions of mean-field type where each particle evolves as a diffusion process with coefficients that depend on the other particles via the empirical measure of the states of all the particles. L\'eonard~\cite{leonard-95-1,leonard-95} extended this to the case of jump processes. Our work can be viewed as an extension of L\'eonard~\cite{leonard-95} to the case of finite state mean-field interacting particle systems with a fully coupled fast varying environment. In the stationary regime, Borkar and Sundaresan~\cite{borkar-sundaresan-12} studied large deviations of the stationary measure of finite state mean-field interacting particle systems using tools from Freidlin and Wentzell~\cite[Chapter~6]{freidlin-wentzell}, and the authors~\cite{mypaper-1} studied large time behaviour, metastability and convergence to stationarity in such systems using tools from Hwang and Sheu~\cite{hwang-sheu-90}. Our results in this paper, along with the results in~\cite{mypaper-1}, can be used to study the large time behaviour and metastability of two time scale mean-field models; see Section~\ref{section:large-time-behaviour}.

The rest of the paper is organised as follows. We start with a formal description of our fully coupled two time scale mean-field model and state our main result and its implications in Section~\ref{section:system-model-main-results}. The proof of the main result is carried out in Sections~\ref{section:exponential-tightness}--\ref{section:complete-proof}. Section~\ref{section:exponential-tightness} establishes exponential tightness of the joint law of the empirical measure process and the occupation  measure process of the fast environment. In Section~\ref{section:itildemartingale}, we define a certain exponential martingale and show a necessary condition that holds for every subsequential rate function. In Section~\ref{section:variational-problem}, we define our candidate rate function using the above exponential martingale and study its relevant properties. In Section~\ref{section:uniqueness-rate-function}, we obtain a characterisation of subsequential rate functions for sufficiently regular elements in the space and Section~\ref{section:uniqueness-rate-function-whole-space} extends this to the whole space using certain approximation arguments. Finally we complete the proof of the main result in Section~\ref{section:complete-proof}.

\section{System model and main result}
\label{section:system-model-main-results}
\subsection{Notation}
We summarise the frequently used notation in the paper. Let $ \langle \cdot, \cdot \rangle$ denote inner product and $\| \cdot \| $ denote the norm on Euclidean spaces. Given a complete separable metric space $S$, let $B(S)$ denote the space of bounded Borel-measurable functions on $S$ equipped with the uniform topology. Let $M(S)$ denote the space of finite measures on $S$ equipped with the topology of weak convergence. Let $M_1(S)$ denote the space of probability measures on $S$ equipped with the L\'evy-Prohorov metric (which generates the topology of weak convergence). (If $S$ is a finite set, then $M_1(S)$ can be viewed as an $(|S|-1)$-dimensional subset of the Euclidean space $\mathbb{R}^{|S|}$; in this case, for $\nu \in M_1(S)$, we shall denote the density of $\nu$ with respect to the counting measure on $S$ by $\nu$). Given $N \in \mathbb{N}$, $M_1^N(S) \subset M_1(S)$ denotes the set of probability measures that can arise as empirical measures of $N$ independent $S$-valued random variables. Given $T > 0$, let $D([0,T], S)$  (resp.~$D(\mathbb{R}_+, S)$) denote the space of c\`adl\`ag functions on $[0,T]$ (resp.~$\mathbb{R}_+$) equipped with the Skorohod-$J_1$ topology (see, for example, Ethier and Kurtz~\cite[Chapter~3]{ethier-kurtz}). Similarly, given a finite set $\Y$, $ \DY \subset D([0,T], \my) $ denotes the space of c\`adl\`ag functions $\theta$ on $[0,T]$ such that for each $0 \leq s \leq t \leq T$, $\theta_t - \theta_s$ is an element of $\my$ and $\theta_t(\Y) =t$. This equipped with its subspace topology is a complete and separable metric space, and is closed in $D([0,T],\my)$. If $X$ is an element of  $D([0,T], S)$, $D([0,\infty), S)$ or $D_{\uparrow}([0,T], \my)$, let $X_t$ and $X(t)$ denote the coordinate projection of $X$ at time $t$.

Denote the moment generating function of the centred unit rate Poisson law by $\tau(u) \coloneqq e^u - u -1, u \in \mathbb{R}$, and its convex dual by
\begin{align*}
\tau^*(u) \coloneqq \left\{
\begin{array}{lll}
+\infty & \text{ if } u < -1 \\
1 & \text{ if } u = -1 \\
(u+1) \log (u+1) - u& \text{ if } u > -1.
\end{array}
\right.
\end{align*}
Given a complete separable metric space $S$ and a finite measure $\vartheta$ on $S$, let $L^\tau(S,\vartheta) $ and $L^{\tau^*}(S, \vartheta)$ denote the Orlicz spaces corresponding to the functions $\tau$ and $\tau^*$, respectively (see, for example, Rao and Ren~\cite[Chapter~3]{rao-ren} for an introduction to Orlicz spaces). The Orlicz norms on these spaces are denoted by $\|\cdot\|_{L^\tau(S,\vartheta)}$ and $\|\cdot\|_{L^{\tau^*}(S,\vartheta)}$, respectively. Given a directed and connected graph $(V,E)$ and $\Delta = (u,v) \in E$, let $u + \Delta$ denote $v$. Given a function $f$ on $[0,T] \times S \times V$, let $Df$ denote the function on $[0,T] \times S \times V  \times E$ defined by $f(t, s, u, \Delta) = f(t,s,v) - f(t,s,u)$ where $\Delta=(u, v) \in E$. Given a subset $W$ of a Euclidean space and $T >0$, let $C^{1,1}([0,T] \times W \times S)$ (resp.~$C^\infty([0,T] \times W \times S)$) denote the space of functions on $f(t, u, s)$, $(t, u, s) \in [0,T] \times W \times S$, that is continuously differentiable (resp.~infinitely differentiable) in both $t$ and $u$. For any function $X$ on $[0,T] \times S$, let $X_t(s)$ and $X(t, s)$ denote the evaluation of $X$ at $(t, s) \in [0,T] \times S$.

We finally recall the definition of a large deviation principle. Let $S$ be a metric space. We say that a sequence $\{X_N\}_{N \geq 1}$ of $S$-valued random variables defined on a probability space $(\Omega, \mathcal{F}, P)$ satisfies the large deviation principle (LDP) with rate function $I : S \to [0, +\infty]$ if 
\begin{itemize}
\item the lower level sets of $I$ are compact, i.e., for each $M > 0$, $\{x \in S : I(x) \leq M\}$ is a compact subset of $S$;
\item for each open set $G \subset S$,
\begin{align*}
\liminf_{N \to \infty} \frac{1}{N} \log P(X_N \in G) \geq -\inf_{x \in G} I(x);
\end{align*}
\item for each closed set $F \subset S$,
\begin{align*}
\limsup_{N \to \infty} \frac{1}{N} \log P(X_N \in F) \leq -\inf_{x \in F} I(x).
\end{align*}
\end{itemize}
We say that $I : S \to [0, +\infty]$ is a subsequential rate function for the family $\{X_N\}_{N \geq 1}$ if there exists a subsequence $\{N_k\}_{k \geq 1}$ of $\mathbb{N}$ such that the sequence $\{X_{N_k}\}_{k \geq 1}$ satisfies the large deviation principle with rate function $I$.
\subsection{System model}
\label{section:system-model}
We describe our model of the mean-field interacting particle system in a fast environment. Let there be $N$ particles and an environment. There is a state associated with each particle as well as the environment at all times; the particle states come from a finite set $\X$ and the environment state comes from a finite set $\Y$. The state of the $n$th particle at time $t$ is denoted by $X_n^N(t) \in \X$, and the state of the environment at time $t$ is denoted by $Y_N(t) \in \Y$. To describe the evolution of the states of the particles, we consider a directed graph $(\X , \EX)$ on the vertex set $\X$ with the interpretation that whenever $(x,x^\prime) \in \EX$, a particle at state $x$ can transit to state $x^\prime$. Similarly, to describe the evolution of the environment, we consider a directed graph $(\Y, \EY)$; $(y,y^\prime) \in \EY$ implies that the environment can transit from state $y$ to state $y^\prime$. 

To describe the particle transitions, we define, for each $y \in \Y$ and $(x,x^\prime) \in \EX$, a function $\lambda_{x,x^\prime}(\cdot,y) : M_1(\X) \to \mathbb{R}_+$, and for each $y \in \Y$, we consider the generator $Q_{N,y}$ acting on functions on $\X^N$ by
\begin{align*}
Q_{N,y} f(\mathbf{x}^N) = \sum_{n=1}^N \sum_{x_n^\prime: (x_n,x_n^\prime) \in \EX} \lambda_{x_n, x_n^\prime}(\overline{\mathbf{x}^N}, y) (f(\mathbf{x}^N_{n,x_n,x_n^\prime}) - f(\mathbf{x}^N)),
\end{align*}
where $\overline{\mathbf{x}^N} \coloneqq \frac{1}{N}\sum_{n=1}^N \delta_{x_n}$ denotes the empirical measure associated with the configuration $\mathbf{x}^N$, and $\mathbf{x}^N_{n, x_n, x_n^\prime}$ denotes the resultant configuration of particles when the $n$th particles changes its state from $x_n$ to $x_n^\prime$ in $\mathbf{x}^N$.  To describe the transitions of the environment, for each $(y, y^\prime) \in \Y$, we define a function $\gamma_{y,y^\prime}(\cdot): M_1(\X) \to \mathbb{R}_+ $, and for each $\xi \in M_1(\X)$, we consider the generator $L_\xi$ acting on functions on $\Y$ by
\begin{align*}
L_\xi g(y) = \sum_{y^\prime: (y,y^\prime) \in \EY} (g(y^\prime)-g(y))\gamma_{y,y^\prime}(\xi).
\end{align*}
Finally,  we consider the generator $\Psi_N$ acting on functions $f$ on $\X^N \times \Y$ by
\begin{align*}
\Psi_N f(\mathbf{x}^N, y) = Q_{N, y}  f(\cdot, y)(\mathbf{x}^N) +N  L_{\overline{\mathbf{x}^N}} f(\mathbf{x}^N, \cdot) (y),
\end{align*}
where $Q_{N, y}  f(\cdot, y)(\mathbf{x}^N)$ (resp.~$L_{\overline{\mathbf{x}^N}} f(\mathbf{x}^N, \cdot) (y)$) indicates that the operator $Q_{N,y}$ (resp.~$L_{\overline{\mathbf{x}^N}}$) acts on the first variable (resp.~second variable) of $f$ and the resultant function is evaluated at $\mathbf{x}^N$ (resp.~$y$).

We make the following assumptions on the particle system:
\begin{enumerate}[label=({A\arabic*})]
\item The graph $(\X, \EX)$ is irreducible; \label{assm:a1}
\item For each $y \in \Y$ and $(x,x^\prime) \in \EX$, the function $\lambda_{x,x^\prime}(\cdot, y)$ is Lipschitz continuous on $M_1(\mathcal{X})$ and $\inf_{\xi \in \mx} \lambda_{x,x^\prime}(\xi, y) > 0$; \label{assm:a2}
\end{enumerate}
and the following assumptions on the environment:
\begin{enumerate}[label=({B\arabic*})]
\item The graph $(\Y, \EY)$ is irreducible; \label{assm:b1}
\item For each $(y,y^\prime) \in \EY$, the function $\gamma_{y,y^\prime}(\cdot)$ is continuous on $\mx$ and $\inf_{\xi \in \mx} \gamma_{y,y^\prime}(\xi) > 0$. \label{assm:b2}
\end{enumerate}

As a consequence of the assumptions~\ref{assm:a2} and~\ref{assm:b2}, we see that the transition rates of the particles as well as that of the environment are bounded, i.e., 

\begin{align*}
\sup_{\xi \in M_1(\X)}\lambda_{x,x^\prime}(\xi,y) < +\infty \,  \forall \,  (x,x^\prime) \in \EX \text{ and } \forall \, y \in \Y
\end{align*}
and
\begin{align*}
 \sup_{\xi \in M_1(\X)} \gamma_{y,y^\prime}(\xi) < +\infty  \, \forall \,  (y,y^\prime) \in \EY,
\end{align*}
and hence the $D([0,\infty), \X^N \times \Y)$-valued martingale problem for $\Psi_{N}$ is well-posed (see, for example, Ethier and Kurtz~\cite[Section~4.1,~Exercise~15]{ethier-kurtz}). Therefore, given an initial configuration of particles $(X_n^N(0), 1 \leq n \leq N) \in \X^N$ and an initial state of the environment $Y_N(0) \in \Y$, we have a Markov process $\{((X_n^N(t),1 \leq n \leq N), Y_N(t)), t \geq 0 \}$ whose sample paths are elements of $D([0,\infty),\X^N \times \Y)$.

To describe the process $\{((X_n^N(t),1 \leq n \leq N), Y_N(t)), t \geq 0 \}$  in words, consider the mapping 
\begin{align*}
\{((X_n^N(t),1 \leq n \leq N), & Y_N(t)), t \geq 0 \}  \mapsto \ \left\{\frac{1}{N}\sum_{n=1}^N \delta_{X_n^N(t)}, t \geq 0\right\} \\
&  \eqqcolon \{\mu_N(t), t \geq 0\} \in D([0,\infty),M_1^N(\X))
\end{align*}
that takes the process $\{((X_n^N(t),1 \leq n \leq N), Y_N(t)), t \geq 0 \}$ and maps it to the empirical measure process $\{\mu_N(t), t\geq 0\}$. Note that, if the environment were frozen to be $y$, then $\mu_N$ is Markov with infinitesimal generator
\begin{align*}
\Phi_{N,y}f(\xi) = \sum_{(x,x^\prime) \in \EX} N \xi(x) \lambda_{x,x^\prime}(\xi, y) \left[ f\left(\xi+\frac{\delta_{x^\prime}}{N} -\frac{\delta_x}{N} \right)- f(\xi) \right].
\end{align*}
We see that a particle in state $x$ at time $t$ makes a transition to state $x^\prime$ at rate $\lambda_{x,x^\prime}(\mu_N(t), Y_N(t))$ independent of everything else. Similarly, the environment makes a transition from state $y$ to $y^\prime$ at time $t$  at rate $N\gamma_{y,y^\prime}(\mu_N(t))$ independent of everything else. Thus, the evolution of each particle depends on the empirical measure of the states of all the particles and the environment, and the evolution of the environment depends on the empirical measure of the states of all the particles. Note that the factor $N$ in the second term of the generator $\Psi_N$ indicates that the process $Y_N$ makes $O(N)$ many transitions while each particle makes $O(1)$ transitions in a given $O(1)$ duration of time. Therefore, we have a ``fully coupled" system where the particles evolve in a fast varying environment. Also, the empirical measure process $\mu_N$ makes $O(N)$ transitions over a given duration of time, but each of those transitions are of size $O(1/N)$ on the probability simplex $M_1(\X)$. We shall refer to $\mu_N$ as the slow process and $Y_N$ as the fast process.
\begin{remark}
Throughout the paper, we assume that all stochastic processes are defined on a complete filtered probability space $(\Omega, \F, (\Ft)_{t \geq 0}, \Prob)$. We denote integration with respect to $\Prob$ by $\E$.
\end{remark}

Fix $T > 0$. We now describe the typical behaviour of our two time scale mean-field system for large $N$ over the time duration $[0,T]$. Towards this, we define the occupation measure of the fast process $Y_N$ by
\begin{align*}
\theta_N(t) \coloneqq \int_0^t 1_{\{Y_N(s) \in \cdot\}} ds,\,  0 \leq t \leq T.
\end{align*}
Note that $\theta_N \in \DY$, $\theta_{N,t}(\Y) =t$ and we can view $\theta_N$ as a measure on $[0,T] \times \Y$. For a fixed empirical measure of the particles $\xi \in M_1(\X)$, assumptions~\ref{assm:b1} and~\ref{assm:b2} imply that there exists a unique invariant probability measure for the Markov process on $\Y$ with infinitesimal generator $L_\xi$ (we denote this by $\pi_\xi$). Therefore, when the empirical measure at time $t$ is at a fixed state $\mu_t$, since the fast process $Y_N$ makes $O(N)$ transitions, we expect that the occupation measure of $Y_N$ for large $N$ becomes ``close" to $\pi_{\mu_t}$, the unique invariant probability measure associated with $L_{\mu_t}$. Due to this ergodic behaviour of the fast process, we anticipate that a particle in state $x$ at time $t$ moves to state $x^\prime$, where $(x, x^\prime) \in \EX$, at rate $\int_\Y  \lambda_{x,x^\prime}(\mu_t, y) \pi_{\mu_t}(dy)$, i.e., the average of $\lambda_{x,x^\prime}(\mu_t, \cdot)$ over $\pi_{\mu_t}$ (for any $\xi \in \mx$, $(x,x^\prime) \in \EX$ and $m \in M_1(\Y)$, we define $\bar{\lambda}_{x,x^\prime}(\xi, m) \coloneqq \int_\Y  \lambda_{x,x^\prime}(\xi, y) m(dy)$).

More precisely, for large enough $N$, we anticipate the following averaging principle for the empirical measure process $\mu_N$. If we assume that the initial conditions $\mu_N(0) \to \nu$ weakly for some deterministic element $\nu \in \mx$, then we anticipate that $\mu_N$ converges in probability, in $\DX$, to the solution to the McKean-Vlasov ODE
\begin{align}
\dot{\mu}_t = \bar{\Lambda}_{\mu_t,\pi_{\mu_t}}^* \mu_t, \, t \geq 0, \, \mu_0 = \nu,
\label{eqn:mve}
\end{align}
where $\bar{\Lambda}_{\mu_t, \pi_{\mu_t}}$ denotes the $|\X| \times |\X|$ rate matrix of the slow process when the empirical measure is $\mu_t$ and the occupation measure of the fast process is $\pi_{\mu_t}$, i.e., $\bar{\Lambda}_{\mu_t, \pi_{\mu_t}}(x,x^\prime) = \bar{\lambda}_{x,x^\prime}(\mu_t, \pi_{\mu_t})$ when $(x,x^\prime) \in \EX$, $\bar{\Lambda}_{\mu_t, \pi_{\mu_t}}(x,x^\prime)=0$ when $(x,x^\prime) \notin \EX$,  $\bar{\Lambda}_{\mu_t, \pi_{\mu_t}}(x,x) = -\sum_{x^\prime \neq x} \bar{\lambda}_{x,x^\prime}(\mu_t, \pi_{\mu_t})$,  and $\bar{\Lambda}_{\mu_t,\pi_{\mu_t}}^* $ denotes its transpose. Note that the above ODE is well-posed, thanks to the Lipschitz assumption on the transition rates~\ref{assm:a2}. See Bordenave et al.~\cite{bordenave-etal-12} for the study of averaging phenomena of a slightly general two time scale model in which each particle has a fast varying environment associated with it. 
\subsection{Main result}
Our main result is on the large deviations of $\{(\mu_N,\theta_N)\}_{N \geq 1}$, the joint empirical measure process associated with the particle system and the occupation measure process associated with the environment $Y_N$, on $\DX \times \DY$. Our main result is the following theorem.
\begin{theorem}
Assume~\ref{assm:a1},~\ref{assm:a2},~\ref{assm:b1}, \ref{assm:b2}, and fix $T> 0$. Suppose that $\{\mu_N(0)\}_{N \geq 1}$ satisfies the LDP on $\mx$ with rate function $I_0$. Then the sequence $\{(\mu_N(t), \theta_N(t)), 0 \leq t \leq T\}_{N \geq 1}$ satisfies the LDP on $D([0,T], \mx) \times D_{\uparrow}([0,T], \my)$ with rate function
\begin{align*}
I(\mu, \theta) \coloneqq I_0(\mu(0)) +J (\mu, \theta),
\end{align*}
where $J$ is defined by
\begin{align}
J(\mu, \theta) &\coloneqq \int_{[0,T]} \left\{ \sup_{\alpha \in \mathbb{R}^{|\X|}} \bigg( \left\langle \alpha, (\dot{\mu}_t -  \bar{\Lambda}^*_{\mu_t, m_t} \mu_t)  \right\rangle \right. \nonumber \\
& \qquad \qquad \left. - \int_{\X \times \EX} \tau(D\alpha(x,\Delta)) \bar{\lambda}_{x,x+d\Delta}(\mu_t, m_t) \mu_t(dx) \bigg) \right. \nonumber\\ 
& \qquad  + \left.  \sup_{g \in B(\Y)} \int_\Y  \biggl( -L_{\mu_t}g (y) \right.  \nonumber \\
&\qquad \qquad - \left. \int_{\EY}\tau(Dg(y,\Delta)) \gamma_{y,y+d\Delta}(\mu_t) \biggr) m_t(dy) \right\} dt
\label{eqn:rate-fn}
\end{align}
whenever the mapping  $[0,T] \ni t \mapsto \mu_t \in \mx$ is absolutely continuous and $\theta$, when viewed as a measure on $[0,T] \times \Y$, admits the representation $\theta(dt dy) = m_t(dy) dt$ for some $m_t \in M_1(\Y)$ for almost all $t\in [0,T]$, and $J(\mu, \theta) = + \infty$ otherwise.
\label{thm:main-finite-duration}
\end{theorem}
\label{Model and main results}
Note that our rate function consists of two parts -- one corresponding to the empirical measure process $\mu_N$ and the other corresponding to the occupation measure of the fast process $Y_N$. The form of the first part of the rate function in~(\ref{eqn:rate-fn}) corresponding to the empirical measure process $\mu_N$ appears in the literature on large deviations of mean-field models (see L\'eonard~\cite[Theorem~3.3]{leonard-95},~\cite[Theorem~1]{djehiche-kaj-95}). The form of the second part is related to the rate function that appears in the study of occupation measure of Markov processes (see Donsker and Varadhan~\cite[Theorem~1]{donsker-varadhan-75-1}). Here, the canonical form of the rate function is $\int_{[0,T]} \sup_{h > 0}  \int_{\Y} -\frac{L_{\mu_t} h(y)}{h(y)} m_t(dy) dt $  and this form of the second part of our rate function in~(\ref{eqn:rate-fn}) can be obtained by taking supremum over functions of the form $e^g$, $g \in B(\Y)$. We see that the first part of the rate function corresponding to the empirical measure process $\mu_N$ has parameters of the mean-field model ``averaged" by the fast variable. Further the second part corresponding to the occupation measure of the fast process has parameters ``frozen" at the current value of the slow variable. The form of our rate function is similar in spirit to that obtained by Puhalskii~\cite{puhalskii-16} in the case of coupled diffusions.

Note that, when $\mu$ is the solution to the McKean-Vlasov equation~(\ref{eqn:mve}) starting at $\mu(0)$ and $\theta$, when viewed as a measure on $[0,T] \times \Y$, is given by $\theta(dy dt) = \pi_{\mu_t}(dy) dt$ where $\pi_{\mu_t}$ is the unique invariant probability measure associated with the infinitesimal generator $L_{\mu_t}$, it is easy to see that the suprema in~(\ref{eqn:rate-fn}) are attained at the identically $0$ functions $\alpha \equiv 0$ and $g \equiv 0$ and hence $J(\mu, \theta) = 0$. Therefore, we recover the typical behaviour of our fully coupled system -- at each time $t>0$, the empirical measure process $\mu_N$ tracks the solution to the McKean-Vlasov equation $\mu_t$ starting at $\mu(0)$ and the occupation measure of the fast process $\theta_N$ tracks the invariant probability measure of the fast process $Y_N$ when the empirical measure is frozen at $\mu_t$. Our result on the large deviations of the joint empirical measure process and the occupation measure of the fast process $\{(\mu_N, \theta_N)\}$ enables us to estimate the probabilities of two kinds of deviations from the typical behaviour -- one where, for a given $\mu$, the occupation measure of the fast process deviates from its typical behaviour (which at time $t$ is $\pi_{\mu_t}(dy)dt$) and the other where $\mu$ deviates from its typical behaviour (which is the solution to~(\ref{eqn:mve}) starting at $\mu(0)$).

We now provide an outline of the proof of Theorem~\ref{thm:main-finite-duration}. Our proof is broadly built upon the methodology of stochastic exponentials for large deviations by Puhalskii~\cite{puhalskii-94,puhalskii-01,puhalskii-16}, where one shows the large deviation principle by first obtaining an equation for a subsequential rate function in terms of a suitable exponential martingale and then obtaining a characterisation of this subsequential rate function. Towards this, we first show that the sequence $\{(\mu_N, \theta_N)\}_{N \geq 1}$ is exponentially tight in $\DX \times \DY$ (see Theorem~\ref{thm:exp-tightness}); this is shown using standard martingale arguments and Doob's inequality. Exponential tightness of the sequence $\{(\mu_N,\theta_N)\}_{N \geq 1}$ implies that there exists a subsequence $\{N_k\}_{k \geq 1}$ of $\mathbb{N}$ such that the family $\{(\mu_{N_k},\theta_{N_k})\}_{k \geq 1}$ satisfies the LDP (see, for example, Dembo and Zeitouni~\cite[Lemma~4.1.23]{dembo-zeitouni}); let $\tilde{I}$ denote the rate function that governs the LDP for the family $\{(\mu_{N_k},\theta_{N_k})\}_{k \geq 1}$. In Sections~\ref{section:itildemartingale}-\ref{section:uniqueness-rate-function-whole-space}, we obtain a characterisation of $\tilde{I}$ when $\tilde{I}$ is such that, for some $\nu \in \mx$, $\tilde{I}(\mu, \theta) = +\infty$ unless $\mu_0 = \nu$; specifically we show that $\tilde{I}(\mu, \theta)$ is given by the right hand side of~(\ref{eqn:rate-fn}). In some more detail, in Section~\ref{section:itildemartingale}, we define an exponential martingale associated with the Markov process $(\mu_N,Y_N)$ for a class of functions $ \alpha : [0,T]\times \mx \to \mathbb{R}^{|\X|}$ and $g : [0,T] \times \mx \times \Y \to \mathbb{R}$ with certain properties, and we obtain an equation that the rate function $\tilde{I}$ must satisfy in terms of this exponential martingale (see Theorem~\ref{thm:rate-function-martingale}). In Section~\ref{section:variational-problem}, we define our candidate rate function $I^*$ in terms of this exponential martingale as a variational problem over functions $\alpha$ and $g$, and we then show that $I^*$ coincides with the RHS of~(\ref{eqn:rate-fn}), and provide a nonvariational expression for $I^*$ using elements from suitable Orlicz spaces (see~Theorem~\ref{thm:varproblem}).  In Section~\ref{section:uniqueness-rate-function}, using the properties of the solution to the variational problem established in Section~\ref{section:variational-problem} and an extension of the equation of $\tilde{I}$ to a larger class of functions $\alpha$ and $g$, we are able to obtain a characterisation of the rate function $\tilde{I}$ for sufficiently regular elements in $\DX \times \DY$ (see~Theorem~\ref{thm:muhat-bdd-away-from-0}). In Section~\ref{section:uniqueness-rate-function-whole-space}, we extend the above characterisation of $\tilde{I}$ to the whole space $\DX \times \DY$ via certain approximation arguments. We finally complete the proof of Theorem~\ref{thm:main-finite-duration} in Section~\ref{section:complete-proof}, by removing the restriction that, for some $\nu \in \mx$, $\tilde{I}(\mu, \theta) =+\infty$ unless $\mu_0 = \nu$.

Our setting of mean-field interaction with jumps introduces some difficulties in characterising a subsequential rate function. One of them is in obtaining regularity properties of the solution to the variational problem  appearing in the definition of $J(\mu, \theta)$ in~(\ref{eqn:rate-fn}) when $(\mu, \theta)$ possesses some good properties. In the recent work of Puhalskii~\cite{puhalskii-16} on large deviations of fully coupled diffusions, the author uses tools from the theory of elliptic partial differential equations for this purpose whereas we resort to tools from convex analysis (L\'eonard~\cite[Sections~4-6]{leonard-95-1}) and parametric continuity of optimisation problems (Sundaram~\cite[Chapter~9]{sundaram-optimisation}) -- see Theorem~\ref{thm:varproblem} and Theorem~\ref{thm:muhat-bdd-away-from-0}. Also, unlike in the case of Gaussian noise in Puhalskii~\cite{puhalskii-16}, our Poissonian noise prevents us from obtaining an explicit form of the solution to the variational problem appearing in the rate function~(\ref{eqn:rate-fn}). Yet another difficulty is in obtaining a characterisation of $\tilde{I}(\mu, \theta)$ when the path $\mu$ hits the boundary of $\mx$. In such cases, the solution to the variational problem that appears in~(\ref{eqn:rate-fn}) blows up near the boundary and hence the condition on $\tilde{I}$ established in Theorem~\ref{thm:u-extension} cannot be directly used. We demonstrate how to approximate $(\mu, \theta)$ via a sequence of regular elements $\{(\mu^i, \theta^i)\}_{i \geq 1}$ so that the solution to the variational problem in $J(\mu^i, \theta^i)$ is well-behaved. We can then use the conclusion of Theorem~\ref{thm:u-extension} on the above sequence and show that $\tilde{I}(\mu^i, \theta^i) \to \tilde{I}(\mu, \theta)$ as $i \to \infty$; see Theorem~\ref{thm:istar=itilde}.

\subsubsection{Marginal $\mu_N$}The above result on large deviations of the joint law of the empirical measure process of the particles and the occupation measure of the fast process enables us to easily obtain large deviations of the empirical measure process $\mu_N$ by using the contraction principle (see, for example, Dembo and Zeitouni~\cite[Theorem~4.2.1]{dembo-zeitouni}).
\begin{corollary}Assume~\ref{assm:a1},~\ref{assm:a2},~\ref{assm:b1}, \ref{assm:b2}, and fix $T >0$. Suppose that $\{\mu_N(0)\}_{N \geq 1}$ satisfies the LDP in $\mx$ with rate function $I_0$. Then $\{\mu_N\}_{N \geq 1}$ satisfies the LDP in $\DX$ with rate function $J_T$ defined as follows. If $[0,T] \ni t \mapsto \mu_t$ is absolutely continuous, then
\begin{align*}
J_T(\mu) &=  I_0(\mu_0) + \int_{[0,T]} \left\{ \sup_{\alpha \in \mathbb{R}^{|\X|}} \bigg(  \langle \alpha, \dot{\mu}_t\rangle - \sup_{m \in M_1(\Y)} \biggr[ \langle \alpha,  \bar{\Lambda}^*_{\mu_t, m} \mu_t \rangle \right. \nonumber \\
& \qquad \qquad \qquad \left. + \int_{\X \times \EX} \tau(D\alpha(x,\Delta)) \bar{\lambda}_{x,x+d\Delta}(\mu_t, m) \mu_t(dx) \bigg) \right. \nonumber\\ 
& \qquad \qquad \qquad - \left.  \sup_{g \in B(\Y)}  \int_\Y \biggl( -L_{\mu_t}g (y) - \int_{\EY}\tau(Dg(y,\Delta)) \gamma_{y,y+d\Delta}(\mu_t) \biggr) m(dy) \biggr] \right\} dt,
\end{align*}
where $\theta$, when viewed as a measure on $[0,T]\times \Y$, admits the representation $\theta(dy dt) = m_t(dy) dt$ for some $m_t \in M_1(\Y)$ for almost all $t \in [0,T]$, and $J_T(\mu) = +\infty$ otherwise.
\label{cor:ldp-mun}
\end{corollary}
\subsubsection{Large time behaviour}
\label{section:large-time-behaviour}
Using the result on the finite duration LDP for the process $\{\mu_N\}_{N\geq 1}$ in Corollary~\ref{cor:ldp-mun}, we can employ the tools of Freidlin and Wentzell~\cite[Chapter~6]{freidlin-wentzell} and Hwang and Sheu~\cite{hwang-sheu-90} to study the large time behaviour of the process $\mu_N$.  The programme to understand the large time behaviour is carried out in~\cite[Section~3]{mypaper-1}. The two crucial properties needed to establish large time behaviour of $\mu_N$ are: (i) the continuity of the Freidlin-Wentzell quasipotential (see~\cite[Section 3]{mypaper-1} for its  definition) and (ii) uniform large deviations of $\mu_N$, uniformly with respect to the initial condition $\mu_N(0)$ lying in a given closed set. One can show that the Freidlin-Wentzell quasipotential is continuous on $\mx \times \mx$  by constructing constant velocity trajectories between any two given points in $\mx$ and estimating the corresponding $J_T$ for that path; see Borkar and Sundaresan~\cite[Lemma~3.4]{borkar-sundaresan-12}. Since the space $\mx$ is compact, one can also establish uniform large deviation estimates, see~\cite[Corollary~2.1]{mypaper-1}. Using the above two properties and the fact that $(\mu_N, Y_N)$ is strong Markov, one can establish results on the large time behaviour of $\mu_N$ such as (i) the mean exit time from a neighbourhood of an $\omega$-limit set of~(\ref{eqn:mve}), (ii) the probability of reaching a given $\omega$-limit set starting from another, etc. -- we refer the reader to~\cite[Section~3]{mypaper-1} for such results.

\section{Exponential tightness}
\label{section:exponential-tightness}
In this section, we prove the exponential tightness of the sequence $\{(\mu_N(t), \theta_N(t)), 0 \leq t \leq T\}_{N \geq 1}$ in $\DX \times \DY$. Towards this, we shall use the following results (Theorems~\ref{thm:exp-tight-cond1}-\ref{thm:exp-tight-cond2}). The proof of these results are standard and will be omitted here (see Feng and Kurtz~\cite[Theorem~4.4]{feng-kurtz} and Puhalskii~\cite[Theorem~B]{puhalskii-94}).
\begin{theorem}A sequence $\{X_N\}=\{X_{N,t}, 0 \leq t \leq T\}$ taking values in $D([0,T], S)$ is exponentially tight if and only if
\begin{enumerate}[label=(\roman*)]
\item for each $M >0$, there exists a compact set $K_M \subset S$ such that
\begin{align*}
\limsup_{N \to \infty}\frac{1}{N}\log P(\exists t \in [0,T] \text{ such that } X_{N,t} \notin K_M ) \leq -M,
\end{align*}
\item there exists a family of functions $ F \subset C(S)$  that  is closed under addition and separates points on $S$ such that for each $f \in F$, $\{f(X_N)\}$ is exponentially tight in $D([0,T], \mathbb{R})$.
\end{enumerate}
\label{thm:exp-tight-cond1}
\end{theorem}
See Feng and Kurtz~\cite[Theorem 4.4]{feng-kurtz} for a proof. We also need the following sufficient condition for exponential tightness in $D([0,T],\mathbb{R})$.
\begin{theorem}Let $\{X_N\}$ be a sequence taking values in $D([0,T],\mathbb{R})$. Suppose that
\begin{enumerate}[label=(\roman*)]
\item we have \begin{align*}
 \lim_{M \to \infty }\limsup_{N \to \infty}\frac{1}{N} \log P(\exists t\in [0,T]\text{ such that } |X_{N,t}| >M ) = -\infty,
\end{align*}
\item for each $\varepsilon > 0$,
\begin{align*}
\lim_{\delta \downarrow 0} \limsup_{N \to \infty}\frac{1}{N} \log \sup_{t_1\in [0,T]}P(\sup_{t_2 \in [t_1, t_1+ \delta]} |X_{N, t_2} - X_{N,t_1}| > \varepsilon) = -\infty.
\end{align*} 
\end{enumerate}
Then $\{X_N\}$ is exponentially tight in $D([0,T],\mathbb{R})$.
\label{thm:exp-tight-cond2}
\end{theorem}
See Puhalskii~\cite[Theorem B]{puhalskii-94} for a proof.

We now show the main result of this section, namely exponential tightness of the sequence $\{(\mu_N, \theta_N)\}_{N \geq 1}$.
\begin{theorem} The sequence of random variables $\{(\mu_N(t), \theta_N(t)), t \in [0,T]\}_{N \geq 1}$ is exponentially tight in $D([0,T], M_1(\X)) \times D_\uparrow([0,T], M(\Y))$, i.e., given any $M>0$, there exists a compact set $K_M \subset D([0,T], M_1(\X)) \times D_\uparrow([0,T], M(\Y))$ such that 
\begin{align*}
\limsup_{N \to \infty} \frac{1}{N} \log P\left( \{(\mu_N(t), \theta_N(t)), 0 \leq t \leq T\} \notin K_M \right) \leq -M
\end{align*} 
\label{thm:exp-tightness}
\end{theorem}
\begin{proof}
It suffices to show that  $\mu_N$ and $\theta_N$ are individually exponentially tight in $\DX$ and $\DY$ respectively (see, for example, Feng and Kurtz~\cite[Lemma~3.6]{feng-kurtz}).

Consider $\theta_N$. Note that, for $0 \leq t \leq T$, we have $|\theta_{N,t}(Y)| \leq t $ for any subset $Y \subset \Y$. Therefore, using the compact set $K_M = \{y \in \mathbb{R}^{|\Y|}: 0 \leq y_i \leq t \,  \forall i\} \subset M(\Y)$, condition $(i)$ of Theorem~\ref{thm:exp-tight-cond1} holds. To verify condition $(ii)$, define the collection of functions $F \coloneqq \{f:M(\Y) \to \mathbb{R} : f(\theta) = \langle \alpha, \theta \rangle, \alpha \in \mathbb{R}^{|\Y|}\}$. Clearly, $F$ is closed under addition and separates points on $M(\Y)$. For any $f$ of the form $f(\theta) = \langle \alpha, \theta \rangle$ for some $\alpha \in \mathbb{R}^{|\Y|}$, note that, with $X_{N,t} =f(X_{N,t})$, condition $(i)$ of Theorem~\ref{thm:exp-tight-cond2} holds since $|X_{N,t}| \leq t \max_{i \in \Y}|\alpha_i|$. To verify condition $(ii)$ of Theorem~\ref{thm:exp-tight-cond2}, note that, for any $0 \leq s \leq t \leq T $, we have $|\theta_{N,t}(Y) - \theta_{N,s}(Y)| \leq t-s$ for any $Y \subset \Y$ and hence $|X_{N,t} - X_{N,s}| \leq (t -s) \max_i |\alpha_i|$. Thus, by choosing a sufficiently small $\delta>0$, it is easy to see that condition~$(ii)$ of Theorem~\ref{thm:exp-tight-cond2} holds. This establishes the exponential tightness of $\theta_N$ in $\DY$.

We now show that $\mu_N$ is exponentially tight in  $D([0,T], M_1(\X))$. Since for each $t > 0$, $\mu_{N,t}$ takes values in a compact space, condition~$(i)$ of Theorem~\ref{thm:exp-tight-cond1} holds trivially. Again, to show condition $(ii)$ in Theorem~\ref{thm:exp-tight-cond1}, we shall make use of Theorem~\ref{thm:exp-tight-cond2}. For this, we fix the class of functions $F \coloneqq \{f:M_1(\X) \to \mathbb{R}_+, f(\xi) = \langle \alpha, \xi \rangle, \alpha \in \mathbb{R}^{|\X|}\}$, which is clearly closed under addition and separates points on $\mx$. Fix $f \in F$ such that $f(\xi) = \langle \alpha, \xi \rangle$ for some $\alpha \in \mathbb{R}^{|\X|}$ and let $X_{N, t} = f(\mu_{N, t}) = \langle \alpha, \mu_{N,t} \rangle$. Note that, we have $|X_{N,t}| \leq \max_x |\alpha_x|$ for all $t \geq 0$ and $N \geq 1$, hence condition $(i)$ of Theorem~\ref{thm:exp-tight-cond2} holds. To check condition $(ii)$, note that,  for each $t_1 \geq 0$ and $\beta > 1$,
\begin{align*}
M_t\coloneqq \exp & \left\{ N\left(\beta X_{N,t} - \beta X_{N,t_1} -\beta \int_{t_1}^t \Phi_{Y_{N,s}} f(\mu_{N,s}) ds\right. \right.\\
& \left. \left. - \int_{t_1}^t  \int_{\X \times \EX}\tau(\beta D\alpha(x,\Delta)) \lambda_{x,x+d\Delta}(\mu_{N,s}, Y_{N,s}) \mu_{N,s}(dx) ds\right)\right\}, t \geq t_1,
\end{align*}
is an $\mathcal{F}_t$-martingale (see L\'eonard~\cite[Lemma~3.3]{leonard-95-1}; alternatively, this can be easily checked using the Dol\'eans-Dade exponential  formula, see, for example, Jacod and Shiryaev~\cite[Chapter~I,~Theorem~4.61]{jacod-shiryaev}). Therefore, given $\varepsilon>0$, $\delta > 0$ and  $t_1>0$, we have
\begin{align*}
\Prob & \left( \sup_{t_2  \in [t_1, t_1+ \delta]} (X_{N, t_2} - X_{N,t_1}) > \varepsilon \right) \\
&= \Prob \left(\sup_{t_2 \in [t_1, t_1+ \delta]} \exp\{N \beta (X_{N, t_2} - X_{N,t_1}) \}> \exp\{N\beta \varepsilon\}\right) \\
& =  \Prob \left(\sup_{t_2 \in [t_1, t_1+ \delta]} M_t \times \exp \left\{N\beta \int_{t_1}^t \Phi_{Y_{N,s}}f(\mu_{N,s})ds  \right. \right. \\
& \qquad \left. \left. + N \int_{t_1}^t  \int_{\X \times \EX}\tau(\beta D\alpha(x,d\Delta))  \lambda_{x,x^\prime}(\mu_{N,s}, Y_{N,s}) \mu_{N,s}(dx) ds  \right\}> \exp\{N\beta \varepsilon\} \right) \\
& \leq \Prob \left(\sup_{t_2 \in [t_1, t_1+ \delta]} M_t \exp\{N\delta c_{\alpha,\beta}\} > \exp\{N\beta  \varepsilon\}\right) \\
& \leq \exp\{-N(\beta \varepsilon - \delta c_{\alpha,\beta} )\} 
\end{align*}
where $c_{\alpha,\beta}$ is a constant depending on $\alpha$ and $\beta$; here the first inequality follows from the boundedness of the transition rates which is a consequence of the Lipschitz assumption~\ref{assm:a2}, and the second inequality follows from Doob's martingale inequality and the fact that $EM_t = EM_{t_1} = 1$. Thus, we obtain
\begin{align*}
 \lim_{\delta\downarrow 0}\limsup_{N \to \infty}\frac{1}{N}\log  \sup_{t_1 \in [0,T]} P\left(\sup_{t_2  \in [t_1, t_1+ \delta]} (X_{N, t_2} - X_{N,t_1}) > \varepsilon\right) \leq -\beta \varepsilon,
\end{align*}
and hence, letting $\beta \to \infty$, we have 
\begin{align*}
 \lim_{\delta\downarrow 0}\limsup_{N \to \infty}\frac{1}{N}\log \sup_{t_1 \in [0,T]}P\left(\sup_{t_2  \in [t_1, t_1+ \delta]} (X_{N, t_2} - X_{N,t_1}) > \varepsilon\right) = -\infty.
\end{align*}
We can now replace $\alpha$ with $-\alpha$ and repeat the above arguments to conclude that
\begin{align*}
 \lim_{\delta\downarrow 0}\limsup_{N \to \infty}\frac{1}{N}\log  \sup_{t_1 \in [0,T]} P\left(\sup_{t_2  \in [t_1, t_1+ \delta]} |X_{N, t_2} - X_{N,t_1}| > \varepsilon\right) = -\infty.
\end{align*}
We have thus verified condition $(ii)$ of Theorem~\ref{thm:exp-tight-cond2} and hence it follows that $\{\mu_N\}_{N\geq 1}$ is exponentially tight in $D([0,T], M_1(\X))$. This completes the proof of the theorem.
\end{proof}
\label{Edxp}

\section{An equation for the subsequential rate function}
\label{section:itildemartingale}
Let $\tilde{I}:\DX \times \DY \to [0, +\infty]$ denote a subsequential rate function for the family $\{(\mu_N,\theta_N)\}_{N \geq 1}$, i.e., for some sequence $\{N_k\}_{k \geq 1}$ of $\mathbb{N}$, the family $\{(\mu_{N_k},\theta_{N_k})\}_{k \geq 1}$ satisfies the large deviation principle with rate function $\tilde{I}$. In this section, we obtain a condition that every such subsequential rate function must satisfy.

We start with some definitions. 
Given $g \in C^{1,1}([0,T]\times M_1(\X) \times \Y)$, define
\begin{align}
V^{g}_t(\mu_N, Y_N)  & \coloneqq g_t(\mu_{N}(t),Y_{N}(t)) - g_0(\mu_{N}(0),Y_{N}(0)) - \int_0^t  \frac{\partial g_s}{\partial s}(\mu_N(s),Y_N(s))ds \nonumber \\
& \qquad  - \int_0^t  \sum_{(x,x^\prime) \in \EX} \biggr[ g_s\left(\mu_N(s) + \frac{\delta_{x^\prime}-\delta_x}{N},Y_N(s)\right) \nonumber \\
& \qquad \qquad - g_s(\mu_N(s),Y_N(s)) \biggr]  \times N \mu_{N,s}(x)\lambda_{x,x^\prime}(\mu_N(s), Y_N(s)) ds \nonumber \\
& \qquad - \int_0^t  \sum_{(x,x^\prime) \in \EX} \tau\biggr(\biggr[ g_s\left(\mu_N(s) + \frac{\delta_{x^\prime}-\delta_x}{N},Y_N(s)\right) \nonumber \\
& \qquad \qquad  - g_s(\mu_N(s),Y_N(s)) \biggr] \biggr) \times N \mu_{N,s}(x)\lambda_{x,x^\prime}(\mu_N(s), Y_N(s)) ds \nonumber \\
\label{eqn:def-v}
\end{align}
Let $n \in \mathbb{N}$. Given the time points $0 = t_0 < t_1 < \cdots < t_n <T$, $\alpha = (\alpha_{t_i})_{i=0}^n$ where $\alpha_{t_i}: M_1(\X) \to \mathbb{R}^{|\X|}$ is continuous for each $0 \leq i \leq n$, and  $\mu \in D([0,T], M_1(\X))$, define
\begin{align}
\int_0^t \alpha_s(\mu_s) d\mu_s \coloneqq \sum_{i=1}^n \langle  \alpha_{t \wedge t_{i-1}}(\mu_{t_{i-1}}) , (\mu_{t \wedge t_i} - \mu_{t \wedge t_{i-1}}) \rangle, t\in [0,T];
\label{eqn:integral-alpha-mu}
\end{align}
note that this object is an element of $D([0,T], \mathbb{R})$. Given $x \in \X$ and $\Delta = (x, x^\prime) \in \EX$, define
\begin{align*}
D\alpha_s(\mu_s)(x, \Delta) \coloneqq \alpha_s(\mu_s)(x^\prime) - \alpha_s(\mu_s)(x). 
\end{align*}
Similarly, given $y \in \Y$ and $\Delta = (y, y^\prime) \in \EY$, define
\begin{align*}
Dg_s(\mu_s, y, \Delta) \coloneqq g_s(\mu_s, y^\prime) - g_s(\mu_s, y).
\end{align*}
Finally, given $(\mu, \theta) \in \DX \times \DY$,  time points $0 = t_0 < t_1 < \cdots < t_n <T$, $\alpha = (\alpha_{t_i})_{i=0}^n$  and $g$ that satisfy the above requirements, define
\begin{align}
U^{\alpha, g}_t(\mu, \theta) & \coloneqq  \int_0^t  \alpha_s(\mu_s)d\mu_s - \int_0^t  \bigg\langle \alpha_s(\mu_s) , \int_\Y \Lambda_{\mu_s, y}^* \mu_s m_s(dy) \bigg\rangle ds \nonumber \\
& \qquad - \int_0^t \int_{\X \times \EX \times \Y} \tau(D\alpha_s(\mu_s)(x,\Delta)) \lambda_{x,x+d\Delta}(\mu_s, y)  \mu_s(dx) m_s(dy) ds \nonumber \\
& \qquad - \int_0^t \int_\Y \biggr(L_{\mu_s}g_s (\mu_s,\cdot)(y) \nonumber \\
& \qquad \qquad + \int_{\EY} \tau(Dg_s(\mu_s, y,\Delta)) \gamma_{y,y+d\Delta}(\mu_s)  \biggr) m_s(dy) ds;
\label{eqn:def-u}
\end{align}
here $\theta$, when viewed as a measure on $[0,T] \times \Y$, admits the representation $\theta(dy dt) = m_t(dy) dt$ for some $m_t \in M_1(\Y)$ for almost all $t \in [0,T]$, which follows from the existence of the regular conditional distribution (see, for example, Ethier and Kurtz~\cite[Theorem 8.1, page 502]{ethier-kurtz}).

We prove the following result, a condition that $\tilde{I}$ must satisfy in terms of the functions $U^{\alpha,g}$.
\begin{theorem}
\label{thm:rate-function-martingale}
Let $\tilde{I}: \DX \times \DY \to [0, + \infty]$ denote a rate function and suppose that there is a subsequence $\{(\mu_{N_k}, \theta_{N_k})\}_{k \geq 1}$ of $\{(\mu_N, \theta_N)\}_{N \geq 1}$ that satisfies the LDP with rate function $\tilde{I}$. Then, for each $\alpha$ and $g$ that satisfy the requirements of the definition of $U$ and $V$ in~(\ref{eqn:def-u}) and (\ref{eqn:def-v}) respectively, we have
\begin{align}
\sup_{(\mu,\theta)\in D([0,T],M_1(\X))\times D_{\uparrow}([0,T],M(\Y))} ( U^{\alpha, g}_T(\mu,\theta) - \tilde{I}(\mu, \theta)) = 0.
\label{eqn:itilde-u}
\end{align}
\end{theorem}
\begin{proof}
Note that, since the transition rates are bounded (which is a consequence of the assumptions~\ref{assm:a2} and \ref{assm:b2}), 
\begin{align*}
N \left(\int_0^t \alpha_s(\mu_s) d\mu_{N,s} - \int_0^t \left\langle \alpha_s(\mu_s) , \int_\Y \Lambda^*_{\mu_{N,s},y} \mu_{N,s} \theta_N(dy ds) \right\rangle \right), t \geq 0,
\end{align*}
is an $\mathcal{F}_t$-martingale. Also, by It\^o's formula,
\begin{align*}
g_t(\mu_{N}(t),&Y_{N}(t))  - g_0(\mu_{N}(0),Y_{N}(0)) - \int_0^t  \frac{\partial g_s}{\partial s}(\mu_N(s),Y_N(s))ds \nonumber \\
& \qquad - \int_0^t  \sum_{(x,x^\prime) \in \EX} \biggr[ g_s\left(\mu_N(s) + \frac{\delta_{x^\prime}-\delta_x}{N},Y_N(s)\right) \nonumber \\
& \qquad \qquad  - g_s(\mu_N(s),Y_N(s)) \biggr]  \times N \mu_{N,s}(x)\lambda_{x,x^\prime}(\mu_N(s), Y_N(s)) ds  \nonumber \\
&\qquad - N \int_0^t  L_{\mu_{N}(s)}g_s(\mu_N(s), \cdot) (Y_N(s)) ds, t \geq 0,
\end{align*}
is an $\mathcal{F}_t$-martingale. Therefore, using the Dol\'eans-Dade exponential formula, it follows that 
\begin{align*}
\exp\{N U^{\alpha, g}_t(\mu_N, \theta_N) + V^g_t(\mu_N, Y_N)\}, t \geq 0,
\end{align*}
is an $\mathcal{F}_t$-martingale, and hence
\begin{align*}
E \exp\{N U^{\alpha, g}_T(\mu_N, \theta_N) + V^g_T(\mu_N, Y_N)\} = 1.
\end{align*}
Clearly, $U_T^{\alpha, g}(\cdot, \cdot)$ is continuous on $\DX \times \DY$, and since $g$ is continuously differentiable in the second argument, $V^g_T(\mu_N, Y_N)$ is bounded, and hence $V^{g}_T(\mu_N, Y_N) /N$ goes to $0$ $P$-a.s. Therefore, the result follows from an application of Varadhan's lemma along the subsequence $\{N_k\}_{k \geq 1}$ (see, for example,~\cite[Theorem 4.3.1]{dembo-zeitouni}).
\end{proof}

\section{The variational problem in $J$}
\label{section:variational-problem}
Motivated by the duality relation~(\ref{eqn:itilde-u}), we define our candidate rate function
\begin{align}
I^*(\mu, \theta) \coloneqq \sup_{\alpha, g} U_T^{\alpha, g}(\mu, \theta),
\label{eqn:istar}
\end{align}
where the supremum is taken over all functions $\alpha$ and $g$ that satisfy the conditions in Theorem~\ref{thm:rate-function-martingale}.

In this section, we study the above variational problem and show that, whenever $I^*(\mu, \theta)< +\infty$, $I^*(\mu, \theta)$ coincides with the RHS of~(\ref{eqn:rate-fn}) and that $I^*(\mu, \theta)$ can be expressed in a non-variational form using elements from suitable Orlicz spaces. We begin with a necessary condition on the elements in $\DX \times \DY$ whose $I^*$ is finite.
\begin{lemma}
\label{lemma:istarfinite}
If $I^*(\mu, \theta) < +\infty$, then the mapping $[0,T] \ni t \mapsto \mu_t \in \mx$ is absolutely continuous.
\end{lemma}
\begin{proof}
Take $g \equiv 0$ and $\alpha$ to be a function of only time (and denote this by $\alpha_t$) in the definition of $U_t^{\alpha, g}$ in~(\ref{eqn:def-u}). Then~(\ref{eqn:istar}) becomes
\begin{align*}
I^*(\mu, \theta) & = \sup_{\alpha, g}  U^{\alpha, g}_T(\mu, \theta)\\
& \geq  \int_0^T  \alpha_t d\mu_t -  \int_0^T \langle \alpha_t, \bar{\Lambda}_{\mu_t, m_t}^* \mu_t \rangle dt \nonumber \\
& \qquad - \int_0^T \int_{\X \times \EX} \tau(D\alpha_t(x,\Delta)) \bar{\lambda}_{x,x+d\Delta}(\mu_t, m_t)  \mu_t(dx)  dt. \\
\end{align*}
Therefore,
\begin{align*}
\int_0^T \alpha_t d\mu_t & \leq I^*(\mu, \theta) + \int_0^T\langle \alpha_t, \bar{\Lambda}_{\mu_t, m_t}^* \mu_t \rangle dt  \\
& \qquad + \int_0^T \int_{\X \times \EX} \tau(D\alpha_t(x,\Delta)) \bar{\lambda}_{x,x+d\Delta}(\mu_t, m_t)  \mu_t(dx)  dt.
\end{align*}
Replacing $c \alpha_t$ in place of $\alpha_t$ in the above equation, dividing throughout by $c$ and choosing $c = 1/ \|D\alpha\|_{L^\tau([0,T]\times \X \times \EX,  \bar{\lambda}_{x,x+d\Delta}(\mu_t,m_t)\mu_t(dx)dt)}$ (i.e. the inverse of the norm of the function $\alpha_t(x, x+\Delta)$ in the Orlicz space $L^\tau([0,T]\times \X \times \EX,  \bar{\lambda}_{x,x+d\Delta}(\mu_t,m_t)\mu_t(dx)dt)$), we have
\begin{align*}
\int_0^T \alpha_t d\mu_t & \leq \|D\alpha\|_{L^\tau([0,T]\times \X \times \EX, \bar{\lambda}_{x,x+d\Delta}(\mu_t,m_t)\mu_t(dx)dt)} (I^*(\mu, \theta)+1) \\
& \qquad + \int_0^T\langle \alpha_t, \bar{\Lambda}_{\mu_t, m_t}^* \mu_t \rangle dt.
\end{align*}
Since $\alpha_t$ is arbitrary, from the definition of $\int_0^t \alpha_t d\mu_t$ in~(\ref{eqn:integral-alpha-mu}), it is clear that the mapping $[0,T] \ni t \mapsto \mu_t \in \mx$ is absolutely continuous.
\end{proof}
We also need the following lemma, whose proof can be found in Puhalskii~\cite[Lemma~A.2, page 460]{puhalskii-01}.
\begin{lemma}
\label{lemma:path-optimisation}
Let $\mathcal{V}$ be a complete separable metric space, and let $\mathcal{U}$ be a dense subspace of $\mathcal{V}$. Let $f(t, v)$ be a function defined on $[0,T] \times \mathcal{V}$ that is measurable in $t$ and continuous in $v$. Further, if $f(t, \beta(t))$ is locally integrable with respect to the Lebesgue measure on $[0,T]$ for all measurable functions $\beta : [0,T] \to \mathcal{U}$, then
\begin{align*}
\sup_{\beta(\cdot) } \int_0^T f(t, \beta(t)) dt= \int_0^T \sup_{y \in \mathcal{U}} f(t, y) dt,
\end{align*}
where the supremum in the LHS is taken over all  $\mathcal{U}$-valued measurable functions $\beta(\cdot)$.
\end{lemma}

Let us introduce some notations. Let $\dcx$ (resp.~$\dcy$) denote the space of functions $D\alpha$ (resp.~$Dg$) on $[0,T]  \times \X \times \EX$ (resp.~$[0,T]  \times \Y \times \EY$) such that $\alpha \in C^1([0,T] \times \X)$ (resp.~$g \in C^1([0,T] \times \Y)$). (For economy of notation in the sequel, we shall also view $\mathbb{R}$-valued functions on $[0,T] \times \X$ as $\mathbb{R}^{|\X|}$-valued functions on $[0,T]$.) Given $(\mu, \theta) \in \DX \times \DY$, let $\mathcal{H}_\X(\mu, \theta)$ denote the $L^{\tau^*}([0,T]\times \X \times \EX, \bar{\lambda}_{x,x+d\Delta}(\mu_t,m_t) \mu_t(dx) dt)$-closure of functions of the form $ \{\exp\{D\alpha\}-1, D\alpha \in \dcx\}$ and let $\mathcal{H}_\Y(\mu, \theta)$ denote the $L^{\tau^*}([0,T]\times \Y \times \EY, \gamma_{y,y+d\Delta}(\mu_t) m_t(dy) dt)$-closure of functions of the form $\{\exp\{Dg\}-1, Dg \in \dcy\}$, where $\theta$ admits the representation $\theta(dydt) = m_t(dy) dt$ for some $m_t \in M_1(\Y)$ for almost all $t \in [0,T]$. We now prove the main result of this section.
\begin{theorem}
\label{thm:varproblem}
Suppose that $(\mu, \theta) \in \DX \times \DY$ is such that $I^*(\mu, \theta) < \infty$. Then, we have
\begin{align}
I^*(\mu, \theta) &= \int_{[0,T]} \left\{ \sup_{\alpha \in \mathbb{R}^{|\X|}} \bigg( \left\langle \alpha, (\dot{\mu}_t -  \bar{\Lambda}^*_{\mu_t, m_t} \mu_t)  \right\rangle \right. \nonumber \\
& \qquad \qquad  \left. - \int_{\X \times \EX} \tau(D\alpha(x,\Delta)) \bar{\lambda}_{x,x+d\Delta}(\mu_t, m_t) \mu_t(dx) \bigg) \right. \nonumber\\ 
& \qquad + \left.   \sup_{g \in B(\Y)} \int_\Y \biggl( -L_{\mu_t}g (y) \right. \nonumber  \\
&\qquad \qquad   \left. - \int_{\EY}\tau(Dg(y,\Delta)) \gamma_{y,y+d\Delta}(\mu_t) \biggr) m_t(dy) \right\} dt,
\label{eqn:istar-supinside}
\end{align}
where $m_t \in M_1(\Y)$ is such that $\theta$, when viewed as a measure on $[0,T] \times \Y$, admits the representation $\theta(dy ds) = m_t(dy) ds$ for almost all $t \in [0,T]$. Moreover, there exist functions $h_\X \in \mathcal{H}_\X(\mu,\theta)$ and  $h_\Y \in \mathcal{H}_\Y(\mu,\theta)$ that satisfy
\begin{align}
\int_{[0,T]\times \X \times \EX } & h_\X D\alpha \bar{\lambda}_{x,x+d\Delta}(\mu_t, m_t)\mu_t(dx) dt \nonumber \\
&  =  \int_{[0,T]}  \left\langle \alpha_t, (\dot{\mu}_t -  \bar{\Lambda}^*_{\mu_t, m_t} \mu_t)  \right\rangle dt, \, \forall \alpha	\in B([0,T]\times\X),
\label{eqn:cond-hx}
\end{align}
and
\begin{align}
\int_{[0,T]\times \Y \times \EY } & h_\Y Dg \bar{\gamma}_{y,y+d\Delta}(\mu_t) m_t(dy) dt \nonumber \\
& = - \int_{[0,T]\times \Y \times \EY }  Dg \bar{\gamma}_{y,y+d\Delta}(\mu_t)m_t(dy) dt, \, \forall g \in B([0,T]\times \Y),
\label{eqn:cond-hy}
\end{align}
respectively, $h_\X \in L^{\tau^*}([0,T]\times \X \times \EX, \bar{\lambda}_{x,x+d\Delta}\mu_t(dx) dt)$ and $h_\Y \in L^{\tau^*}([0,T]\times \Y \times \EY, \gamma_{y,y+d\Delta}(\mu_t) m_t(dy) dt)$, and $I^*(\mu, \theta)$ admits the representation
\begin{align}
I^*(\mu, \theta) &= \int_{[0,T]\times \X \times \EX} \tau^*(h_\X) \bar{\lambda}_{x,x+d\Delta}(\mu_t, m_t) \mu_t(dx) dt \nonumber \\
&\qquad + \int_{[0,T]\times \Y \times \EY} \tau^*(h_\Y)  \gamma_{y,y+d\Delta}(\mu_t) m_t(dy) dt.
\label{eqn:ratefn-nonvarform}
\end{align}
Furthermore, if $\inf_{t \in [0,T]} \min_{x \in \X} \mu_t(x) > 0$ and $ \inf_{t \in [0,T]} \min_{y \in \Y}m_t(y) > 0$, the suprema in~(\ref{eqn:istar-supinside}) over $\alpha$ and $g$ are attained by $\hat{\alpha}_t \in \mathbb{R}^{|\X|}$ and $\hat{g}_t \in B(\Y)$ that satisfy
\begin{align}
\dot{\mu}_t(x)& -   (\bar{\Lambda}^*_{\mu_t, m_t} \mu_t)(x)  \nonumber \\
& + \mu_t(x) \sum_{\substack{x^\prime\in \X: \\ (x,x^\prime) \in \EX}} (\exp\{\hat{\alpha}_t(x^\prime) - \hat{\alpha}_t(x)\} - 1) \bar{\lambda}_{x,x^\prime}  (\mu_t, m_t) \nonumber \\
& - \sum_{\substack{x_0 \in \X: \\ (x_0, x) \in \EX}}  \mu_t(x_0)  (\exp\{\hat{\alpha}_t(x) - \hat{\alpha}_t(x_0)\} -1) \bar{\lambda}_{x_0,x}(\mu_t,m_t)= 0,\, \, \forall x \in \X,
\label{eqn:cond-alpha}
\end{align}
and
\begin{align}
m_t(y) & \sum_{\substack{y^\prime \in \Y : \\ (y,y^\prime) \in \EY}} \exp\{\hat{g}_t(y^\prime) - \hat{g}_t(y)\} \gamma_{y,y^\prime} (\mu_t) \nonumber \\
& - \sum_{\substack{y_0 \in \Y:\\ (y_0, y) \in \EY}} m_t(y_0) \exp\{\hat{g}_t(y) - \hat{g}_t( y_0)\}\gamma_{y_0, y}(\mu_t) = 0, \, \, \forall y \in \Y,
\label{eqn:cond-g}
\end{align}
for almost all $t \in [0,T]$, respectively.
\end{theorem}
\begin{proof}
For the first part of the theorem, we shall make use of Lemma~\ref{lemma:path-optimisation}. Note that, by Lemma~\ref{lemma:istarfinite}, we have that the mapping $[0,T] \ni t \mapsto \mu_t \in \mx$ is absolutely continuous and $\theta$ admits the representation $\theta(dy dt) = m_t(dy) dt$ where $m_t \in M_1(\Y)$ for almost all $t \in [0,T]$. Therefore, for each $t\geq 0$, $U_t^{\alpha, g}$ in~(\ref{eqn:def-u}) can be written as
\begin{align*}
U^{\alpha, g}_t(\mu, \theta) & = \int_0^t  \langle \alpha_s(\mu_s) , \dot{\mu}_s \rangle ds -  \int_0^t \langle \alpha_s(\mu_s) ,  \bar{\Lambda}_{\mu_s, m_s}^* \mu_s  \rangle ds  \nonumber \\
& \qquad - \int_0^t \int_{\X \times \EX} \tau(D\alpha_s(\mu_s)(x,\Delta)) \bar{\lambda}_{x,x+d\Delta}(\mu_s, m_s)  \mu_s(dx) ds \nonumber \\
&  \qquad - \int_0^t \int_\Y \biggr(L_{\mu_s}g_s(\mu_s,\cdot)(y)  \\
&  \qquad \qquad + \int_{\EY} \tau(Dg_s(\mu_s,y,\Delta)) \gamma_{y,y+d\Delta}(\mu_s)  \biggr) m_s(dy) ds,
\end{align*}
where $\alpha$ and $g$ be satisfy the requirements in the definition of $U_t^{\alpha, g}$ in~(\ref{eqn:def-u}). Thus,
\begin{align*}
I^*(\mu, \theta) & = \sup_{\alpha} \int_{[0,T]} \biggr(\langle \alpha_t(\mu_t) , \dot{\mu}_t \rangle -  \langle \alpha_t(\mu_t) ,  \bar{\Lambda}_{\mu_t, m_t}^* \mu_t \rangle  \nonumber \\
&  \qquad \qquad -  \int_{\X \times \EX} \tau(D\alpha_t(\mu_t)(x,\Delta)) \bar{\lambda}_{x,x+d\Delta}(\mu_t, m_t)  \mu_t(dx) \nonumber \biggr) dt \\
& \qquad  + \sup_{g} \int_{[0,T]} \int_\Y \biggr(-L_{\mu_t}g_t(\mu_t, \cdot)(y) \\
& \qquad \qquad \qquad - \int_{\EY} \tau(Dg_t(\mu_t,y,\Delta)) \gamma_{y,y+d\Delta}(\mu_t)  \biggr) m_t(dy) dt
\end{align*}
where the supremum is taken over all functions $\alpha$ and $g$ that satisfy the conditions in the definition of $U_t^{\alpha, g}$ in~(\ref{eqn:def-u}). Note that, since $\mu$ is kept fixed, an approximation argument using mollifiers implies that the above supremum over $\alpha$ can be replaced by supremum over $\alpha_s$, where $\alpha_s$ is any $\mathbb{R}^{|\X|}$-valued bounded measurable function on $[0,T]$. Once again, since $\mu$ is fixed, we can replace the supremum over $g \in C^{1,1}([0,T],M_1(\X) \times \Y)$ with the supremum over $g$ where $g$ is any bounded measurable function on $[0,T] \times \Y$. Therefore,
\begin{align*}
I^*(\mu, \theta) & = \sup_{\alpha} \int_{[0,T]} \biggr(\langle \alpha_t(\mu_t) , \dot{\mu}_t \rangle  -  \langle \alpha_t(\mu_t) ,  \bar{\Lambda}_{\mu_t, m_t}^* \mu_t \rangle   \nonumber \\
&  \qquad \qquad -  \int_{\X \times \EX} \tau(D\alpha_t(\mu_t)(x,\Delta)) \bar{\lambda}_{x,x+d\Delta}(\mu_t, m_t)  \mu_t(dx) \nonumber \biggr) dt \\
& \qquad  + \sup_{g} \int_{[0,T]} \int_\Y \biggr(-L_{\mu_t}g_t(\mu_t, \cdot)(y) \\
& \qquad \qquad \qquad - \int_{\EY} \tau(Dg_t(\mu_t,y,\Delta)) \gamma_{y,y+d\Delta}(\mu_t)  \biggr) m_t(dy) dt
\end{align*}
where the supremum is taken over bounded measurable functions $\alpha:[0,T] \to \mathbb{R}^{|\X|}$ and $g:[0,T] \times \Y \to \mathbb{R}$. We can now apply Lemma~\ref{lemma:path-optimisation} to conclude that $I^*(\mu, \theta)$ is given by~(\ref{eqn:istar-supinside}).

We obtain the existence of functions $h_\X \in \mathcal{H}_\X(\mu,\theta)$ and $h_\Y\in \mathcal{H}_\Y(\mu, \theta)$ that satisfy the conditions~(\ref{eqn:cond-hx}) and~(\ref{eqn:cond-hy}) and the non-variational representation of $I^*$ in~(\ref{eqn:ratefn-nonvarform}) by carrying out the convex analytic programme of L\'eonard~\cite[Sections~5-6]{leonard-95-1} to the bounded linear functionals
\begin{align*}
\alpha \mapsto \int_{[0,T]}   \left\langle \alpha, (\dot{\mu}_t -  \bar{\Lambda}^*_{\mu_t, m_t} \mu_t)  \right\rangle dt \nonumber
\end{align*}
and 
\begin{align*}
g \mapsto \int_{[0,T]\times \Y \times \EY}  \big(g(y+\Delta) - g(y)\big) \gamma_{y,y+d\Delta}(\mu_t) m_t(dy) dt
\end{align*}
on  the closure of $\{D\alpha, \alpha \in B([0,T]\times\X)\}$ and $\{Dg, g \in B([0,T]\times \Y)\}$ in the Orlicz spaces $L^{\tau} ([0,T]\times \X \times \EX, \bar{\lambda}_{x,x+d\Delta}(\mu_t,m_t) \mu_t(dx) dt) $ and $L^{\tau}([0,T]\times \Y \times \EY, \gamma_{y,y+d\Delta}(\mu_t) m_t(dy) dt)$ respectively; the proof follows verbatim from L\'eonard~\cite{leonard-95-1} to our case, and we omit the details here.

Finally, to show the existence of supremisers $\hat{\alpha}_t$ and $\hat{g}$ in~(\ref{eqn:istar-supinside}) and the conditions~(\ref{eqn:cond-alpha}) and~(\ref{eqn:cond-g}) in the case when $\inf_{t \in [0,T]}\min_{x \in \X} \mu_t(x) >0$ and $\inf_{t \in [0,T]} \min_{y \in \Y} m_t(y) > 0$, note that, for each $t \in [0,T]$ for which $\dot{\mu}_t$ exists, the mappings
\begin{align}
\alpha_t \mapsto  \left\langle \alpha_t, (\dot{\mu}_t -  \bar{\Lambda}^*_{\mu_t, m_t} \mu_t)  \right\rangle - \int_{\X \times \EX} \tau(D\alpha_t(x,\Delta)) \bar{\lambda}_{x,x+d\Delta}(\mu_t, m_t) \mu_t(dx)
\label{eqn:map-alpha}
\end{align}
and, viewing $g_t$ as an element of $\mathbb{R}^{|\Y|}$,
\begin{align}
g_t \mapsto -\int_{\Y}  \biggr( L_{\mu_t}g_t (y) +\int_{\EY}\tau(Dg_t(y,\Delta)) \gamma_{y,y+d\Delta}(\mu_t) \biggr) m_t(y) 
\label{eqn:map-g}
\end{align}
are concave on $\mathbb{R}^{|\X|}$ and $\mathbb{R}^{|\Y|}$ respectively. Therefore, there is an $\hat{\alpha}_t$ and a $\hat{g}_t$ that attain the suprema in~(\ref{eqn:istar-supinside}); the conditions in~(\ref{eqn:cond-alpha}) and~(\ref{eqn:cond-g}) on $\hat{\alpha}_t$ and $\hat{g}_t$ easily follow by writing down the first order conditions for optimality of the mappings in~(\ref{eqn:map-alpha}) and~(\ref{eqn:map-g}) respectively.
\end{proof}

\section{Characterisation of the subsequential rate function for sufficiently regular elements}
\label{section:uniqueness-rate-function}
Let $\tilde{I}:\DX \times \DY \to [0, + \infty]$ be a subsequential rate function for the family $\{(\mu_N,\theta_N)\}_{N \geq 1}$, i.e., for some sequence $\{N_k\}_{k \geq 1}$ of $\mathbb{N}$, $\{(\mu_{N_k},\theta_{N_k})\}_{k \geq 1}$ satisfies the large deviation principle with rate function $\tilde{I}$. In addition suppose that, for some $\nu \in \mx$, $\tilde{I}(\mu, \theta) = + \infty$ unless $\mu_0 = \nu$. In this section, we characterise $\tilde{I}$ for sufficiently regular elements in $\DX \times \DY$, i.e., we show that $\tilde{I}(\hat{\mu}, \hat{\theta}) = I^*(\hat{\mu}, \hat{\theta})$ for all elements $(\hat{\mu}, \hat{\theta}) \in \DX \times \DY$ that satisfy certain regularity properties, where $I^*$ is given by~(\ref{eqn:istar-supinside}) (see Theorem~\ref{thm:muhat-bdd-away-from-0}).
\subsection{An extension of Theorem~\ref{thm:rate-function-martingale}}
We first extend the conclusion of Theorem~\ref{thm:rate-function-martingale} to a larger class of functions $\alpha$ and $g$.
Let $\Gamma \subset \DX \times \DY$ denote the set of points $(\mu, \theta)$ such that the mapping $[0,T] \ni t \mapsto \mu_t \in \mx$ is absolutely continuous, and $\theta$, when viewed as a measure on $[0,T] \times \Y$ admits the representation $\theta(dy dt) =m_t(dy) dt$ where $m_t \in M_1(\Y)$ for almost all $t \in [0,T]$. In particular, $(\mu, \theta) \in \Gamma$ implies that the mapping $t \mapsto \mu_t$ is differentiable for almost all $t \in [0,T]$. Given bounded measurable functions $\alpha : [0,T]\times \mx \to \mathbb{R}^{|\X|}$ and $g : [0,T] \times \mx \times \Y \to \mathbb{R}$ such that for all $t \in [0,T]$ and $y \in \Y$ both $\alpha(t, \cdot)$ and $g(t, \cdot, y)$ are continuous on $\mx$, we define, with a slight abuse of notation, for $(\mu, \theta) \in \Gamma$ and $t \in [0,T]$, 
\begin{align}
U^{\alpha, g}_t(\mu, \theta) & \coloneqq  \int_{[0,t]}  \biggr\{ \langle \alpha_s(\mu_s), \dot{\mu}_s- \bar{\Lambda}_{\mu_s, m_s}^* \mu_s\rangle  \nonumber \\
& \qquad \qquad - \int_{\X \times \EX } \tau(D\alpha_s(\mu_s)(x,\Delta)) \bar{\lambda}_{x,x+d\Delta}(\mu_s, m_s)  \mu_s(dx)  \nonumber \\
& \qquad -\int_\Y \biggr(L_{\mu_s}g_s(\mu_s, \cdot) (y) \nonumber \\
& \qquad \qquad + \int_{\EY} \tau(Dg_s(\mu_s,y,\Delta)) \gamma_{y,y+d\Delta}(\mu_s)  \biggr) m_s(dy) \biggr\} ds.
\label{eqn:def-u-extension}
\end{align}
Note that the boundedness of $\alpha$ and $g$ in the above definition implies that  $D\alpha \in L^{\tau}([0,T]\times \X \times \EX,  \bar{\lambda}_{x,x+d\Delta}(\mu_t, m_t) \mu_t(dx)dt)$, and $Dg \in L^\tau([0,T]\times \EY \times	\Y,  \gamma_{y,y+d\Delta}(\mu_t) m_t(dy)dt )$.

Let $\tilde{I}:\DX \times \DY \to [0, + \infty]$ be a subsequential rate function for the family $\{(\mu_N,\theta_N)\}_{N \geq 1}$. Note that, by Theorem~\ref{thm:rate-function-martingale} and the definition of $I^*$ in~(\ref{eqn:istar}), we have that $\tilde{I}(\mu,\theta) \geq I^*(\mu, \theta)$ for all $(\mu, \theta) \in \DX \times \DY$. Given $\delta > 0$, define
\begin{align*}
K_\delta = \{(\mu, \theta): \tilde{I}(\mu, \theta) \leq \delta\};
\end{align*}
since $\tilde{I}$ has compact level sets, $K_\delta$ is compact in $\DX \times \DY$. By Lemma~\ref{lemma:istarfinite} and the fact that $\tilde{I} \geq I^*$, we have that $K_\delta \subset \Gamma$. We now prove the following extension to Theorem~\ref{thm:rate-function-martingale}.
\begin{theorem}
\label{thm:u-extension}
Let $\tilde{I} : \DX \times \DY \to [0, + \infty]$ be a subsequential rate function. Let $\alpha: [0,T]\times \mx \to \mathbb{R}^{|\X|}$, $g: [0,T] \times \mx \times \Y \to \mathbb{R}$ be bounded and measurable functions such that both $\alpha$ and $g$  are continuous on $\mx$. Then,
\begin{align*}
\sup_{(\mu, \theta)\in \Gamma} ( U^{\alpha, g}_T(\mu,\theta) - \tilde{I}(\mu, \theta)) = 0.
\end{align*}
Moreover, there exists some $\delta>0$ (depending on $\alpha$ and $g$) such that
\begin{align}
\sup_{(\mu, \theta)\in K_\delta} ( U^{\alpha, g}_T(\mu,\theta) - \tilde{I}(\mu, \theta)) = 0,
\label{eqn:sup-attained}
\end{align}
and the above supremum is attained.
\end{theorem}
\begin{proof}
We first define certain approximations of functions $\alpha$ and $g$ that meet the requirements of Theorem~\ref{thm:rate-function-martingale} and prove certain convergence properties of these approximations. We then use the conclusion of Theorem~\ref{thm:rate-function-martingale} for these approximations and pass to the limit to obtain~(\ref{eqn:sup-attained}). Our proof is inspired by ideas from Puhalskii~\cite[Lemma~7.2 and Theorem~7.1]{puhalskii-01}, with necessary modifications to our mean-field with jumps setting.

Since $\alpha$ is a Carath\'eodory function, using the Scorza-Dragoni theorem, for each $i \geq 1$, there exists a compact set $F_i \subset [0,T]$ and a measurable function $\bar{\alpha}_i:[0,T] \times \mx \to \mathbb{R}^{|\X|}$ such that $\bar{\alpha}_i = \alpha $ on $F_i \times \mx$, $\bar{\alpha}_i$ is continuous on $F_i \times \mx$, and $\text{Leb}([0,T] \setminus F_i) \leq 1/i$ (see, for example, Ekeland and Temam~\cite[page 235]{ekeland-temam}). Since $[0,T] \setminus F_i$ is open in $[0,T]$, we can write it as a countable union of disjoint open intervals, and hence we can extend $\bar{\alpha}_i$ to a continuous function on $[0,T] \times \mx$ by a linear interpolation between the two endpoints of the above open intervals; we again denote this function by $\bar{\alpha}_i$. Put $\alpha_i(t,\mu_t)= \bar{\alpha}_i( \frac{\lfloor tn(i) \rfloor}{n(i)} ,\mu_{\frac{\lfloor tn(i) \rfloor}{n(i)}})$, where $n(i) \to \infty$ as $i \to \infty$. By continuity of $\tau$, boundedness of $\alpha$ and $\alpha_i$, boundedness of transition rates of the particles (which is a consequence of assumption~\ref{assm:a2}), we have that, for each $\delta > 0$,
\begin{align}
\sup_{(\mu, \theta)   \in K_\delta} & \biggr| \int_{[0,T] \times \X \times \EX }  \tau(D\alpha_i(t,\mu_t)(x,\Delta)) \bar{\lambda}_{x,x+d\Delta}(\mu_t, m_t)  \mu_t(dx) dt  -  \nonumber  \\
& \qquad \int_{[0,T] \times \X \times \EX } \tau(D\alpha(t,\mu_t)(x,\Delta)) \bar{\lambda}_{x,x+d\Delta}(\mu_t, m_t)  \mu_t(dx) dt \biggr|  \nonumber \\
& = \sup_{(\mu, \theta)   \in K_\delta}  \biggr| \int_{K_i^c \times \X \times \EX }  \tau(D\alpha_i(t,\mu_t)(x,\Delta)) \bar{\lambda}_{x,x+d\Delta}(\mu_t, m_t)  \mu_t(dx) dt  -  \nonumber  \\
& \qquad \int_{K_i^c \times \X \times \EX } \tau(D\alpha(t,\mu_t)(x,\Delta)) \bar{\lambda}_{x,x+d\Delta}(\mu_t, m_t)  \mu_t(dx) dt \biggr|  \nonumber \\
& \leq \text{Leb}(K_i^c) \times c_{\alpha}  \to 0
\label{eqn:convergence-alpha-1}
\end{align}
as $i \to \infty$, where $c_\alpha > 0$ is a constant depending on $\alpha$. Furthermore, given $\delta > 0$ and $(\mu, \theta) \in K_\delta$, by Lemma~\ref{lemma:istarfinite}, the mapping $[0,T] \ni t \mapsto \mu_t \in \mx$ is absolutely continuous. Hence, noting that $\mu$ is kept fixed, by~(\ref{eqn:cond-hx}) in Theorem~\ref{thm:varproblem}, there exists $h_\X \in \mathcal{H}(\mu, \theta)$ such that
\begin{align*}
 \int_{[0,T]}  & \left\langle \alpha(t,\mu_t), (\dot{\mu}_t -  \bar{\Lambda}^*_{\mu_t, m_t} \mu_t)  \right\rangle dt\\
& = \int_{[0,T]\times \X \times \EX }  h_\X D\alpha \bar{\lambda}_{x,x+d\Delta}(\mu_t, m_t)\mu_t(dx) dt,
\end{align*}
and
\begin{align*}
 \int_{[0,T]} &  \left\langle \alpha_i(t,\mu_t), (\dot{\mu}_t -  \bar{\Lambda}^*_{\mu_t, m_t} \mu_t)  \right\rangle dt \\
& = \int_{[0,T]\times \X \times \EX }  h_\X D\alpha_i \bar{\lambda}_{x,x+d\Delta}(\mu_t, m_t) \mu_t(dx) dt.
\end{align*}
Therefore,
\begin{align*}
\biggr|\int_{[0,T]} \langle \alpha_i(t,\mu_t) & - \alpha(t,\mu_t), \dot{\mu}_t- \bar{\Lambda}_{\mu_t, m_t}^* \mu_t\rangle dt \biggr|  \\
& =  \biggr| \int_{[0,T] \times \X \times \EX} h_\X (D\alpha_i - D\alpha) \bar{\lambda}_{x,x+d\Delta}(\mu_t, m_t) \mu_t(dx) dt \biggr|   \\
& \leq  \int_{[0,T] \times \X \times \EX} |h_\X (D\alpha_i - D\alpha)| \bar{\lambda}_{x,x+d\Delta}(\mu_t, m_t) \mu_t(dx) dt    \\
& \leq 2 \| h_\X\|_{L^{\tau^*}([0,T] \times \X \times \EX, \bar{\lambda}_{x,x+d\Delta}(\mu_t,m_t) \mu_t(dx) dt)}   \\
& \qquad \times  \|D\alpha_i - D\alpha\|_{{L^{\tau}([0,T] \times \X \times \EX, \bar{\lambda}_{x,x+d\Delta}(\mu_t,m_t) \mu_t(dx) dt)}}  \\
& \leq 2 \max\{1, \delta+T\}  \\ \nonumber 
& \qquad  \times \|D\alpha_i - D\alpha\|_{{L^{\tau}([0,T] \times \X \times \EX, \bar{\lambda}_{x,x+d\Delta}(\mu_t,m_t) \mu_t(dx) dt)}},
\end{align*}
where the second inequality follows from H\"older's inequality in Orlicz spaces and the third inequality follows from the non-variational representation of the candidate rate function in $I^*$ in~(\ref{eqn:ratefn-nonvarform}), which gives that $\| h_\X\|_{L^{\tau^*}([0,T] \times \X \times \EX, \bar{\lambda}_{x,x+d\Delta}(\mu_t,m_t) \mu_t(dx) dt)} \leq \max \{1, I^*(\mu, \theta)+T\}$, along with the fact that $(\mu,\theta) \in K_\delta$ and $I^*(\mu,\theta) \leq \tilde{I}(\mu, \theta)$. Hence,
\begin{align}
\sup_{(\mu, \theta) \in K_\delta}\biggr|\int_{[0,T]} \langle \alpha_i(t,\mu_t) & - \alpha(t,\mu_t), \dot{\mu}_t- \bar{\Lambda}_{\mu_t, m_t}^* \mu_t\rangle dt \biggr| \to 0
\label{eqn:convergence-alpha-2}
\end{align}
as $i \to \infty$. Similarly, by standard arguments using mollifiers and the Scorza-Dragoni theorem, we can show that there exist functions $g_i$ on $[0,T] \times \mx \times \Y$ such that $g_i(\cdot, \cdot , y) \in C^{\infty}([0,T] \times \mx)$ for all $y \in \Y$ and $\text{Leb} \{t \in [0,T] :  g_i(t, \cdot, \cdot)  \neq g(t, \cdot, \cdot)\} \leq 1/i$ for each $i \geq 1$. Therefore, using boundedness of the functions $g$, $g_i, i \geq 1$, and boundedness of the transition rates of the fast process (which is a consequence of assumption~\ref{assm:b2}), we see that
\begin{align}
\sup_{(\mu,\theta) \in K_\delta} &  \biggr| \int_{[0,T] \times \Y} \biggr(L_{\mu_t}g_i(t,\mu_t, \cdot) (y) \nonumber  \\
& + \int_{\EY} \tau(Dg_i(t,\mu_t,y,\Delta)) \gamma_{y,y+d\Delta}(\mu_t)  \biggr) m_t(dy) dt \nonumber \\
& - \int_{[0,T] \times \Y} \biggr(L_{\mu_t}g(t,\mu_t, \cdot) (y)  \nonumber \\
& + \int_{\EY} \tau(Dg(t,\mu_t,y,\Delta)) \gamma_{y,y+d\Delta}(\mu_t)  \biggr) m_t(dy)dt \biggr|   \to 0
\label{eqn:convergence-g}
\end{align}
as $i \to \infty$. Since $\alpha_i$ and $g_i$, $i \geq 1$, satisfy the conditions on $\alpha$ and $g$ respectively in the definitions of $U$ in~(\ref{eqn:def-u}) and $V$ in~(\ref{eqn:def-v}), Theorem~\ref{thm:rate-function-martingale} implies that
\begin{align*}
\sup_{(\mu,\theta) \in \DX \times \DY} (U_T^{\alpha_i, g_i}(\mu,\theta) -\tilde{I}(\mu, \theta)) = 0.
\end{align*}
By Lemma~\ref{lemma:istarfinite} and the fact that $\tilde{I} (\mu, \theta) \geq I^*(\mu, \theta)$, we see that $\tilde{I}(\mu,\theta) =+\infty$ whenever $(\mu, \theta) \notin \Gamma$, and hence we immediately get
\begin{align}
\sup_{(\mu,\theta) \in \Gamma} (U_T^{\alpha_i, g_i}(\mu,\theta) -\tilde{I}(\mu, \theta)) = 0.
\label{eqn:temp0}
\end{align}
Let us now show that
\begin{align}
\sup_{(\mu,\theta) \in K_\delta} (U_T^{\alpha_i, g_i}(\mu,\theta) -\tilde{I}(\mu, \theta)) = 0
\label{eqn:u-itilde-suitable-delta}
\end{align}
holds for a suitable $\delta > 0$ and all $i \geq 1$. Note that, using the boundedness of the functions $\alpha$, $g$, $\alpha_i$ and $g_i$, $i \geq 1$, and boundedness of the transition rates (as a consequence of assumptions~\ref{assm:a2} and~\ref{assm:b2}), we have
\begin{align*}
U^{2\alpha_i, 2g_i}_T(\mu, \theta) & =  \int_{[0,T]}  \biggr\{ 2 \langle \alpha_i(t,\mu_t), \dot{\mu}_t- \bar{\Lambda}_{\mu_t, m_t}^* \mu_t\rangle  \nonumber \\
& \qquad \qquad- \int_{\X \times \EX } \tau(2D\alpha_i(t,\mu_t)(x,\Delta)) \bar{\lambda}_{x,x+d\Delta}(\mu_t, m_t)  \mu_t(dx)  \nonumber \\
& \qquad -\int_\Y \biggr(2L_{\mu_t}g_t(\mu_t, \cdot) (y) \nonumber \\
& \qquad \qquad + \int_{\EY} \tau(2Dg_i(t,\mu_t,y,\Delta)) \gamma_{y,y+d\Delta}(\mu_t)  \biggr) m_t(dy) \biggr\} dt \\
& \geq 2 U^{\alpha_i,g_i}(\mu, \theta) - 2T c_{\alpha,g}
\end{align*}
for all $i \geq 1$, where $c_{\alpha,g}  >0$ is  a constant depending on $\alpha$ and $g$. Therefore, for a fixed $M >0$, we have
\begin{align*}
\sup_{(\mu, \theta): U_T^{\alpha_i, g_i}(\mu,\theta)  \geq M} & (U^{\alpha_i,g_i}(\mu, \theta) - \tilde{I}(\mu, \theta)) \\
& \leq \sup_{(\mu, \theta): U_T^{\alpha_i, g_i}(\mu,\theta) \geq M} (2U^{\alpha_i,g_i}(\mu, \theta) - \tilde{I}(\mu, \theta)) - M \\
& \leq \sup_{(\mu, \theta): U_T^{\alpha_i, g_i}(\mu,\theta) \geq M} (U^{2\alpha_i,2g_i}(\mu, \theta) - \tilde{I}(\mu, \theta)) + 2Tc_{\alpha, g}-M \\
& \leq 2Tc_{\alpha,g}-M.
\end{align*}
Therefore the above implies that,
\begin{align*}
\sup_{(\mu, \theta)\in \Gamma} & (U^{\alpha_i, g_i}(\mu, \theta) -\tilde{I}(\mu, \theta)) \\
& \leq \sup_{(\mu, \theta)\in K_\delta} (U^{\alpha_i, g_i}(\mu, \theta) -\tilde{I}(\mu, \theta)) \\
& \qquad \vee  \sup_{(\mu, \theta): U^{\alpha_i,g_i}(\mu,\theta) \geq M} (U^{\alpha_i, g_i}(\mu, \theta) -\tilde{I}(\mu, \theta)) \\
&\qquad  \vee (M-\delta) \\
& \leq \sup_{(\mu, \theta)\in K_\delta} (U^{\alpha_i, g_i}(\mu, \theta) -\tilde{I}(\mu, \theta)) \vee (2Tc_{\alpha,g}-M)  \vee (M -\delta).
\end{align*}
Hence, choosing $M=1+2Tc_{\alpha,g}$ and $\delta= M+1$, the above and~(\ref{eqn:temp0}) imply~(\ref{eqn:u-itilde-suitable-delta}). Letting $i \to \infty$, using convergences~(\ref{eqn:convergence-alpha-1})-(\ref{eqn:convergence-alpha-2}) for the slow process, and~(\ref{eqn:convergence-g}) for the fast process,~(\ref{eqn:u-itilde-suitable-delta}) becomes
\begin{align}
\sup_{(\mu,\theta) \in K_\delta} (U_T^{\alpha, g}(\mu,\theta) -\tilde{I}(\mu, \theta)) = 0.
\label{eqn:temp1}
\end{align}

Since the functions $U^{\alpha_i, g_i}_T$ (defined in~(\ref{eqn:def-u})), $i \geq 1$, are continuous on $\Gamma$ and since for all $\delta^\prime >0$
\begin{align*}
\lim_{i \to \infty} \sup_{(\mu, \theta) \in K_{\delta^\prime}} |U^{\alpha_i, g_i}_T(\mu, \theta) - U^{\alpha, g}_T(\mu, \theta)| \to 0
\end{align*}
as $i \to \infty$, it follows that, for all $\delta^\prime > 0$, $U^{\alpha, g}_T$ (defined in~(\ref{eqn:def-u-extension})) is continuous on $K_{\delta^\prime}$. Hence, using the compactness of the level sets of $\tilde{I}$, we see that the supremum in~(\ref{eqn:temp1}) is attained. This completes the proof of the theorem.
\end{proof}

\subsection{Characterisation of $\tilde{I}$ for regular elements}
We now prove the main result of this section, namely $\tilde{I}(\mu, \theta) = I^*(\mu, \theta)$ for all $(\mu, \theta) \in \DX \times \DY$ that satisfy certain regularity properties.
\begin{theorem}
Let $\nu \in \mx$ and let $\tilde{I} : \DX \times \DY \to [0, + \infty]$ be a subsequential rate function such that $\tilde{I}(\mu, \theta) = +\infty$ unless $\mu_0 = \nu$.  Suppose that $(\hat{\mu},\hat{\theta}) \in \DX \times  \DY$ is such that
	\begin{itemize}
		\item $\inf_{t \in [0,T]} \min_{x \in \X}\hat{\mu}_t(x) > 0$,
		\item the mapping $[0,T] \ni t \mapsto \hat{\mu}_t \in \mx$ is Lipschitz continuous,
		\item $\hat{\theta}$, when viewed as a measure on $[0,T] \times \Y$, admits the representation  $\hat{\theta}(dy dt) = \hat{m}_t(dy) dt$ for some $\hat{m}_t \in M_1(\Y)$ for almost all $t \in [0,T]$, and $\inf_{t \in [0,T]} \min_{y \in \Y} \hat{m}_t(y) >0$.
	\end{itemize}	
Then $\tilde{I} (\hat{\mu}, \hat{\theta}) = I^* (\hat{\mu}, \hat{\theta})$.
\label{thm:muhat-bdd-away-from-0}
\end{theorem}
\begin{proof}
Let $\delta = \inf_{t > 0} \min_{x \in \X} \hat{\mu}_t(x)$. For each $t \in [0,T]$, consider the parametrised optimisation problems
	\begin{align}
	\sup_{\alpha_t \in \mathbb{R}^{|\X|}} \biggr\{ \langle \alpha_t, \dot{\hat{\mu}}_t - \bar{\Lambda}^*_{u, \hat{m}_t}u \rangle - \int_{\X \times \EX} \tau(D\alpha_t(x,\Delta))\bar{\lambda}_{x,x+d\Delta}(u,\hat{m}_t) u(dx) \biggr\},
	\label{eqn:optproblem-alpha}
	\end{align}
 $u \in \mx$ is such that $u(x) \geq \delta/2 $ for all $x \in \X$, and 
	\begin{align}
	\sup_{g_t \in B(\Y)} \biggr\{ -\int_\Y \left(L_ug_t(\cdot)(y) + \int_{\EY} \tau(Dg_t(y,\Delta)) \gamma_{y,y+d\Delta}(u) \right) \hat{m}_t(dy) \biggr\},
	\label{eqn:optproblem-g}
	\end{align}
$u \in \mx$. Note that the mappings
	\begin{align}
	\alpha_t \mapsto \langle \alpha_t, \dot{\hat{\mu}}_t - \bar{\Lambda}^*_{u, \hat{m}_t}u \rangle - \int_{\X \times \EX} \tau(D\alpha_t(x,\Delta))\bar{\lambda}_{x,x+d\Delta}(u,\hat{m}_t) u(dx),
	\end{align}
where $u$ is such that $u(x) \geq \delta/2$ for all $x \in \X$, and since $\inf_{t \in [0,T]} \min_{y \in \Y} \hat{m}_t(y) > 0$, viewing $g _t$ as an element of $\mathbb{R}^{|\Y|}$,
	\begin{align}
	g_t \mapsto -\int_\Y \left(L_ug_t(\cdot)(y) + \int_{\EY} \tau(Dg_t(y,\Delta)) \gamma_{y,y+d\Delta}(u) \right) \hat{m}_t(dy)
	\end{align}
	are concave on $\mathbb{R}^{|\X|}$ and $\mathbb{R}^{|\Y|}$ respectively. Therefore, we see that there exist an $\hat{\alpha}_t(u) \in \mathbb{R}^{|\X|}$ and a $\hat{g}_t(u) \in \mathbb{R}^{|\Y|}$ that solve~(\ref{eqn:optproblem-alpha}) and~(\ref{eqn:optproblem-g}) respectively. Guided by~(\ref{eqn:cond-alpha}) and~(\ref{eqn:cond-g}), $\hat{\alpha}_t(u)$ and $\hat{g}_t(u)$ satisfy the first order optimality conditions
	\begin{align}
	\dot{\hat{\mu}}_t(x)& -   (\bar{\Lambda}^*_{u, \hat{m}_t} u)(x)  \nonumber \\
	& +u(x) \sum_{\substack{x^\prime\in \X: \\ (x,x^\prime) \in \EX}} (\exp\{\hat{\alpha}_t(u)(x^\prime) - \hat{\alpha}_t(u)(x)\} - 1) \bar{\lambda}_{x,x^\prime}  (u, \hat{m}_t) \nonumber \\
	& - \sum_{\substack{x_0 \in \X: \\ (x_0, x) \in \EX}} u(x_0) (\exp\{\hat{\alpha}_t(u)(x) - \hat{\alpha}_t(u)(x_0)\} -1) \bar{\lambda}_{x_0,x}(u,\hat{m}_t)  = 0,\, \, \forall x \in \X,
	\label{eqn:alphahat}
	\end{align}
where $t \in [0,T]$ and $u \in \mx$ is such that $u(x) \geq \delta/2$ for all $x \in \X$, and
	\begin{align}
	\hat{m}_t(y) & \sum_{\substack{y^\prime \in \Y : \\ (y,y^\prime) \in \EY}} \exp\{\hat{g}_t(u,y^\prime) - \hat{g}_t(u,y)\} \gamma_{y,y^\prime} (u) \nonumber \\
	& - \sum_{\substack{y_0 \in \Y:\\ (y_0, y) \in \EY}} \hat{m}_t(y_0) \exp\{\hat{g}_t(u,y) - \hat{g}_t(u, y_0)\}\gamma_{y_0, y}(u) = 0, \, \, \forall y \in \Y,
	\label{eqn:ghat}
	\end{align}
 where $t \in [0,T]$ and $u \in \mx$, respectively.

We now define bounded measurable functions $\hat{\alpha} : [0,T] \times \mx \to \mathbb{R}^{|\X|}$ and $\hat{g}:[0,T] \times \mx \times \Y \to \mathbb{R}$ that are continuous on $\mx$ such that $\hat{\alpha}(u)$ (resp.~$\hat{g}(u)$) solves the optimisation problem in~(\ref{eqn:optproblem-alpha}) (resp.~(\ref{eqn:optproblem-g})). Note that the objective function in~(\ref{eqn:optproblem-g}) is uniquely determined by $\{g(t,y^\prime) - g(t,y), (y,y^\prime) \in \EY\}$, and by assumption~\ref{assm:a1}, the objective function in~(\ref{eqn:optproblem-alpha}) is uniquely determined by $\{\alpha_t(x^\prime) - \alpha_t(x), (x,x^\prime) \in \EX\}$. Since $\inf_{t \in [0,T]} \min_{x \in \X} \hat{\mu}_t(x) > 0$, the mapping $t \mapsto \hat{\mu}_t$ is Lipschitz continuous and the transition rates of the slow process are bounded (which is a consequence of assumption~\ref{assm:a2}), we see that we can restrict the supremum over $\alpha_t$ in~(\ref{eqn:optproblem-alpha}) to a single compact and convex subset of $\mathbb{R}^{|\X|}$, regardless of $t \in [0,T]$ and $u \in \mx$ with $u(x) \geq \delta/2$ for all $x \in \X$. Similarly, since $\inf_{t \in [0,T]} \min_{y \in \Y} \hat{m}_t(y) > 0$ and the transition rates of the fast process are bounded (which follows from assumption~\ref{assm:b2}), we see that we can restrict the supremum in~(\ref{eqn:optproblem-g}) to a single compact and convex subset of $\mathbb{R}^{|\Y|}$, regardless of $t \in [0,T]$ and $u \in \mx$. Also, note that the mappings~(\ref{eqn:optproblem-alpha}) and~(\ref{eqn:optproblem-g}), when viewed as
	\begin{align*}
	\{\alpha_t&(x^\prime)  - \alpha_t(x), (x,x^\prime) \in \EX \} \\
	&\mapsto \langle \alpha_t, \dot{\hat{\mu}}_t - \bar{\Lambda}^*_{u, \hat{m}_t}u \rangle - \int_{\X \times \EX} \tau(D\alpha_t(x,\Delta))\bar{\lambda}_{x,x+d\Delta}(u,\hat{m}_t) u(dx)
	\end{align*}
	and,
	\begin{align*}
	\{g_t&(y^\prime)-g_t(y), (y,y^\prime) \in \EY\} \\
	& \mapsto -\int_\Y \left(L_ug_t(\cdot)(y) + \int_{\EY} \tau(Dg_t(y,\Delta)) \gamma_{y,y+d\Delta}(u) \right) \hat{m}_t(dy)
	\end{align*}
are strictly concave on $\mathbb{R}^{|\EX|}$ and $\mathbb{R}^{|\EY|}$ respectively; hence there exists a unique $\{\hat{\alpha}_t(u)(x^\prime)  -\hat{ \alpha}_t(u)(x), (x,x^\prime) \in \EX \}$ and a unique $\{\hat{g}_t(u,y^\prime)-\hat{g}_t(u,y), (y,y^\prime) \in \EY\}$ that solve~(\ref{eqn:optproblem-alpha}) and~(\ref{eqn:optproblem-g}) respectively. Fixing $\hat{\alpha}_t(u)(x_0) = 0$ for some $x_0 \in \X$, where $t \in [0,T]$ and $u \in \mx$ with $u(x) \geq \delta/2$ for all $x \in \X$, fixing $g_t(u,y) = 0 $ for some $y_0 \in \Y$, where $t \in [0,T]$ and $u \in \mx$, defining $\hat{\alpha}_t(u)(x) = 0 \, \forall x \in \X$ whenever $u \in \mx$ is such that $u(x) < \delta/4$ for some $x \in \X$, and defining $\hat{\alpha}(u)$ whenever $u$ is such that $u(x) \in [\delta/4, \delta/2]$ for some $x \in \X$ using a linear interpolation, we obtain bounded functions $\hat{\alpha}:[0,T] \times \mx \to \mathbb{R}^{|\X|}$ and $\hat{g}:[0,T] \times \mx \times \Y \to \mathbb{R}$. By a measurable selection theorem (see, for example,~Ekeland and Temam~\cite[Theorem~1.2, page 236]{ekeland-temam}), it follows that the mappings $[0,T] \times \mx \ni (t,u) \mapsto \hat{\alpha}_t(u)\in \mathbb{R}^{|\X|}$ and $[0,T] \times \mx \times \Y \ni (t, u, y) \mapsto \hat{g}_t(u,y)\in \mathbb{R}$ are measurable. By the Berge's maximum theorem (see, for example, Sundaram~\cite[Theorem~9.17, page 237]{sundaram-optimisation}) it follows that the functions $\hat{\alpha}$ and $\hat{g}$ are continuous on $\mx$.  

Since $\hat{\alpha}$ and $\hat{g}$ satisfy the assumptions of  Theorem~\ref{thm:u-extension}, there exists $(\tilde{\mu}, \tilde{\theta}) \in \Gamma$ that attains the supremum in~(\ref{eqn:sup-attained}) with $\hat{\alpha}$ and $\hat{g}$ in place of $\alpha$ and $g$, respectively. That is,
	\begin{align*}
	U_T^{\hat{\alpha},\hat{g}}(\tilde{\mu}, \tilde{\theta})= \tilde{I}(\tilde{\mu}, \tilde{\theta}).
	\end{align*}
	On the other hand, by~(\ref{eqn:istar-supinside}) and the above,
	\begin{align*}
	I^*(\tilde{\mu}, \tilde{\theta}) \geq U_T^{\hat{\alpha}, \hat{g}}(\tilde{\mu}, \tilde{\theta}) = \tilde{I}(\tilde{\mu}, \tilde{\theta}),
	\end{align*} 
	and since $\tilde{I}(\tilde{\mu}, \tilde{\theta}) \geq I^*(\tilde{\mu}, \tilde{\theta})$, we have that
	\begin{align}
	U_T^{\hat{\alpha}, \hat{g}}(\tilde{\mu}, \tilde{\theta})= I^*(\tilde{\mu}, \tilde{\theta})= \tilde{I}(\tilde{\mu}, \tilde{\theta}).
	\label{eqn:istar=itilde}
	\end{align}

	Note that $\tilde{\mu}_0  = \nu$ since $\tilde{I}(\tilde{\mu},\tilde{\theta}) < +\infty$.  We now proceed to show that $ \tilde{m}_t = \hat{m}_ t $ for almost all $t \in [0,T]$ and $\tilde{\mu} = \hat{\mu}$. This would establish $\tilde{I}(\hat{\mu},\hat{\theta	}) = I^*(\hat{\mu},\hat{\theta})$.

By~(\ref{eqn:istar=itilde}), we have
\begin{align}
\tilde{m}_t(y) & \sum_{\substack{y^\prime \in \Y : \\ (y,y^\prime) \in \EY}} \exp\{\hat{g}_t(\tilde{\mu}_t,y^\prime) - \hat{g}_t(\tilde{\mu}_t,y)\} \gamma_{y,y^\prime} (\tilde{\mu}_t) \nonumber \\
	& - \sum_{\substack{y_0 \in \Y:\\ (y_0, y) \in \EY}} \tilde{m}_t(y_0) \exp\{\hat{g}_t(\tilde{\mu}_t,y) - \hat{g}_t(\tilde{\mu}_t, y_0)\}\gamma_{y_0, y}(\tilde{\mu}_t) = 0, \, \, \forall y \in \Y,
\label{eqn:ghat-2}
\end{align}
for almost all $t \in [0,T]$. By assumption~\ref{assm:b2}, the Markov process on $\Y$ with transition rates $ \exp\{\hat{g}_t(\tilde{\mu}_t,y^\prime) - \hat{g}_t(\tilde{\mu}_t,y)\} \gamma_{y,y^\prime} (\tilde{\mu}_t), (y,y^\prime) \in \EY$, possesses a unique invariant probability measure; comparing~(\ref{eqn:ghat}) with $u = \tilde{\mu}_t$ and~(\ref{eqn:ghat-2}), we get
	\begin{align}
	\tilde{m}_t = \hat{m}_t
	\label{eqn:mthad=mttilde}
	\end{align}
	for almost all $t \in [0,T]$. 

On one hand, by using the first order optimality condition in~(\ref{eqn:alphahat}) with $u = \hat{\mu}_t$, and the just established fact that $\tilde{m}_t  = \hat{m}_t$ for almost all $t \in [0,T]$, we get
	\begin{align}
	\dot{\hat{\mu}}_t(x)& -   (\bar{\Lambda}^*_{\hat{\mu}_t, \tilde{m}_t} \hat{\mu}_t)(x)  \nonumber \\
	& +\hat{\mu}_t(x) \sum_{\substack{x^\prime\in \X: \\ (x,x^\prime) \in \EX}} (\exp\{\hat{\alpha}_t(\hat{\mu}_t)(x^\prime) - \hat{\alpha}_t(\hat{\mu}_t)(x)\} - 1) \bar{\lambda}_{x,x^\prime}  (\hat{\mu}_t, \tilde{m}_t) \nonumber \\
	& - \sum_{\substack{x_0 \in \X: \\ (x_0, x) \in \EX}} \hat{\mu}_t(x_0) (\exp\{\hat{\alpha}_t(\hat{\mu}_t)(x) - \hat{\alpha}_t(\hat{\mu}_t)(x_0)\} -1) \bar{\lambda}_{x_0,x}(\hat{\mu}_t,\tilde{m}_t)  = 0,\, \, \forall x \in \X,
	\label{eqn:alphahat-1}
	\end{align}
for almost all $t \in [0,T]$. On the other hand, by~(\ref{eqn:istar=itilde}), we get
	\begin{align}
	\dot{\tilde{\mu}}_t(x)& -   (\bar{\Lambda}^*_{\tilde{\mu}_t, \tilde{m}_t} \tilde{\mu}_t)(x)  \nonumber \\
	& +\tilde{\mu}_t(x) \sum_{\substack{x^\prime\in \X: \\ (x,x^\prime) \in \EX}} (\exp\{\hat{\alpha}_t(\tilde{\mu}_t)(x^\prime) - \hat{\alpha}_t(\tilde{\mu}_t)(x)\} - 1) \bar{\lambda}_{x,x^\prime}  (\tilde{\mu}_t, \tilde{m}_t) \nonumber \\
	& - \sum_{\substack{x_0 \in \X: \\ (x_0, x) \in \EX}} \tilde{\mu}_t(x_0) (\exp\{\hat{\alpha}_t(\tilde{\mu}_t)(x) - \hat{\alpha}_t(\tilde{\mu}_t)(x_0)\} -1) \bar{\lambda}_{x_0,x}(\tilde{\mu}_t,\tilde{m}_t)  = 0,\, \, \forall x \in \X,
	\label{eqn:alphahat-2}
	\end{align}
	for almost all $t \in [0,T]$. Note that, by the optimality condition~(\ref{eqn:alphahat}) and by~(\ref{eqn:mthad=mttilde}), the mapping
		\begin{align*}
	u &\mapsto \biggr( (\bar{\Lambda}^*_{u, \tilde{m}_t} u)(x)  +u(x) \sum_{\substack{x^\prime\in \X: \\ (x,x^\prime) \in \EX}} (\exp\{\hat{\alpha}_t(u)(x^\prime) - \hat{\alpha}_t(u)(x)\} - 1) \bar{\lambda}_{x,x^\prime}  (u, \tilde{m}_t) \nonumber \\
	& \qquad - \sum_{\substack{x_0 \in \X: \\ (x_0, x) \in \EX}} u(x_0) (\exp\{\hat{\alpha}_t(u)(x) - \hat{\alpha}_t(u)(x_0)\} -1) \bar{\lambda}_{x_0,x}(u,\tilde{m}_t), \,  x \in \X \biggr) \in \mathbb{R}^{|\X|}
	\end{align*} 
on $\{u \in \mx: u(x) \geq \delta/2 \, \forall x \in \X\}$ is identically equal to $\dot{\hat{\mu}}_t$ for almost all $t \in [0,T]$. Hence, by~(\ref{eqn:alphahat-1}) and~(\ref{eqn:alphahat-2}), and noting that $\tilde{\mu}_0 = \hat{\mu}_0 = \nu$, Gronwall inequality implies that $\tilde{\mu}_t =  \hat{\mu}_t$ for all $t \in [0,T]$.
	
We have thus shown that $(\tilde{\mu},\tilde{\theta}) = (\hat{\mu}, \hat{\theta})$, and the second equality in~(\ref{eqn:istar=itilde}) implies that $\tilde{I}(\hat{\mu},\hat{\theta}) = I^*(\hat{\mu},\hat{\theta})$. This completes the proof of the theorem.
\end{proof}

\section{Approximating the subsequential rate function}
\label{section:uniqueness-rate-function-whole-space}
Let $\tilde{I}:\DX \times \DY \to [0, + \infty]$ be a subsequential rate function for the family $ \{(\mu_N, \theta_N)\}_{N \geq 1}$, and suppose that, for some $\nu \in \mx$, $\tilde{I}(\mu, \theta) = +\infty$ unless $\mu_0 =\nu$. In this section, we show that $\tilde{I}(\mu, \theta) = I^*(\mu, \theta)$ for all $(\mu, \theta) \in \DX \times \DY$. We shall proceed through a sequence of lemmas. In each lemma, we shall extend the conclusion $\tilde{I}(\mu, \theta) = I^*(\mu, \theta)$ to a larger class of elements $(\mu, \theta)$ by producing a sequence $(\mu^i, \theta^i)$ such that $\tilde{I}(\mu^i, \theta^i) = I^*(\mu^i, \theta^i)$ for all $i \geq 1$, $(\mu^i, \theta^i) \to (\mu, \theta)$ in $\DX \times \DY$ as $i \to\infty$, and $I^*(\mu^i, \theta^i) \to I^*(\mu, \theta)$ as $i \to \infty$. Using these approximations, we finally show that $\tilde{I}(\mu, \theta) = I^*(\mu, \theta)$ for all $(\mu, \theta) \in \DX \times \DY$ (see Theorem~\ref{thm:istar=itilde}).

\begin{figure}
\centering
\begin{tikzpicture}
\node at (-0.2,0) {$0$};
\node at (-0.2,5) {$1$};
\draw (0,0) -- (0,5);
\draw (0,0) -- (5,0);
\node at (2,-0.3) {$\hat{\tau}^i$};
\node at (1,-0.3) {$\frac{1}{i}$};
\node at (0.5,-0.3) {$\tau^i$};
\node at (4,-0.3) {$T$};
\node at (4.5,4.2){$\hat{\mu}(x)$};
\node at (4.5,3.2){$\hat{\mu}^i(x)$};
\draw [thick,red] (0,0) to[out=45,in=200] (1,1.2) to[out=20,in=240] (2,2) to[out=60,in=270] (3,3) to[out=90,in=225] (4,4);
\draw [thick,blue] (0,0) to[out=10,in=180] (0.5,0.3) to [out=20,in=240] (2,1.2);
\draw [thick,red] (2,1.2) to[out=20,in=240] (3,2) to[out=60,in=270] (4,3);
\draw [dashed] (0.5,0.3) -- (0.5,0);
\draw [dashed] (1,1.2) -- (2,1.2);
\draw [dashed] (1,1.2) -- (1,0);
\draw [dashed] (2,1.2) -- (2,0);
\draw [dashed] (4,5) -- (4,0);
\draw [dashed] (0,5) -- (4,5);
\end{tikzpicture}
\caption{Figure depicting the idea of construction of $\hat{\mu}^i$ in the proof of Lemma~\ref{lemma:fix-initial-condition}}
\label{fig:initial-cond}
\end{figure}
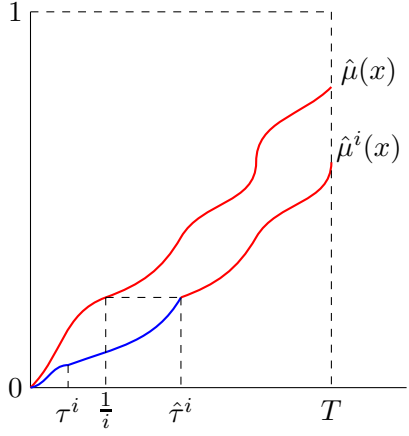

We start with an extension of the conclusion of Theorem~\ref{thm:muhat-bdd-away-from-0} to all initial conditions $\nu$.
\begin{lemma}Let $\nu \in \mx$ and let $\tilde{I} : \DX \times \DY \to [0, + \infty]$ be a subsequential rate function such that $\tilde{I}(\mu, \theta) = +\infty$ unless $\mu_0 = \nu$.  Suppose that $(\hat{\mu},\hat{\theta}) \in \DX \times  \DY$ is such that
\begin{itemize}
\item $I^*(\hat{\mu}, \hat{\theta}) < +\infty$,
\item $\inf_{t \in [\delta,T]} \min_{x \in \X}\hat{\mu}_t(x) > 0$ for all $\delta > 0$,
\item the mapping $[0,T] \ni t \mapsto \hat{\mu}_t \in \mx$ is Lipschitz continuous,
\item $\hat{\theta}$, when viewed as a measure on $[0,T] \times \Y$, admits the representation  $\hat{\theta}(dy dt) = \hat{m}_t(dy) dt$ for some $\hat{m}_t \in M_1(\Y)$ for almost all $t \in [0,T]$, and $\inf_{t \in [0,T]} \min_{y \in \Y} \hat{m}_t(y) >0$.
\end{itemize}	
Then $\tilde{I}(\hat{\mu},\hat{\theta}) = I^*(\hat{\mu},\hat{\theta})$.
		\label{lemma:fix-initial-condition}
	\end{lemma}	
\begin{proof}
We begin with some notations. Let $\X_0 = \{x \in \X: \hat{\mu}_0(x) = 0\}$. For each $x \in \X_0$, let $\{x^x_k, 1 \leq k \leq l(x)\}$ be such that $\hat{\mu}_0(x_1^x) \geq 1/|\X_0|$ (in particular, $x_1^x \notin \X_0$), $(x^x_k,x^x_{k+1}) \in \EX$ for all $1 \leq k \leq l(x)-1$ and $(x^x_{l(x)},x) \in \EX$, i.e., the collection of edges $\{(x^x_k, x^x_{k+1}), 1 \leq k \leq l(x)-1\} \cup (x^x_{l(x)}, x)$ form a directed path of length $l(x)$ from $x_1^x$ to $x$. Also, for the given $\nu \in \mx$, let $\mu(\nu, \hat{\theta}) \in D([0,\infty),\mx)$ denote the unique solution to the ODE $\dot{\mu}_t =\bar{\Lambda}^*_{\mu_t, \hat{m}_t} \mu_t$ with initial condition $ \mu_ 0 = \nu$.

For each $i \geq 1$, we define a path $\hat{\mu}^i \in \DX$ as follows. Define $\hat{\mu}^i_t = \mu_t(\hat{\mu}_0,\hat{\theta})$ for $t \in [0, \tau^i]$ where $\tau^i = \inf \{t>0: \mu_t(\hat{\mu}_0, \hat{\theta})(x) = \hat{\mu}_{1/i}(x)/2 \text{ for some } x \in \X_0\}$. Note that $\tau^i < +\infty$ for $i$ sufficiently large. Also note that $\hat{\mu}^i_{\tau^i}(x) > 0$ for all $x \in \X$, and that the supremum over $\alpha_t$ in the definition of $I^*(\hat{\mu}^i, \hat{\theta})$ (see~(\ref{eqn:istar-supinside})) is attained at $\alpha_t = 0$ for all $t \in [0,\tau^i]$. Let $\varepsilon_i(x) = \hat{\mu}_{1/i}(x) - \hat{\mu}^i_{\tau^i}(x)$ for $x \in \X$ and $i \geq 1$. Since the mapping $t \mapsto \hat{\mu}_t$ is Lipschitz continuous, we see that $\tau^i \to 0$ as $i \to \infty$, and $\varepsilon_i(x) \to 0$ as $i \to \infty$ for all $x \in \X$. For each $x \in \tilde{\X}_0 \coloneqq \X_0 \cap \{x \in \X_0 : \varepsilon_i(x) > 0\}$, we shall now move the mass $\varepsilon_i(x)$ from the vertex $x_1^x$ to $x$ via the edges defined in the previous paragraph using a piecewise constant velocity path. Denote the elements of $\tilde{\X}_0$ by $x_1, x_2 \ldots, x_{|\tilde{\X}_0|}$, let  $l(x_0) = 0$ and $\varepsilon_i(x_0) = 0$. Given $r \in \{0, 1, \ldots, |\tilde{\X}_0|-1 \}$, $s \in \{0, 1, \ldots, l(x_{r+1})-1\}$, and $t \in [\tau^i+\sum_{m=0}^r l(x_m) \varepsilon_i(x_m) + s \varepsilon_i(x_{r+1}),\tau^i+\sum_{m=0}^r l(x_m) \varepsilon_i(x_m) + (s+1) \varepsilon_i(x_{r+1}))$, define
\begin{align*}
\dot{\hat{\mu}}^i_t(x) \coloneqq \left\{
\begin{array}{lll}
-1 & \text{ if } x = x^{x_{r+1}}_{s+1} \\
1 & \text{ if } x = x^{x_{r+1}}_{s+2} \\
0 & \text{ otherwise},
\end{array}
\right.
\end{align*}
i.e., we transport a mass of $\varepsilon_i(x_{r+1})$ at unit rate from the node $x^{x_{r+1}}_{s+1}$ to $x^{x_{r+1}}_{s+2}$ during the above time interval. Note that we have $\hat{\mu}^i_t(x) = \hat{\mu}_t(x)$ for all $x \in \tilde{\X}_0$ at time $t = \tau^i+\sum_{m=1}^{|\tilde{\X}_0|} l(x_m) \varepsilon_i(x_m)$. Similarly, for $x \in \X \setminus \tilde{\X}_0$ with $\varepsilon_i(x) > 0$, one defines a sequence of edges from a suitable $x^\prime \in \X \setminus \tilde{\X}_0$ (possibly from multiple $x^\prime \in \X \setminus \tilde{\X}_0$) with $\varepsilon_i(x^\prime) < 0$ and moves the mass $\varepsilon_i(x)$ to $x$ through similar piecewise constant velocity trajectories defined above. For each $x \in \X \setminus \tilde{\X}_0$ with $\varepsilon_i(x) < 0$, we similarly move the mass $\varepsilon_i(x)$ from $x$ to suitable vertices in $\X \setminus \tilde{\X}_0$ via piecewise constant velocity trajectories.  At the end of this procedure, we have $\hat{\mu}^i_{\hat{\tau}^i} = \hat{\mu}_{1/i}$ for some $\hat{\tau}^i \geq \tau^i$. We now define $\hat{\mu}^i_t = \hat{\mu}_{t+1/i - \hat{\tau}^i}$ for all $t \in [\hat{\tau}^i, T]$ (see Figure~\ref{fig:initial-cond} for a pictorial representation of $\hat{\mu}^i$). Since $\varepsilon_i(x) \to 0$ as $i \to \infty$ for all $x \in \X$, we have that $\hat{\tau}^i \to 0$ as $i \to \infty$. 

Also, for each $i \geq 1$ and $t \in [0,T]$, define the probability measure $\hat{m}_t^i$ on $\Y$ by
\begin{align*}
\hat{m}_t^i(y) \coloneqq \left\{
\begin{array}{lll}
\hat{m}_t(y) & \text{ if } t \in [0,\tau^i],\\
\hat{m}_{\tau^i}(y) & \text{ if } t\in [\tau^i,\hat{\tau}^i], \\
\hat{m}_{t+1/i-\hat{\tau}^i}(y) & \text{ if } t \in (\hat{\tau}^i,T], 
\end{array}
\right.
\end{align*}
for all $y \in \Y$, and define the measure $\hat{\theta}^i$ on $[0,T] \times \Y$ by $\hat{\theta}^i(dy dt) = \hat{m}_t^i(dy) dt$. Clearly, $\hat{\theta}^i \in \DY$.

Thanks to the fact that $\hat{\mu}^i_{{\tau}^i}(x) > 0$ for all $x \in \X$ and the fact that $\alpha_t = 0$ attains the supremum in the definition of $I^*(\hat{\mu}^i, \hat{\theta}^i)$ for all $t \in [0, \tau^i]$, using arguments similar to those used in the proof of Theorem~\ref{thm:muhat-bdd-away-from-0}, one can now construct a bounded measurable function $\hat{\alpha}^i : [0,T] \times \mx \to \mathbb{R}^{|\X|}$ such that $\hat{\alpha}_t^i(\hat{\mu}^i_t)$ attains the supremum over ${\alpha}_t$ in the definition of $I^*(\hat{\mu}^i, \hat{\theta}^i)$ (in~(\ref{eqn:istar-supinside})) and $\hat{\alpha}^i_t(\cdot)$ is continuous on $\mx$ for all $t \in [0,T]$. Similarly, since $\hat{\theta}^i$ satisfies the conditions of Theorem~\ref{thm:muhat-bdd-away-from-0}, one can construct a bounded measurable function $\hat{g}^i: [0,T] \times \mx \times \Y \to \mathbb{R}$ such that $\hat{g}^i_t(\hat{\mu}^i_t, \cdot)$ attains the supremum over $g_t$ in the definition of $I^*(\hat{\mu}^i,\hat{\theta}^i)$ and $\hat{g}^i_t(\cdot)$ is continuous on $\mx$ for each $t \in [0,T]$.  Hence, using arguments similar to those used in the proof of Theorem~\ref{thm:muhat-bdd-away-from-0}, one concludes that $\tilde{I}(\hat{\mu}^i, \hat{\theta}^i) = I^*(\hat{\mu}^i, \hat{\theta}^i)$ for all $i \geq 1 $.

Let us now show that $I^*(\hat{\mu}^i, \hat{\theta}^i) \to I^*(\hat{\mu}, \hat{\theta})$ as $i \to \infty$. For the fast component, since $\hat{\tau}^i \to 0$, we see that $\hat{\theta}^i \to \hat{\theta}$ in $\DY$ as $i \to \infty$. By assumption~\ref{assm:b2}, we see that
\begin{align*}
0 \leq  \sup_{i \geq 1, t \in [0,T]} & \biggr\{ \sup_{g_t \in \mathbb{R}^{|\Y|}} - \int_{\Y} \biggr( L_{\hat{\mu}^i_t}g_t(\cdot)(y)  \\
& +  \int_{\EY}  \tau(Dg_t(y,\Delta)) \gamma_{y,y+d\Delta}(\hat{\mu}^i_t)\biggr) \hat{m}^i_t(dy) \biggr\} < +\infty,
\end{align*}
and hence the bounded convergence theorem immediately yields
\begin{align*}
\int_{[0,\hat{\tau}^i]} \sup_{g_t \in B(\Y)} & \biggr\{ -\int_\Y \left(L_{\hat{\mu}^i_t}g_t(\cdot)(y) + \int_{\EY} \tau(Dg_t(y,\Delta)) \gamma_{y,y+d\Delta}(\hat{\mu}^i_t) \right) \hat{m}^i_t(dy) \biggr\} dt \to 0
\end{align*}
and
\begin{align*}
\int_{[T+1/i-\hat{\tau}^i,T]}& \sup_{g_t \in B(\Y)} \left\{ -\int_\Y \left(L_{\hat{\mu}_t}g_t(\cdot)(y) + \int_{\EY} \tau(Dg_t(y,\Delta)) \gamma_{y,y+d\Delta}(\hat{\mu}_t) \right) \hat{m}_t(dy) \right\} dt \to 0
\end{align*}
as $i \to \infty$. Noting that
$\hat{m}^i_t = \hat{m}_{t+1/i-\hat{\tau}^i}$ and $\hat{\mu}^i_t = \hat{\mu}_{t+1/i-\hat{\tau}^i}$ for all $t \in [\hat{\tau}^i, T]$, the above convergences imply that
\begin{align*}
\int_{[0,T]}& \sup_{g_t \in B(\Y)} \left\{ -\int_\Y \left(L_{\hat{\mu}^i_t}g_t(\cdot)(y) + \int_{\EY} \tau(Dg_t(y,\Delta)) \gamma_{y,y+d\Delta}(\hat{\mu}^i_t) \right) \hat{m}^i_t(dy) \right\} dt \\
& \to \int_{[0,T]}\sup_{g_t \in B(\Y)} \left\{ -\int_\Y \left(L_{\hat{\mu}_t}g_t(\cdot)(y) + \int_{\EY} \tau(Dg_t(y,\Delta)) \gamma_{y,y+d\Delta}(\hat{\mu}_t) \right) \hat{m}_t(dy)\right\} dt
\end{align*}
as $i \to \infty$. 

For the slow component, since $\hat{\tau}^i \to 0$ as $i \to \infty$, using the absolute continuity of the mapping $t \mapsto \hat{\mu}_t $ and the definition of the paths $\hat{\mu}^i$, it follows from the dominated convergence theorem that $\hat{\mu}^i_t \to \hat{\mu}_t$ as $i \to \infty$ uniformly in $t \in [0,T]$ and hence we have that $\hat{\mu}^i \to \hat{\mu}$ in $\DX$ as $i \to \infty$. Let us first show that
\begin{align*}
\int_{[0,\hat{\tau}^i]}  \sup_{\alpha_t \in \mathbb{R}^{|\X|}} \biggr\{ \langle \alpha_t, \dot{\hat{\mu}}^i_t  - \bar{\Lambda}^*_{\hat{\mu}^i_t, \hat{m}^i_t}\hat{\mu}^i_t \rangle - \int_{\X \times \EX} \tau(D\alpha_t(x,\Delta))\bar{\lambda}_{x,x+d\Delta}(\hat{\mu}^i_t,\hat{m}^i_t) \hat{\mu}^i_t(dx) \biggr\} dt
\end{align*}
converges to $0$ as $i \to \infty$. Towards this, let $t \in [\tau^i+\sum_{m=0}^r l(x_m) \varepsilon_i(x_m) + s \varepsilon_i(x_{r+1}),\tau^i+\sum_{m=0}^r l(x_m) \varepsilon_i(x_m) + (s+1) \varepsilon_i(x_{r+1}))$ where  $r \in \{0, 1, \ldots, |\tilde{\X}_0|-1 \}$, and $s \in \{0, 1, \ldots, l(x_{r+1})-1\}$. Note that, we have 
\begin{align*}
\sup_{\alpha_t \in \mathbb{R}^{|\X|}} & \biggr\{ \langle \alpha_t, \dot{\hat{\mu}}^i_t  - \bar{\Lambda}^*_{\hat{\mu}^i_t, \hat{m}^i_t}\hat{\mu}^i_t \rangle - \int_{\X \times \EX} \tau(D\alpha_t(x,\Delta))\bar{\lambda}_{x,x+d\Delta}(\hat{\mu}^i_t,\hat{m}^i_t) \hat{\mu}^i_t(dx) \biggr\}  \\
& \leq \sup_{\alpha_t \in \mathbb{R}^{|\X|}} \biggr( (\alpha_t(x^{x_{r+1}}_{s+2}) - \alpha_t(x^{x_{r+1}}_{s+1})) - (\exp\{\alpha_t(x^{x_{r+1}}_{s+2}) - \alpha_t(x^{x_{r+1}}_{s+1}) \} - 1) \\
& \qquad \qquad \times \bar{\lambda}_{x^{x_{r+1}}_{s+1},x^{x_{r+1}}_{s+2}}(\hat{\mu}^i_t, \hat{m}^i_t) \hat{\mu}^i_t(x^{x_{r+1}}_{s+1}) \biggr) \\
& \qquad -  \inf_{\alpha_t \in \mathbb{R}^{|\X|}}\sum_{\substack{(x,x^\prime) \in \EX : \\ (x,x^\prime) \neq (x^{x_{r+1}}_{s+1}, x^{x_{r+1}}_{s+2})}} (\exp\{\alpha_t(x^\prime) - \alpha_t(x)\} - 1) \bar{\lambda}_{x,x^\prime}(\hat{\mu}^i_t, \hat{m}^i_t) \hat{\mu}^i_t(x) \\
& \leq \log\frac{1}{c\hat{\mu}^i_t(x^{x_{r+1}}_{s+1})} + c_1
\end{align*}
where $c =\min_{(x,x^\prime) \in \EX}\min_{y \in \Y} \min_{\xi \in \mx}\lambda_{x,x^\prime}(\xi, y)$ and $c_1 > 0$ is a suitable constant to bound the extra additive terms. Hence, using a variable change $u = c\hat{\mu}^i_t(x^{x_{r+1}}_{s+1})$, we see that
\begin{align*}
\int \sup_{\alpha_t \in \mathbb{R}^{|\X|}} & \biggr\{ \langle \alpha_t, \dot{\hat{\mu}}^i_t  - \bar{\Lambda}^*_{\hat{\mu}^i_t, \hat{m}^i_t}\hat{\mu}^i_t \rangle - \int_{\X \times \EX} \tau(D\alpha_t(x,\Delta))\bar{\lambda}_{x,x+d\Delta}(\hat{\mu}^i_t,\hat{m}^i_t) \hat{\mu}^i_t(dx) \biggr\} dt \\
& \leq -\frac{1}{c}(u \log u -u)|_{c{\hat{\mu}^i}_{t_1}(x^{x_{r+1}}_{s+1})}^{c{\hat{\mu}^i}_{t_2}(x^{x_{r+1}}_{s+1})} +c_1\varepsilon_i(x_{r+1})  \\
& = o(1)
\end{align*}
as $i \to \infty$, where $t_1 = \tau^i+\sum_{m=0}^r l(x_m) \varepsilon_i(x_m) + s \varepsilon_i(x_{r+1}) $, $t_2 = t_1 + \varepsilon_i(x_{r+1})$ and the above integral is evaluated over the time interval $[\tau^i+\sum_{m=0}^r l(x_m) \varepsilon_i(x_m) + s \varepsilon_i(x_{r+1}),\tau^i+\sum_{m=0}^r l(x_m) \varepsilon_i(x_m) + (s+1) \varepsilon_i(x_{r+1}))$. Hence, repeating the above calculation for each constant velocity section of the path $\hat{\mu}^i$ during the time interval $[\tau^i,\hat{\tau}^i]$, we see that
\begin{align*}
\int_{[0,\hat{\tau}^i]}  \sup_{\alpha_t \in \mathbb{R}^{|\X|}} \biggr\{ \langle \alpha_t, \dot{\hat{\mu}}^i_t  - \bar{\Lambda}^*_{\hat{\mu}^i_t, \hat{m}^i_t}\hat{\mu}^i_t \rangle - \int_{\X \times \EX} \tau(D\alpha_t(x,\Delta))\bar{\lambda}_{x,x+d\Delta}(\hat{\mu}^i_t,\hat{m}^i_t) \hat{\mu}^i_t(dx) \biggr\} dt
\end{align*}
converges to $0$ as $i \to \infty$. Therefore, noting that $\hat{\mu}^i_t = \hat{\mu}_{t+1/i - \hat{\tau}^i}$ and $\hat{m}^i_t = \hat{m}_{t+1/i-\hat{\tau}^i}$ for $t \in [\hat{\tau}^i, T]$, and $\hat{\mu}^i_t = \mu_t(\hat{\mu}_0,\hat{\theta})$ on $t \in [0,\tau^i]$, we have
\begin{align*}
\biggr| & \int_{[0,T]}\sup_{\alpha_t \in \mathbb{R}^{|\X|}} \biggr\{ \langle \alpha_t, \dot{\hat{\mu}}_t  - \bar{\Lambda}^*_{\hat{\mu}_t, \hat{m}_t}\hat{\mu}_t \rangle - \int_{\X \times \EX} \tau(D\alpha_t(x,\Delta))\bar{\lambda}_{x,x+d\Delta}(\hat{\mu}_t,\hat{m}_t) \hat{\mu}_t(dx) \biggr\} dt  \\
& \qquad -  \int_{[0,T]} \sup_{\alpha_t \in \mathbb{R}^{|\X|}} \biggr\{ \langle \alpha_t, \dot{\hat{\mu}}^i_t  - \bar{\Lambda}^*_{\hat{\mu}^i_t, \hat{m}^i_t}\hat{\mu}^i_t \rangle - \int_{\X \times \EX} \tau(D\alpha_t(x,\Delta))\bar{\lambda}_{x,x+d\Delta}(\hat{\mu}^i_t,\hat{m}^i_t) \hat{\mu}^i_t(dx) \biggr\} dt \biggr| \\
& \leq 
\int_{[0,1/i]}  \sup_{\alpha_t \in \mathbb{R}^{|\X|}} \biggr\{ \langle \alpha_t, \dot{\hat{\mu}}_t  - \bar{\Lambda}^*_{\hat{\mu}_t, \hat{m}_t}\hat{\mu}_t \rangle - \int_{\X \times \EX} \tau(D\alpha_t(x,\Delta))\bar{\lambda}_{x,x+d\Delta}(\hat{\mu}_t,\hat{m}_t) \hat{\mu}_t(dx) \biggr\} dt \\
&\qquad + \int_{[T+1/i - \hat{\tau}^i,T]}  \sup_{\alpha_t \in \mathbb{R}^{|\X|}} \biggr\{ \langle \alpha_t, \dot{\hat{\mu}}_t  - \bar{\Lambda}^*_{\hat{\mu}_t, \hat{m}_t}\hat{\mu}_t \rangle - \int_{\X \times \EX} \tau(D\alpha_t(x,\Delta))\bar{\lambda}_{x,x+d\Delta}(\hat{\mu}_t,\hat{m}_t) \hat{\mu}_t(dx) \biggr\} dt \\
& \qquad + \int_{[0,\hat{\tau}^i]}  \sup_{\alpha_t \in \mathbb{R}^{|\X|}} \biggr\{ \langle \alpha_t, \dot{\hat{\mu}}^i_t  - \bar{\Lambda}^*_{\hat{\mu}^i_t, \hat{m}^i_t}\hat{\mu}^i_t \rangle - \int_{\X \times \EX} \tau(D\alpha_t(x,\Delta))\bar{\lambda}_{x,x+d\Delta}(\hat{\mu}^i_t,\hat{m}^i_t) \hat{\mu}^i_t(dx) \biggr\} dt \\
& \to 0
\end{align*}
as $i \to \infty$. We have thus shown that $I^*(\hat{\mu}^i, \hat{\theta}^i) \to I^*(\hat{\mu}, \hat{\theta})$ as $i \to \infty$.

Since $(\hat{\mu}^i, \hat{\theta}^i) \to (\hat{\mu},\hat{\theta})$ in $\DX \times \DY$ as $i\to \infty$, the lower semicontinuity of $\tilde{I}$ implies that $\liminf_{i \to \infty} \tilde{I}(\hat{\mu}^i, \hat{\theta}^i) \geq \tilde{I}(\hat{\mu}, \hat{\theta})$. Therefore, using the above convergence and the fact that $\tilde{I}(\hat{\mu}^i, \hat{\theta}^i) = I^*(\hat{\mu}^i, \hat{\theta}^i)$ for all $i \geq 1$, we see that $\tilde{I}(\hat{\mu}, \hat{\theta}) \leq I^*(\hat{\mu}, \hat{\theta})$. On the other hand, since $\tilde{I}(\hat{\mu}, \hat{\theta})  \geq I^*(\hat{\mu}, \hat{\theta})$, it follows that $\tilde{I}(\hat{\mu}, \hat{\theta}) = I^*(\hat{\mu}, \hat{\theta})$. This completes the proof of the lemma.
\end{proof}

\begin{remark}
\label{remark:itilde=istar}
We shall repeatedly use the immediately preceding argument; starting with an element $(\hat{\mu}, \hat{\theta}) \in \DX \times \DY$, we shall produce a sequence $(\hat{\mu}^i, \hat{\theta}^i) \in \DX \times \DY$, $i \geq 1$, such that $\tilde{I}(\hat{\mu}^i, \hat{\theta}^i) = I^*(\hat{\mu}^i, \hat{\theta}^i)$ for all $i \geq 1$, $(\hat{\mu}^i, \hat{\theta}^i) \to (\hat{\mu}, \hat\theta)$ in $\DX \times \DY$ as $i \to \infty$ and $I^*(\hat{\mu}^i, \hat{\theta}^i) \to I^*(\hat{\mu}, \hat{\theta})$ as $i\to \infty$, and use the above argument to conclude that $\tilde{I}(\hat{\mu}, \hat{\theta}) = I^*(\hat{\mu}, \hat{\theta})$.
\end{remark}

We now extend the conclusion of the previous lemma to all elements $\hat{\theta}  \in \DY$.
\begin{lemma}Let $\nu \in \mx$ and let $\tilde{I} : \DX \times \DY \to [0, + \infty]$ be a subsequential rate function such that $\tilde{I}(\mu, \theta) = +\infty$ unless $\mu_0 = \nu$. Suppose that
 $(\hat{\mu},\hat{\theta}) \in \DX \times \DY$ is such that 
\begin{itemize}
\item $I^*(\hat{\mu}, \hat{\theta}) < +\infty$,
\item $\inf_{t \in [\delta,T]} \min_{x \in \X}\hat{\mu}_t(x) > 0$ for all $\delta > 0$,
\item the mapping $[0,T] \ni t \mapsto \hat{\mu}_t \in \mx$ is Lipschitz continuous.
\end{itemize}	 
Then $\tilde{I}(\hat{\mu}, \hat{\theta}) = I^*(\hat{\mu}, \hat{\theta})$.
\label{lemma:fix-mhat}
\end{lemma}
\begin{proof} 
Let $\hat{\theta}$, when viewed as a measure on $[0,T] \times \Y$, admit the representation $\hat{\theta}(dy dt) = \hat{m}_t(dy) dt$, where $\hat{m}_t \in M_1(\Y)$ for almost all $t \in [0,T]$. For each $i \geq 1$ and for each $t \in [0,T]$, define the probability measure $\hat{m}_t^i$ on $\Y$ by
		\begin{align} \hat{m}_t^i(y) = \frac{\hat{m}_t(y)+1/i}{1+|\Y|/i}, y \in \Y, \label{eqn:mti-approx}
		\end{align}
and, for each $i \geq 1$, define the measure $\hat{\theta}^i(dy dt)$ on $[0,T] \times \my$ by $\hat{\theta}^i(dy dt) \coloneqq \hat{m}^i_t(dy)dt$.
Clearly, $\hat{\theta}^i \in \DY$ for all $i \geq 1$, and $\hat{\theta}^i \to \hat{\theta}$ in $\DY$ as $i \to \infty$. Since $(\hat{\mu},\hat{\theta}^i)$ satisfies the assumptions of Lemma~\ref{lemma:fix-initial-condition}, we have $\tilde{I}(\hat{\mu},\hat{\theta}^i) = I^*(\hat{\mu},\hat{\theta}^i)$.

Since, for each $t \in [0,T]$, the mapping
\begin{align*}
(g_t,m_t) \mapsto  \max\biggr\{- \int_{\Y} \biggr( L_{\hat{\mu}_t}g_t(\cdot)(y) + \int_{\EY}  \tau(Dg_t(y,\Delta)) \gamma_{y,y+d\Delta}(\hat{\mu}_t)\biggr) m_t(dy) , 0 \biggr\}
\end{align*}
on $(\mathbb{R}\cup\{+\infty, -\infty\})^{|\Y|} \times M_1(\Y)$ is bounded and continuous (thanks to assumption~\ref{assm:b2}), by an application of the Berge's maximum theorem, it follows that the mapping
\begin{align}
m_t \mapsto \sup_{g_t \in \mathbb{R}^{|\Y|}} - \int_{\Y} \biggr( L_{\hat{\mu}_t}g_t(\cdot)(y) + \int_{\EY}  \tau(Dg_t(y,\Delta)) \gamma_{y,y+d\Delta}(\hat{\mu}_t)\biggr) m_t(dy)
\label{eqn:mt-param-problem}
\end{align}
is continuous on $M_1(\Y)$. Similarly, for each $t \geq 0$, by assumption~\ref{assm:a2}, it follows that the mapping
\begin{align*}
(\alpha_t, m_t) \mapsto   \langle \alpha_t , \dot{\hat{\mu}}_t - \bar{\Lambda}^*_{\hat{\mu}_t,m_t} \rangle - \int_{\X \times \EX} \tau(D\alpha_t(x,\Delta)) \bar{\lambda}_{x.x+d\Delta}(\hat{\mu}_t, m_t) \hat{\mu}_t(dx) 
\end{align*}
is bounded and continuous on $\mathbb{R}^{|\X|} \times M_1(\Y)$. Again, by the Berge's maximum theorem,
\begin{align*}
m_t \mapsto  \sup_{\alpha_t \in \mathbb{R}^{|\X|}} \biggr\{ \langle \alpha_t , \dot{\hat{\mu}}_t - \bar{\Lambda}^*_{\hat{\mu}_t,m_t} \rangle - \int_{\X \times \EX} \tau(D\alpha_t(x,\Delta)) \bar{\lambda}_{x,x+d\Delta}(\hat{\mu}_t, m_t) \hat{\mu}_t(dx) \biggr\}
\end{align*}
is continuous on $M_1(\Y)$. Therefore, for each $t \in [0,T]$, we  see that
\begin{align*}
 \sup_{\alpha_t \in \mathbb{R}^{|\X|}} &  \biggr\{ \langle \alpha_t , \dot{\hat{\mu}}_t - \bar{\Lambda}^*_{\hat{\mu}_t,\hat{m}^i_t} \rangle - \int_{\X \times \EX} \tau(D\alpha_t(x,\Delta)) \bar{\lambda}_{x,x+d\Delta}(\hat{\mu}_t, \hat{m}^i_t) \hat{\mu}_t(dx) \biggr\} \\
 & \to  \sup_{\alpha_t \in \mathbb{R}^{|\X|}} \biggr\{ \langle \alpha_t , \dot{\hat{\mu}}_t - \bar{\Lambda}^*_{\hat{\mu}_t,\hat{m}_t} \rangle - \int_{\X \times \EX} \tau(D\alpha_t(x,\Delta)) \bar{\lambda}_{x,x+d\Delta}(\hat{\mu}_t, \hat{m}_t) \hat{\mu}_t(dx) \biggr\}, 
\end{align*}
and
\begin{align*}
 \sup_{g_t \in B(\Y)} & - \int_{\Y} \biggr( L_{\hat{\mu}_t}g_t(\cdot)(y) + \int_{\EY}  \tau(Dg_t(y,\Delta)) \gamma_{y,y+d\Delta}(\hat{\mu}_t)\biggr) \hat{m}^i_t(dy) \\
 & \to \sup_{g_t \in B(\Y)}  - \int_{\Y} \biggr( L_{\hat{\mu}_t}g_t(\cdot)(y) + \int_{\EY}  \tau(Dg_t(y,\Delta)) \gamma_{y,y+d\Delta}(\hat{\mu}_t)\biggr) \hat{m}_t(dy)
\end{align*}
as $i \to \infty$. Noting that
\begin{align*}
0 \leq \sup_{i \geq 1, t \in [0,T]}& \sup_{\alpha_t \in \mathbb{R}^{|\X|}} \biggr\{ \langle \alpha_t , \dot{\hat{\mu}}_t - \bar{\Lambda}^*_{\hat{\mu}_t,\hat{m}^i_t} \rangle  \\
& - \int_{\X \times \EX} \tau(D\alpha_t(x,\Delta)) \bar{\lambda}_{x,x+d\Delta}(\hat{\mu}_t, \hat{m}^i_t) \hat{\mu}_t(dx) \biggr\} < +\infty
\end{align*}
and
\begin{align*}
0 \leq  \sup_{i \geq 1, t \in [0,T]} &\sup_{g_t \in \mathbb{R}^{|\Y|}}  \biggr\{ - \int_{\Y} \biggr( L_{\hat{\mu}_t}g_t(\cdot)(y)  \\
& +  \int_{\EY}  \tau(Dg_t(y,\Delta)) \gamma_{y,y+d\Delta}(\hat{\mu}_t)\biggr) \hat{m}^i_t(dy) \biggr\} < +\infty,
\end{align*}
using the bounded convergence theorem, we obtain that $I^*(\hat{\mu},\hat{\theta}^i) \to I^*(\hat{\mu},\hat{\theta})$ as $i \to \infty$. Thanks to Remark~\ref{remark:itilde=istar}, this completes the proof of the lemma.
\end{proof}

We now extend the conclusion of the previous lemma to the case when the mapping $[0,T] \ni t \mapsto \mu_t \in \mx$ is not necessarily Lipschitz continuous.

\begin{lemma}
Let $\nu \in \mx$ and let $\tilde{I} : \DX \times \DY \to [0, + \infty]$ be a subsequential rate function such that $\tilde{I}(\mu, \theta) = +\infty$ unless $\mu_0 = \nu$. Suppose that $(\hat{\mu},\hat{\theta}) \in \DX \times \DY$ is such that $I^*(\hat{\mu},\hat{\theta}) < +\infty$, and $\inf_{t \in [\delta,T]} \min_{x \in \X} \hat{\mu}_t(x) > 0$ for all $\delta > 0$. Then $\tilde{I}(\hat{\mu}, \hat{\theta}) = I^*(\hat{\mu},\hat{\theta})$.
\label{lemma:fix-muhat-lipschitz}
\end{lemma}
\begin{proof}
Let us first suppose that the mapping $t \mapsto \hat{\mu}_t$ is locally Lipschitz continuous at $t=0$ so that $\sup_{t \in [0, \eta]}\|\dot{\hat{\mu}}_t\| < +\infty$ for some $\eta > 0$. Define a sequence of paths $\hat{\mu}^i$, $i\geq 1$, by $\hat{\mu}_0^i = \hat{\mu}_0$, and
\begin{align*}
\dot{\hat{\mu}}^i_t = \dot{\hat{\mu}}_t 1_{\{\|\dot{\hat{\mu}}_t\| \leq i\}} + \bar{\Lambda}^*_{\hat{\mu}^i_t,\hat{m}_t} \hat{\mu}^i_t 1_{\{\|\dot{\hat{\mu}}_t\| > i\}} , \, t \in [0,T].
\end{align*}
Since $I^*(\hat{\mu},\hat{\theta}) < +\infty$, by Lemma~\ref{lemma:istarfinite}, it follows that the mapping $t \mapsto \hat{\mu}_t$ is absolutely continuous and by the dominated convergence theorem one easily concludes that $\hat{\mu}^i_t \to \hat{\mu}_t$ as $i \to \infty$ uniformly in $t \in [0,T]$. Thus, by the assumption $\inf_{t \in [\delta,T]} \min_{x \in \X} \hat{\mu}_t(x) > 0$ for all $\delta > 0$, it follows that $\hat{\mu}^i \in \DX$ for all $i$ sufficiently large. Note that $(\hat{\mu}^i, \hat{\theta})$ satisfies the conditions of Lemma~\ref{lemma:fix-mhat} and hence $\tilde{I}(\hat{\mu}^i, \hat{\theta}) = I^*(\hat{\mu}^i, \hat{\theta})$ for all $i \geq 1$, that $\hat{\mu}^i \to \hat{\mu}$ in $\DX$ as $i \to \infty$, and that $\hat{\mu}^i_t = \hat{\mu}_t$ for all $t \in [0, \eta]$ for all sufficiently large $i$.

Let us now show that $I^*(\hat{\mu}^i, \hat{\theta}) \to I^*(\hat{\mu}^i, \hat{\theta}) $ as $i \to \infty$. By the arguments similar to those used in the proof of Lemma~\ref{lemma:fix-mhat}, using Berge's maximum theorem, for each $t \in [0,T]$, the mapping
\begin{align*}
u  \mapsto \sup_{g_t \in B(\Y)}-\int_\Y \left(L_ug_t(\cdot)(y) + \int_{\EY} \tau(Dg_t(y,\Delta)) \gamma_{y,y+d\Delta}(u) \right) \hat{m}_t(dy) 
\end{align*}
is continuous on $\mx$, and hence
\begin{align*}
\sup_{g_t \in B(\Y)} & -\int_\Y \left(L_{\hat{\mu}^i_t}g_t(\cdot)(y) + \int_{\EY} \tau(Dg_t(y,\Delta)) \gamma_{y,y+d\Delta}(\hat{\mu}^i_t) \right) \hat{m}_t(dy) \\
& \to \sup_{g_t \in B(\Y)}-\int_\Y \left(L_{\hat{\mu}_t}g_t(\cdot)(y) + \int_{\EY} \tau(Dg_t(y,\Delta)) \gamma_{y,y+d\Delta}(\hat{\mu}_t) \right) \hat{m}_t(dy) 
\end{align*}
as $i \to \infty$.
Therefore, by the bounded convergence theorem, we have
\begin{align*}
\int_{[0,T]}& \sup_{g_t \in B(\Y)} \left\{ -\int_\Y \left(L_{\hat{\mu}^i_t}g_t(\cdot)(y) + \int_{\EY} \tau(Dg_t(y,\Delta)) \gamma_{y,y+d\Delta}(\hat{\mu}^i_t) \right) \hat{m}_t(dy) \right\} dt \\
& \to \int_{[0,T]}\sup_{g_t \in B(\Y)} \left\{ -\int_\Y \left(L_{\hat{\mu}_t}g_t(\cdot)(y) + \int_{\EY} \tau(Dg_t(y,\Delta)) \gamma_{y,y+d\Delta}(\hat{\mu}_t) \right) \hat{m}_t(dy)\right\} dt
\end{align*}
as $i \to \infty$.

For the slow component, define
\begin{align*}
Z^i_t & \coloneqq  \sup_{\alpha_t \in \mathbb{R}^{|\X|}} \biggr\{ \langle \alpha_t, \dot{\hat{\mu}}^i_t  - \bar{\Lambda}^*_{\hat{\mu}^i_t, \hat{m}_t}\hat{\mu}^i_t \rangle\\
& \qquad -  \int_{\X \times \EX} \tau(D\alpha_t(x,\Delta))\bar{\lambda}_{x,x+d\Delta}(\hat{\mu}^i_t,\hat{m}_t) \hat{\mu}^i_t(dx) \biggr\}, \, t \in [0,T],
\end{align*}
and
\begin{align*}
Z_t & \coloneqq \sup_{\alpha_t \in \mathbb{R}^{|\X|}} \biggr\{ \langle \alpha_t, \dot{\hat{\mu}}_t - \bar{\Lambda}^*_{\hat{\mu}_t, \hat{m}_t}\hat{\mu}_t \rangle  \\
& \qquad - \int_{\X \times \EX} \tau(D\alpha_t(x,\Delta))\bar{\lambda}_{x,x+d\Delta}(\hat{\mu}_t,\hat{m}_t) \hat{\mu}_t(dx) \biggr\}, \, t \in [0,T].
\end{align*}
Since $I^*(\hat{\mu}, \hat{\theta}) < +\infty$ it follows that $Z_t < +\infty$ for almost all $t \in [0,T]$. Thanks to the assumption $\inf_{t \in [\delta,T]} \min_{x \in \X} \hat{\mu}_t(x) >0$ for all $\delta >0$, using the Berge's maximum theorem, for almost all $t \in [\eta,T]$, we see that the mapping
\begin{align*}
u \mapsto \sup_{\alpha_t \in \mathbb{R}^{|\X|}} \biggr\{ \langle \alpha_t, \dot{\hat{\mu}}_t  - \bar{\Lambda}^*_{u, \hat{m}_t}u \rangle -  \int_{\X \times \EX} \tau(D\alpha_t(x,\Delta))\bar{\lambda}_{x,x+d\Delta}(u,\hat{m}_t) u(dx) \biggr\}
\end{align*}
on $\mx$ is continuous at $\hat{\mu}_t$. Hence, noting that $Z^i_t = Z_t$ on $t \in [0,\eta]$ for all $i$ sufficiently large, for all $t \in [0,T] \cap \{s\in[0,T]:Z_s<+\infty\}$ we have $\dot{\hat{\mu}}^i_t = \dot{\hat{\mu}}_t$ for all $i$ sufficiently large, and $\hat{\mu}^i_t \to \hat{\mu}_t $ as $i \to \infty$ uniformly in $t \in [0,T]$, it follows that for all $t \in [0,T] \cap \{s \in [0,T]:Z_s < +\infty\}$ $Z^i_t \to Z_t$ as $i \to \infty$. Let us now show the convergence of the corresponding integrals. Fix $t \in (0,T]$ such that $Z_t < +\infty$ and let $\hat{\alpha}^i_t \in \mathbb{R}^{|\X|}$ and $\hat{\alpha}_t \in  \mathbb{R}^{|\X|} $ attain the supremum in the definition of $Z^i_t$ and $Z_t$ respectively.
Whenever $ \| \dot{\hat{\mu}}^i_t \| \leq i$, we have,
\begin{align}
0 \leq Z^i_t & = \langle \hat{\alpha}^i_t, \dot{\hat{\mu}}^i_t \rangle  - \sum_{(x, x^\prime) \in \EX} (\exp\{\hat{\alpha}^i_t(x^\prime) - \hat{\alpha}^i_t(x)\} -1) \bar{\lambda}_{x,x^\prime}(\hat{\mu}^i_t,\hat{m}_t) \hat{\mu}^i_t(x) \nonumber \\
& = \langle \hat{\alpha}^i_t, \dot{\hat{\mu}}_t \rangle  - \sum_{(x,x^\prime) \in \EX} (\exp\{\hat{\alpha}^i_t(x^\prime) - \hat{\alpha}^i_t(x)\} -1) \bar{\lambda}_{x,x^\prime}(\hat{\mu}_t,\hat{m}_t) \hat{\mu}_t(x) \nonumber \\
& \qquad -  \sum_{(x,x^\prime) \in \EX} (\exp\{\hat{\alpha}^i_t(x^\prime) - \hat{\alpha}^i_t(x)\} -1) \times (\bar{\lambda}_{x,x^\prime}(\hat{\mu}^i_t,\hat{m}_t) \hat{\mu}^i_t(x) - \bar{\lambda}_{x,x^\prime}(\hat{\mu}_t,\hat{m}_t) \hat{\mu}_t(x)) \nonumber \\
& \leq Z_t - \sum_{(x,x^\prime) \in \EX} (\exp\{\hat{\alpha}^i_t(x^\prime) - \hat{\alpha}^i_t(x)\} -1) \times (\bar{\lambda}_{x,x^\prime}(\hat{\mu}^i_t,\hat{m}_t) \hat{\mu}^i_t(x)- \bar{\lambda}_{x,x^\prime}(\hat{\mu}_t,\hat{m}_t) \hat{\mu}_t(x)).
\label{eqn:zt-domination}
\end{align}
Since $\hat{\mu}^i_t = \hat{\mu}_t$, $t \in [0, \eta]$, for all large enough $i$, the second term above vanishes whenever $t \in [0,\eta]$. Since $\hat{\mu}^i_t \to \hat{\mu}_t$ as $i \to \infty$ uniformly in $t \in [0,T]$, the first order optimality condition for $(\hat{\alpha}^i_t(x), x \in \X)$ (see~(\ref{eqn:alphahat-1})) implies that, for some constants $c_{\eta} > 0$, we have
\begin{align*}
\max_{(x,x^\prime) \in \EX} \exp\{\hat{\alpha}^i_t(x^\prime) - \hat{\alpha}^i_t(x)\} \leq c_{\eta} ( 1+\| \dot{\hat{\mu}}_t\|)
\end{align*}
whenever $t \in [\eta, T] \cap \{s \in [0,T]: Z_s<+\infty\}$. In particular, the right hand side of~(\ref{eqn:zt-domination}) is integrable. Hence, noting that $Z^i_t = 0$ in the alternative case when $\|\dot{\hat{\mu}}_t \| > i$, by an application of the dominated convergence theorem, we have that
\begin{align*}
\int_{[0,T]} \sup_{\alpha_t \in \mathbb{R}^{|\X|}} \biggr\{ \langle \alpha_t, \dot{\hat{\mu}}^i_t  - \bar{\Lambda}^*_{\hat{\mu}^i_t, \hat{m}_t}\hat{\mu}^i_t \rangle - \int_{\X \times \EX} \tau(D\alpha_t(x,\Delta))\bar{\lambda}_{x,x+d\Delta}(\hat{\mu}^i_t,\hat{m}_t) \hat{\mu}^i_t(dx) \biggr\} dt\
\end{align*}
converges to 
\begin{align*}
\int_{[0,T]} \sup_{\alpha_t \in \mathbb{R}^{|\X|}} \biggr\{ \langle \alpha_t, \dot{\hat{\mu}}_t  - \bar{\Lambda}^*_{\hat{\mu}_t, \hat{m}_t}\hat{\mu}_t \rangle - \int_{\X \times \EX} \tau(D\alpha_t(x,\Delta))\bar{\lambda}_{x,x+d\Delta}(\hat{\mu}_t,\hat{m}_t) \hat{\mu}_t(dx) \biggr\} dt\
\end{align*}
as $i \to \infty$. Hence, combining the convergences for the slow and the fast components, we have $I^*(\hat{\mu}^i, \hat{\theta}) \to I^*(\hat{\mu}, \hat{\theta})$ as $i \to \infty$. Further, by Remark~\ref{remark:itilde=istar}, it follows that $\tilde{I}(\hat{\mu}, \hat{\theta}) = I^*(\hat{\mu}, \hat{\theta})$.

In the general case when the mapping $t \mapsto \hat{\mu}_t$ is not locally Lipschitz continuous at $t = 0$, using arguments similar to those used in the proof of Lemma~\ref{lemma:fix-initial-condition}, one constructs a sequence $\hat{\tau}^i$, $i \geq 1$, and a sequence of elements $(\hat{\mu}^i,\hat{\theta}^i) \in \DX \times \DY$, $i \geq 1$, such that $\hat{\tau}^i \to 0$ as $i \to \infty$, $\sup_{t \in [0, \hat{\tau}^i]} \|\dot{\hat{\mu}}^i_t\| < +\infty$ (therefore the mapping $t \mapsto \hat{\mu}^i_t$ is locally Lipschitz continuous at $t=0$), $(\hat{\mu}^i,\hat{\theta}^i) \to (\hat{\mu},\hat{\theta})$ in $\DX \times \DY$ as $i \to \infty$, $\hat{\mu}^i_t = \hat{\mu}_{t+1/i-\hat{\tau}^i}$ and $\hat{m}^i_t = \hat{m}_{t+1/i-\hat{\tau}^i}$ for all $t \in [\hat{\tau}^i, T]$, and
\begin{align*}
\int_{[0, \hat{\tau}^i]\cup[T+1/i-\hat{\tau}^i,T]} \sup_{\alpha_t \in \mathbb{R}^{|\X|}} \biggr\{ \langle \alpha_t, \dot{\hat{\mu}}^i_t  - \bar{\Lambda}^*_{\hat{\mu}^i_t, \hat{m}^i_t}\hat{\mu}^i_t \rangle - \int_{\X \times \EX} \tau(D\alpha_t(x,\Delta))\bar{\lambda}_{x,x+d\Delta}(\hat{\mu}^i_t,\hat{m}^i_t) \hat{\mu}^i_t(dx) \biggr\} dt \\
\qquad + \int_{[0,\hat{\tau}^i]\cup[T+1/i-\hat{\tau}^i,T] } \sup_{g_t \in B(\Y)} \biggr\{ -\int_\Y \left(L_{\hat{\mu}^i_t}g_t(\cdot)(y) + \int_{\EY} \tau(Dg_t(y,\Delta)) \gamma_{y,y+d\Delta}(\hat{\mu}^i_t) \right) \hat{m}^i_t(dy) \biggr\} dt
\end{align*}
converges to $0$ as $i \to \infty$ (by using the small cost construction of constant velocity paths). Based on what we have already shown for paths that are locally Lipschitz continuous at $t = 0$, we see that $\tilde{I}(\hat{\mu}^i, \hat{\theta}^i) = I^*(\hat{\mu}^i, \hat{\theta}^i)$ for all $i \geq 1$. Again, using arguments similar to those used in the proof of Lemma~\ref{lemma:fix-initial-condition}, we conclude that $I^*(\hat{\mu}^i, \hat{\theta}^i) \to I^*(\hat{\mu}, \hat{\theta})$ as $i \to \infty$. Once again, by Remark~\ref{remark:itilde=istar}, we have $\tilde{I}(\hat{\mu},\hat{\theta}) = I^*(\hat{\mu},\hat{\theta})$. This completes the proof of the lemma.
\end{proof}

We finally show that $\tilde{I}(\mu, \theta) = I^*(\mu, \theta) $ for all $(\mu, \theta) \in \DX \times \DY$, by allowing the path $\mu$ to hit the boundary of $\mx$.
\begin{theorem}
Let $\nu \in \mx$ and let $\tilde{I} : \DX \times \DY \to [0, + \infty]$ be a subsequential rate function such that $\tilde{I}(\mu, \theta) = +\infty$ unless $\mu_0 = \nu$. Then, for all $(\hat{\mu},\hat{\theta}) \in \DX \times \DY$, we have $\tilde{I}(\hat{\mu},\hat{\theta}) = I^*(\hat{\mu}, \hat{\theta})$.
	\label{thm:istar=itilde}
	\end{theorem}
\begin{proof}
Since $\tilde{I}(\mu, \theta) \geq I^*(\mu, \theta)$ for all $(\mu, \theta) \in \DX \times \DY$, it suffices to focus on a $(\hat{\mu},\hat{\theta}) \in \DX \times \DY$ such that $I^*(\hat{\mu}, \hat{\theta}) < +\infty$ and $\hat{\mu}_0 = \nu$. By Lemma~\ref{lemma:istarfinite}, we have that the mapping $[0,T] \ni t \mapsto \hat{\mu}_t \in \mx$ is absolutely continuous. In particular, $\dot{\hat{\mu}}_t$ exists for almost all $t \in [0,T]$ and $\hat{\mu}_t = \nu + \int_{[0,t]} \dot{\hat{\mu}}_s ds$ for all $t \in [0,T]$.

We shall construct a sequence of paths $\hat{\mu}^i \in \DX$, $i \geq 1$, such that $\hat{\mu}^i \to \hat{\mu}$ in $\DX$ as $i \to \infty$, $\tilde{I}(\hat{\mu}^i, \hat{\theta}) = I^*(\hat{\mu}, \hat{\theta})$ for all $i \geq 1$, and $I^*(\hat{\mu}^i, \hat{\theta}) \to I^*(\hat{\mu}, \hat{\theta})$ as $i \to \infty$.

Let $\varepsilon_i(x) = \frac{\hat{\mu}_{1/i}(x)+1/i}{1+|\X|/i}$, $x \in \X$ and $i \geq 1$.  Using arguments similar to those used in the proof of Lemma~\ref{lemma:fix-initial-condition}, we first construct a sequence of times $\hat{\tau}^i$, $i \geq 1$, and a sequence of piecewise constant velocity trajectories $\hat{\mu}^i_t$, $t \in [0,\hat{\tau}^i]$, with the property that $\hat{\mu}^i_0 = \hat{\mu}_0$ for all $i\geq 1$,  $\hat{\mu}^i_{\hat{\tau}^i}(x) = \varepsilon_i(x)$ for all $x \in \X$ and $i \geq 1$,  $\hat{\tau}^i \to 0$ as $i \to \infty$, and
\begin{align}
\int_{[0,\hat{\tau}^i]}  \sup_{\alpha_t \in \mathbb{R}^{|\X|}} \biggr\{ \langle \alpha_t, \dot{\hat{\mu}}^i_t  - \bar{\Lambda}^*_{\hat{\mu}^i_t, \hat{m}_t}\hat{\mu}^i_t \rangle - \int_{\X \times \EX} \tau(D\alpha_t(x,\Delta))\bar{\lambda}_{x,x+d\Delta}(\hat{\mu}^i_t,\hat{m}_t) \hat{\mu}^i_t(dx) \biggr\} dt \to 0
\label{eqn:conv-over-taui}
\end{align}
as $i \to \infty$. We then define the path $\hat{\mu}^i_t$ on $t \in (\hat{\tau}^i, T]$ by
\begin{align*}
\hat{\mu}^i_t(x)=  \frac{\hat{\mu}_{t+1/i-\hat{\tau}^i}(x) + 1/i}{1+|\X|/i}, x \in \X. 
\end{align*}
Clearly, $\hat{\mu}^i_t \to \hat{\mu}_t$ as $i \to \infty$ uniformly in $t \in [0,T]$ and hence $\hat{\mu}^i \to \hat{\mu}$ in $\DX$ as $i \to \infty$. Note that $(\hat{\mu}^i, \hat{\theta})$ satisfies the conditions of Lemma~\ref{lemma:fix-muhat-lipschitz} and hence we have $\tilde{I}(\hat{\mu}^i, \hat{\theta}) = I^*(\hat{\mu}^i, \hat{\theta})$ for all $i \geq 1$. 

We now show that $I^*(\hat{\mu}^i,\hat{\theta}) \to I^*(\hat{\mu},\hat{\theta})$ as $i \to \infty$. Using arguments similar to those used in the proof of Lemma~\ref{lemma:fix-muhat-lipschitz}, it is easy to show that
\begin{align}
\int_{[0,T]}& \sup_{g_t \in B(\Y)} \left\{ -\int_\Y \left(L_{\hat{\mu}^i_t}g_t(\cdot)(y) + \int_{\EY} \tau(Dg_t(y,\Delta)) \gamma_{y,y+d\Delta}(\hat{\mu}^i_t) \right) \hat{m}_t(dy) \right\} dt \nonumber \\
& \to \int_{[0,T]}\sup_{g_t \in B(\Y)} \left\{ -\int_\Y \left(L_{\hat{\mu}_t}g_t(\cdot)(y) + \int_{\EY} \tau(Dg_t(y,\Delta)) \gamma_{y,y+d\Delta}(\hat{\mu}_t) \right) \hat{m}_t(dy)\right\} dt
\label{eqn:conv-mhat-part}
\end{align}
as $i \to \infty$.

To show convergence of the integral corresponding to the slow process, define
\begin{align*}
Z^i_t & \coloneqq  \sup_{\alpha_t \in \mathbb{R}^{|\X|}} \biggr\{ \langle \alpha_t, \dot{\hat{\mu}}^i_{t-1/i+\hat{\tau}^i}  - \bar{\Lambda}^*_{\hat{\mu}^i_{t-1/i+\hat{\tau}^i}, \hat{m}_t}\hat{\mu}^i_{t-1/i+\hat{\tau}^i} \rangle\\
& \qquad -  \int_{\X \times \EX} \tau(D\alpha_t(x,\Delta))\bar{\lambda}_{x,x+d\Delta}(\hat{\mu}^i_{t-1/i+\hat{\tau}^i},\hat{m}_t) \hat{\mu}^i_{t-1/i+\hat{\tau}^i}(dx) \biggr\}, \, t \in [1/i,T+1/i-\hat{\tau}^i],
\end{align*}
and
\begin{align*}
Z_t & \coloneqq  \sup_{\alpha_t \in \mathbb{R}^{|\X|}} \biggr\{ \langle \alpha_t, \dot{\hat{\mu}}_t - \bar{\Lambda}^*_{\hat{\mu}_t, \hat{m}_t}\hat{\mu}_t \rangle  \\
& \qquad - \int_{\X \times \EX} \tau(D\alpha_t(x,\Delta))\bar{\lambda}_{x,x+d\Delta}(\hat{\mu}_t,\hat{m}_t) \hat{\mu}_t(dx) \biggr\}, \, t \in [0,T].
\end{align*}
Note the shift in the time index in the definition of $Z_t^i$ to enable direct comparison between $Z_t$ and $Z_t^i$. For $t \in [1/i,T]$, we then have
\begin{align*}
Z^i_t = & \frac{1}{1+|\X|/i}  \sup_{\alpha_t \in \mathbb{R}^{|\X|}}\biggr\{ \langle \alpha_t, \dot{\hat{\mu}}_t \rangle - \sum_{(x,x^\prime) \in \EX}  (\exp\{\alpha_t(x^\prime)- \alpha_t(x)\} -1)\bar{\lambda}_{x,x^\prime}(\hat{\mu}^i_{t-1/i+\hat{\tau}^i},\hat{m}_t) (\hat{\mu}_t(x)+1/i) \biggr\}.
\end{align*}
The objective function above can be simplified as
\begin{align*}
 & \langle \alpha_t, \dot{\hat{\mu}}_t \rangle - \sum_{(x,x^\prime) \in \EX}  \exp\{\alpha_t(x^\prime)- \alpha_t(x)\} \bar{\lambda}_{x,x^\prime}(\hat{\mu}^i_{t-1/i+\hat{\tau}^i},\hat{m}_t) (\hat{\mu}_t(x)+1/i)  \\
& =  \langle \alpha_t, \dot{\hat{\mu}}_t \rangle - \sum_{(x,x^\prime) \in \EX}  \exp\{\alpha_t(x^\prime)- \alpha_t(x)\}  \bar{\lambda}_{x,x^\prime}(\hat{\mu}_t,\hat{m}_t) \hat{\mu}_t(x)  \\
& \qquad - \sum_{(x,x^\prime) \in \EX}  \exp\{\alpha_t(x^\prime)- \alpha_t(x)\} \biggr[ (\bar{\lambda}_{x,x^\prime}(\hat{\mu}^i_{t-1/i+\hat{\tau}^i},\hat{m}_t) - \bar{\lambda}_{x,x^\prime}(\hat{\mu}_t,\hat{m}_t)) \hat{\mu}_t(x) + \frac{\bar{\lambda}_{x,x^\prime}(\hat{\mu}^i_{t-1/i+\hat{\tau}^i}, \hat{m}_t)}{i} \biggr]\\
& \leq \langle \alpha_t, \dot{\hat{\mu}}_t \rangle - \sum_{(x,x^\prime) \in \EX}  \exp\{\alpha_t(x^\prime)- \alpha_t(x)\}  \bar{\lambda}_{x,x^\prime}(\hat{\mu}_t,\hat{m}_t) \hat{\mu}_t(x)  \\
& \qquad   - \sum_{(x,x^\prime) \in \EX}  \exp\{\alpha_t(x^\prime)- \alpha_t(x)\} \biggr(- \frac{c_L\hat{\mu}_t(x)}{i} + \frac{c}{i} \biggr)
\end{align*}
where the last inequality follows from assumption~\ref{assm:a2}; here $c = \min_{(x,x^\prime) \in \X} \min_{y \in \Y}  \min_{\xi \in \mx} \lambda_{x,x^\prime}(\xi, y)$ and $c_L = \max_{(x,x^\prime) \in \EX} \max_{y \in \Y} c_L^{x,x^\prime,y}$ where $c_L^{x,x^\prime, y}$ is the Lipschitz constant of $\lambda_{x,x^\prime}(\cdot, y), (x,x^\prime) \in \EX, y \in \Y$. Fix $t \in [1/i, T+1/i-\hat{\tau}^i]$ with $Z_t < +\infty$ and let $(\hat{\alpha}^i_t(x), x \in \X) \in \mathbb{R}^{|\X|}$ denote the optimiser in the definition of $Z^i_t$. Then the above computation gives us
\begin{align*}
Z^i_t & \leq \frac{1}{1+|\X|/i} \left\{ \langle \hat{\alpha}^i_t, \dot{\hat{\mu}}_t \rangle - \sum_{(x,x^\prime) \in \EX}  \exp\{\hat{\alpha}^i_t(x^\prime)- \hat{\alpha}^i_t(x)\}  \bar{\lambda}_{x,x^\prime}(\hat{\mu}_t,\hat{m}_t) \hat{\mu}_t(x)  \right\}\\
& \qquad - \frac{1}{1+|\X|/i} \left\{ \sum_{(x,x^\prime) \in \EX}  \exp\{\hat{\alpha}^i_t(x^\prime)- \hat{\alpha}^i_t(x)\}\biggr(- \frac{c_L\hat{\mu}_t(x)}{i} + \frac{c}{i} \biggr) \right\} \\
& \leq \frac{1}{1+|\X|/i} Z_t - \frac{1}{1+|\X|/i} \left\{ \sum_{(x,x^\prime) \in \EX}  \exp\{\hat{\alpha}^i_t(x^\prime)- \hat{\alpha}^i_t(x)\} \biggr(- \frac{c_L\hat{\mu}_t(x)}{i} + \frac{c}{i} \biggr) \right\}.
\end{align*}
If $\hat{\mu}_t(x) < c/c_L$ for some $x \in \X$, we see that all the terms in the summation corresponding to the edges $(x,x^\prime) \in \EX$ are negative. On the other hand, if $\hat{\mu}_t(x) > c/c_L$, noting that $\hat{\tau}^i \to 0$ as $i \to \infty$ and the convergence of $\hat{\mu}^i_t$  to $\hat{\mu}_t$ as $i \to \infty$  uniformly in $t \in [0,T]$, the first order optimality condition for $(\hat{\alpha}^i_t(x), x \in \X)$ implies that, for some constant $c_2 > 0$, 
\begin{align*}
\max_{x^\prime \in \X : (x,x^\prime) \in \EX}\exp\{\hat{\alpha}^i_t(x^\prime) - \hat{\alpha}^i_t(x)\} \leq c_2 (1+\|\dot{\hat{\mu}}_t\|),
\end{align*}
and hence for all $t \in [1/i, T+1/i-\hat{\tau}^i]$ with $Z_t < +\infty$, we obtain that
\begin{align*}
Z^i_t \leq \frac{1}{1+|\X|/i} \{Z_t + c_2 |\EX| (1+\| \dot{\hat{\mu_t}}\|) \}.
\end{align*}
Hence by the dominated convergence theorem, we see that
\begin{align*}
\int_{[0,T]}\sup_{\alpha_t \in \mathbb{R}^{|\X|}} & \biggr\{ \langle \alpha_t, \dot{\hat{\mu}}^i_t  - \bar{\Lambda}^*_{\hat{\mu}^i_t, \hat{m}_t}\hat{\mu}^i_t \rangle-  \int_{\X \times \EX} \tau(D\alpha_t(x,\Delta))\bar{\lambda}_{x,x+d\Delta}(\hat{\mu}^i_t,\hat{m}_t) \hat{\mu}^i_t(dx) \biggr\} \times 1_{\{t \geq \hat{\tau}^i\}} dt 
\end{align*}
converges to
\begin{align*}
\int_{[0,T]} \sup_{\alpha_t \in \mathbb{R}^{|\X|}} \biggr\{ \langle \alpha_t, \dot{\hat{\mu}}_t - \bar{\Lambda}^*_{\hat{\mu}_t, \hat{m}_t}\hat{\mu}_t \rangle  - \int_{\X \times \EX} \tau(D\alpha_t(x,\Delta))\bar{\lambda}_{x,x+d\Delta}(\hat{\mu}_t,\hat{m}_t) \hat{\mu}_t(dx) \biggr\}dt
\end{align*}
as $i \to \infty$. This along with the convergences~(\ref{eqn:conv-over-taui}) and~(\ref{eqn:conv-mhat-part}) implies that $I^*(\hat{\mu}^i,\hat{\theta})\to I^*(\hat{\mu}, \hat{\theta})$ as $i \to \infty$. The procedure of Remark~\ref{remark:itilde=istar} then completes the proof of the theorem.
\end{proof}
	
\section{Completing the Proof of Theorem~\ref{thm:main-finite-duration}}
\label{section:complete-proof}
	We finally complete the proof of Theorem~\ref{thm:main-finite-duration} by extending the conclusion of Theorem~\ref{thm:istar=itilde} to all subsequential rate functions $\tilde{I}$, i.e. we remove the  restriction that, for some $\nu \in \mx$, $\tilde{I}(\mu, \theta) = +\infty$ unless $\mu_0 = \nu$.

	\begin{proof}[Proof of Theorem~\ref{thm:main-finite-duration}]
Fix $\nu \in \mx$ and suppose that $\{\mu_N\}$ is such that $\limsup_{N \to \infty}\frac{1}{N}\log \Prob(|\mu_N(0)-\nu| \geq \varepsilon) = -\infty$ for each $\varepsilon > 0$. By Theorem~\ref{thm:exp-tightness}, the family $\{(\mu_N,\theta_N)\}_{N \geq 1}$ is exponentially tight in $\DX \times \DY$. Therefore, there exists a subsequence $\{N_k\}_{k \geq 1}$ of $\mathbb{N}$ such that $\{(\mu_{N_k},\theta_{N_k})\}_{k \geq 1}$ satisfies the LDP with  rate function $\tilde{I}$ (see, for example, Dembo and Zeitouni~\cite[Lemma~4.1.23]{dembo-zeitouni}); by the above condition on the family $\{\mu_N\}$ and by the contraction principle, we see that $\tilde{I}(\mu, \theta) = +\infty$ unless $\mu_0 =\nu$. Therefore, by Theorem~\ref{thm:istar=itilde}, $\tilde{I} = I^*$ on $\DX \times \DY$. Hence $\tilde{I}$ is uniquely determined for all such subsequences, and it follows that the family $\{(\mu_N,\theta_N)\}_{N \geq 1}$ satisfies the LDP with rate function $I^*$ (see, for example, Dembo and Zeitouni~\cite[Exercise~4.4.15 (b)]{dembo-zeitouni}) defined as follows: $I^*(\mu, \theta)$ is defined by~(\ref{eqn:istar}) whenever $\mu$ is such that $\mu(0) = \nu$, and $I^*(\mu,\theta) = + \infty$ otherwise.

In the general case when $\{\mu_N(0)\}$ satisfies the LDP on $\mx$ with rate function $I_0$, let $p^{(N)}_{\nu_N}$ denote the regular conditional distribution of $(\mu_N, \theta_N)$ on $\DX \times \DY$ given $\mu_N(0) = \nu_N \in M_1^N(\X)$. By the above argument, whenever $\nu_N \to \nu$ in $\mx$, $p^{(N)}_{\nu_N}$ satisfies the LDP on $\DX \times \DY$ with rate function $I^*(\mu, \theta)+ \infty1_{\mu(0) \neq \nu}$. Therefore, the family $\{(\mu_N, \theta_N)\}_{N \geq 1}$ satisfies the LDP on $\DX \times \DY $ with rate function $I_0(\mu(0))+ I^*(\mu, \theta)$ (see, for example, Chaganty~\cite{chaganty-97}). This completes the proof of Theorem~\ref{thm:main-finite-duration}.
\end{proof}

\appendix

\section{Examples of two time scale mean-field models}
\label{appendix:examples}
We describe two applications that can be studied using our two time scale mean-field model --  a retrial system with orbit queues and a wireless local area network (WLAN) with local interaction.

\emph{Example 1}. We first describe a retrial system with orbit queues (see Figure~\ref{fig:retrial}). Such systems have been used to model multiple competing jobs in a carrier sense multiple access network (see Avrachenkov et al.~\cite{avrachenkov-etal-14} and the references therein). In this model, there is a single exponential server with service rate $N$, $N$ statistically identical Poisson arrival streams (of rate $\lambda$) and $N$ orbit queues of identical (finite) size $K$, one corresponding to each arrival stream. Whenever an arriving customer finds an empty server, it occupies the server and spends a random amount of time, exponentially distributed with mean $1/N$, and then leaves the system. If the arriving customer sees a busy server, it waits in the orbit queue corresponding to that arrival stream, if the queue is not full. Whenever an orbit queue is nonempty and the server is free, the head of the line customer in that orbit queue attempts for service at a fixed positive rate $\alpha$. In this setting, the state of the server (i.e. idle/busy) represents the environment, and the number of waiting customers in an orbit queue represents the state of that node. Note that the state of each orbit queue evolves slowly (i.e. $O(1)$ many transitions in a given $O(1)$ duration of time). But since there are $N$ orbit queues and each nonempty queue attempts for service with a fixed positive rate, the environment makes $O(N)$ many transitions in a given $O(1)$ duration of time. Also, the transition rates of the number of customers in a queue depend on the state of the server and the transition rates of the environment depend on the fraction of non-empty orbit queues. Figure~\ref{fig:transition_rates_node} depicts the transition rates of each orbit queue when the server state is $y$ ($y = 0$ indicates idle state and $y=1$ indicates busy state), and Figure~\ref{fig:transition_rates_server} depicts the transition rates of the server when the empirical measure of the states of all the orbit queues is $\xi$. Clearly, this system falls within the framework of our two time scale mean-field model.

\setlength{\unitlength}{1cm}
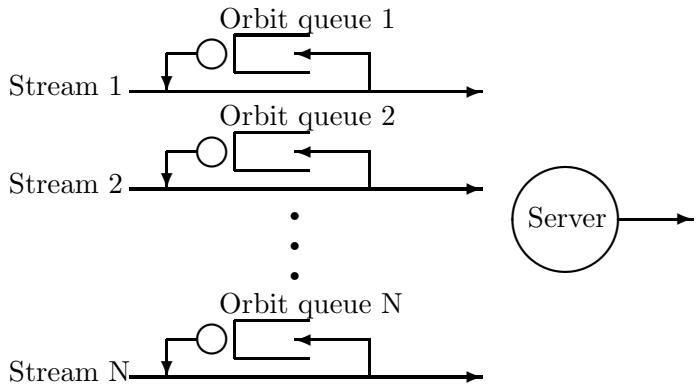
\begin{figure}[h!]
\begin{picture}(7,7)(-4,1)
\thicklines
\put(6.6,4.1){\circle{5}}
\put(6.1,4){Server}
\put(7.3,4.1){\vector(1,0){1}}

\put(0.8,5.8){\vector(1,0){4.7}}
\put(4,6.3){\vector(-1,0){1}}
\put(4,5.8){\line(0,1){0.5}}
\put(3.2,6.55){\line(-1,0){1}}
\put(3.2,6.05){\line(-1,0){1}}
\put(2.2,6.05){\line(0,1){0.5}}
\put(1.9,6.3){\circle{0.4}}
\put(1.7,6.3){\line(-1,0){0.4}}
\put(1.3,6.3){\vector(0,-1){0.5}}

\put(0.8,4.5){\vector(1,0){4.7}}
\put(4,5){\vector(-1,0){1}}
\put(4,4.5){\line(0,1){0.5}}
\put(3.2,5.25){\line(-1,0){1}}
\put(3.2,4.75){\line(-1,0){1}}
\put(2.2,4.75){\line(0,1){0.5}}
\put(1.9,5){\circle{0.4}}
\put(1.7,5){\line(-1,0){0.4}}
\put(1.3,5){\vector(0,-1){0.5}}

\put(0.8,2){\vector(1,0){4.7}}
\put(4,2.5){\vector(-1,0){1}}
\put(4,2){\line(0,1){0.5}}
\put(3.2,2.75){\line(-1,0){1}}
\put(3.2,2.25){\line(-1,0){1}}
\put(2.2,2.25){\line(0,1){0.5}}
\put(1.9,2.5){\circle{0.4}}
\put(1.7,2.5){\line(-1,0){0.4}}
\put(1.3,2.5){\vector(0,-1){0.5}}

\put(-0.8,5.75){Stream 1}
\put(-0.8,4.45){Stream 2}
\put(-0.8,1.95){Stream N}

\put(2,6.65){Orbit queue 1}
\put(2,5.35){Orbit queue 2}
\put(2,2.85){Orbit queue N}

\put(3,4.15){\circle*{0.1}}
\put(3,3.75){\circle*{0.1}}
\put(3,3.35){\circle*{0.1}}
\end{picture}
\caption{A retrial system with $N$ orbit queues}
\label{fig:retrial}
\end{figure}

\setlength{\unitlength}{1mm}
\begin{figure}[h]
\centering
\begin{picture}(50,30)(5,0)
\thicklines
\put(-20,15){\circle{9}}
\put(0,15){\circle{9}}
\put(20,15){\circle{9}}
\put(80,15){\circle{9}}
\put(-21,14){$0$}
\put(-1,14){$1$}
\put(19,14){$2$}
\put(79,14){$K$}

\qbezier(-17,18)(-10,25)(-3,18)
\put(-8,21.5){\vector(1,0){0.1}}
\qbezier(3,18)(10,25)(17,18)
\put(12,21.5){\vector(1,0){0.1}}
\qbezier(23,18)(30,25)(37,18)
\put(32,21.5){\vector(1,0){0.1}}
\qbezier(63,18)(70,25)(77,18)
\put(72,21.5){\vector(1,0){0.1}}
\qbezier(-17,12)(-10,5)(-3,12)
\put(-11,8.5){\vector(-1,0){0.1}}
\qbezier(3,12)(10,5)(17,12)
\put(9,8.5){\vector(-1,0){0.1}}
\qbezier(23,12)(30,5)(37,12)
\put(29,8.5){\vector(-1,0){0.1}}
\qbezier(63,12)(70,5)(77,12)
\put(69,8.5){\vector(-1,0){0.1}}
\put(42,15){\circle*{1}}
\put(50,15){\circle*{1}}
\put(58,15){\circle*{1}}
\put(-11,23){$\lambda y$}
\put(9,23){$\lambda y$}
\put(29,23){$\lambda y$}
\put(69,23){$\lambda y$}
\put(-15,5){$\alpha(1-y)$}
\put(4,5){$\alpha(1-y)$}
\put(25,5){$\alpha(1-y)$}
\put(65,5){$\alpha(1-y)$}
\end{picture}
\caption{Transition rates of an orbit queue when the server state is $y$}
\label{fig:transition_rates_node}
\end{figure}
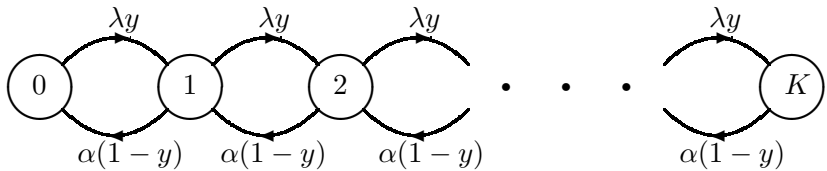

\setlength{\unitlength}{1mm}
\begin{figure}
\centering
\begin{picture}(50,25)(-30,5)
\thicklines
\put(-20,15){\circle{9}}
\put(0,15){\circle{9}}
\put(-21,16){$0$}
\put(-22,13){idle}
\put(-1,16){$1$}
\put(-3,13){busy}

\qbezier(-17,18)(-10,25)(-3,18)
\put(-8,21.5){\vector(1,0){0.1}}
\qbezier(-17,12)(-10,5)(-3,12)
\put(-11,8.5){\vector(-1,0){0.1}}
\put(-22,23){$N(\lambda + \alpha(1-\xi(0)))$}
\put(-11,5){$N$}
\end{picture}
\caption{Transition rates of the server when the empirical measure of nodes is $\xi$; $\xi(0)$ denotes the fraction of empty orbit queues}
\label{fig:transition_rates_server}
\end{figure}
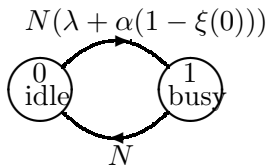

\emph{Example 2}. We now describe the setting of WLAN. Let there be $N$ nodes. Time is divided into slots. Each node has a state associated with it, which represents the probability of attempting a packet transmission in a slot. Since the network could be spread over a large geographical area, the nodes are grouped into $C$ classes; every node that belongs to a class can hear the transmissions of every other node in that class. Figure~\ref{fig:wlan} depicts an example network with $7$ nodes and $3$ classes. The interaction among the nodes comes from the distributed channel access algorithm executed by the nodes. This interaction results in the evolution of the state of each node in the following fashion: a node that incurs a collision upon a packet transmission moves to a different state with a reduced probability of attempt, and upon a successful transmission moves to another state with an increased probability of attempt. Since multiple nodes could transmit at the same slot, the channel corresponding to a class of nodes could be in three different states in a given time slot: (i) an idle slot (denoted by state $0$), (ii) a collision (state $2$) or (iii) a successful packet transmission (state $1$). We denote the channel state corresponding to each class of nodes as the environment, i.e., at each time slot, the environment is an element of $\{0,1,2\}^{C}$ with the $c$th coordinate representing the channel state of the $c$th class of nodes. Since there are $O(N)$ many nodes in each class, we see that the environment makes $O(N)$ many transitions over a given $O(1)$ time duration. Also, we see that the transition rates of the environment depend on the attempt probabilities of the nodes in that class, but only through the empirical measure of the states of the nodes in that class. On the other hand, the transition rates of the states of a node depend on the attempt probabilities of the nodes in that class (again, only through the empirical measure) as well as the environment. Hence, we have a two time scale mean-field model that describes the network, but one that operates in discrete-time. We now see how to translate this to an approximate continuous-time model.

\setlength{\unitlength}{1cm}
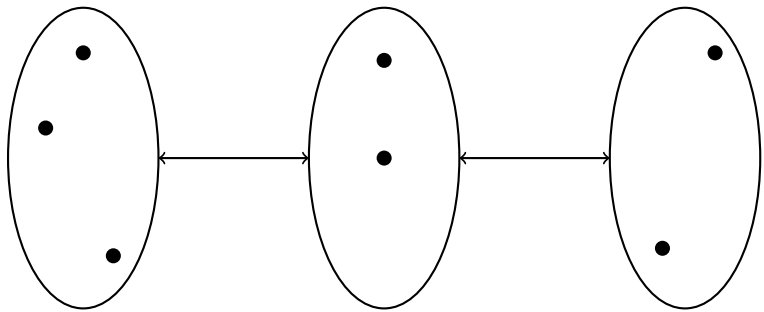
\begin{figure}
\centering
\begin{tikzpicture}
\thicklines
\draw[thick](2,0) ellipse (1 and 2);
\draw[thick](-2,0) ellipse (1 and 2);
\draw[thick](6,0) ellipse (1 and 2);
\fill (-2,1.4) circle(.1);
\fill (-1.6,-1.3) circle(.1);
\fill (-2.5,0.4) circle(.1);
\fill (2,1.3) circle(.1);
\fill (2,0) circle(.1);
\fill (6.4,1.4) circle(.1);
\fill (5.7,-1.2) circle(.1);
\draw[<->, thick] (-1,0) -- (1,0);
\draw[<->, thick] (3,0) -- (5,0);
\end{tikzpicture}
\caption{A wireless local area network with $3$ classes and $7$ users; interference among classes are indicated by arrows}
\label{fig:wlan}
\end{figure}

\setlength{\unitlength}{1mm}
\begin{figure}
\centering
\begin{picture}(50,16)(0,2)
\thicklines
\put(-20,15){\circle{9}}
\put(0,15){\circle{9}}
\put(20,15){\circle{9}}
\put(56,15){\circle{9}}
\put(-21,13.5){$0$}
\put(-1,13.5){$1$}
\put(19,13.5){$2$}
\put(55,13.5){$K$}
\put(-16,15){\vector(1,0){12}}
\put(4,15){\vector(1,0){12}}
\put(24,15){\vector(1,0){6}}
\qbezier(0,11)(-10,5)(-20,11)
\qbezier(20,11)(-5,0)(-20,11)
\qbezier(56,11)(0,-10)(-20,11)
\put(-11,8){\vector(-1,0){0.1}}
\put(-3,5.5){\vector(-1,0){0.1}}
\put(9,0.2){\vector(-1,0){0.1}}
\put(33,15){\circle*{1}}
\put(41,15){\circle*{1}}
\put(49,15){\circle*{1}}
\end{picture}
\caption{Set of allowed transitions for a particle in a WLAN}
\label{fig:transition_rates_node-wlan}
\end{figure}
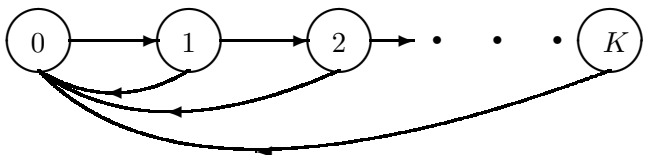

Figure~\ref{fig:transition_rates_node-wlan} depicts the set of allowed transitions of a node; in typical WLAN implementations, the most aggressive state is $0$ and the least aggressive state is $K$. A node moves from state $i$ to state $i+1$ when it incurs a collision, and moves from state $i$ to $0$ when a packet is successfully transmitted. To describe the transition rates of the continuous time model, we shall consider a scaled version of the above discrete time model where each time slot is of duration $1/N$. Let $p_i/N$ denote the attempt probability of a node in state $i$, and let $A$ denote the interference matrix among classes, specifically, $A_{c,d} = 1$ implies that a class $c$ node's transmission is interfered by a class $d$ node's transmission. Let $V_c = \{d : A_{c,d}=1\}$ denote the classes that interfere with class $c$ nodes' transmissions. Also, for each $i \in \{0, 1, \ldots, K\}$ and $c \in \{1,2,\ldots, C\}$, let $\xi^c_i$ denote the fraction of nodes (among the nodes in class $c$) in state $i$ and let $y \in \{0,1,2\}^C$ denote the state of the background process. The transition probability of tagged node in class $c$ from state $i$ to state $0$ is
\begin{align*}
& \left( \frac{p_i}{N}  \prod_{d \in V_c, d \neq c} \left[ 1_{\{y_d = 0\}} \prod_{j \in c, j \neq i} \biggr(1-\frac{p_j}{N}\biggr)\right] \right)\\
& \qquad \times \left( \prod_{d \in V_c, d \neq c} \left[ \prod_{d^\prime \in V_d}1_{\{y_{d^\prime} = 0\}} \left(\prod_{j \in d^\prime} (1-\frac{p_j}{N})\right) + \left( 1-\prod_{d^\prime \in V_d}1_{\{y_{d^\prime} = 0\}}\right) \right] \right);
\end{align*}
scaling the above by $N$ and noting that $\prod_{j \in d}(1-p_j/N) \sim \exp\{-\sum_{i=0}^K p_i \xi_i^d\}$, the corresponding transition rate of the continuous time model can be approximated as
\begin{align*}
p_i \left(\prod_{d  \in V_c} 1_{\{y_d = 0\}} \right)\times \left( \prod_{d \in V_c} \left[ \prod_{d^\prime \in V_d} 1_{\{y_{d^\prime}=0\}} \left(\exp\left\{-\sum_{i=0}^K p_i \xi^d_i\right\} -1\right) + 1 \right] \right).
\end{align*}
Similarly, the transition rate of a class $c$ node from state $i$ to state $i+1$ is
\begin{align*}
p_i \left(\prod_{d  \in V_c} 1_{\{y_d = 0\}} \right)\times \left( 1 - \prod_{d \in V_c} \left[ \prod_{d^\prime \in V_d} 1_{\{y_{d^\prime}=0\}} \left(\exp\left\{-\sum_{i=0}^K p_i \xi^d_i\right\} -1\right) + 1 \right] \right).
\end{align*}
We can also write down the transition rates of the background process; for example, a transition from the all-$0$ state to the state $y$ with $y_c = 1 $ and $y_d = 0$ for all $d \neq c$ (which happens when a node in class $c$ starts a transmission) occurs with rate
\begin{align*}
\left(N \sum_{i=0}^K p_i \xi^c_i \right) \times \exp\left\{-\sum_{i=0}^K p_i \xi^c_i\right\}.
\end{align*}
A study of the above model in the large-$N$ regime has been done by Bordenave et al.~\cite{bordenave-etal-12} towards understanding the average throughput obtained by a node in a given class, whereas our result in this paper provides a finer asymptotic analysis, in the realm of large deviations, which enables us to study metastability in such systems. 
For a continuous-time model of WLAN without a fast environment, see Boorstyn et al.~\cite{boorstyn-etal-87}.

\bibliographystyle{abbrv}
\bibliography{Report.bib}

\noindent \scriptsize{\textsc{S. Yasodharan and R. Sundaresan\newline
Department of Electrical Communication Engineering \newline
Indian Institute of Science\newline
Bangalore 560\,012, India\newline
email: sarath@iisc.ac.in, rajeshs@iisc.ac.in}}

\end{document}